\documentclass{amsart}
\usepackage[cp1251]{inputenc}
\usepackage[english]{babel}
\usepackage{amsmath}
\usepackage{amsfonts,amssymb}
\usepackage{mathrsfs}
\usepackage{bm}
\usepackage{esint}
\usepackage{amsthm}
\usepackage{tikz-cd}
\usepackage{a4wide}

\usepackage{hyperref}

\makeatother

\numberwithin{equation}{section}

\newtheorem{theorem}{Theorem}[section]
\newtheorem{thmdef}[theorem]{Theorem/Definition}

\newtheorem{corollary}[theorem]{Corollary}

\newtheorem{lemma}[theorem]{Lemma}

\newtheorem{proposition}[theorem]{Proposition}
\newtheorem{definition}[theorem]{Definition}

\newtheorem{example}[theorem]{Example}
\newtheorem{remark}[theorem]{Remark}

\newcommand{\N}{\mathbb{N}}

\newcommand{\R}{\mathbb{R}}
\newcommand{\Z}{\mathbb{Z}}
\newcommand{\ppi}{{\mbox{\boldmath$\pi$}}}

\newcommand{\restr}[1]{\lower3pt\hbox{$|_{#1}$}}
\newcommand{\e}{{\rm e}}

\newcommand{\X}{{\rm X}}
\newcommand{\Y}{{\rm Y}}
\renewcommand{\Z}{{\rm Z}}

\newcommand{\sfd}{{\sf d}}
\newcommand{\mm}{{\mathfrak m}}

\renewcommand{\L}{{\rm L}}

\newcommand{\LIP}{{\rm Lip}}

\newcommand{\CD}{{\sf CD}}
\newcommand{\RCD}{{\sf RCD}}

\renewcommand{\d}{{\rm d}}

\newcommand{\lip}{{\rm lip}}
\DeclareMathOperator{\aplip}{{\rm ap\,-lip}}
\newcommand{\Lip}{{\rm Lip}}

\newcommand{\weakto}{\rightharpoonup}
\newcommand{\nchi}{{\raise.3ex\hbox{$\chi$}}}
\newcommand{\supp}{{\rm supp}}
\newcommand{\lims}{\varlimsup}
\newcommand{\limi}{\varliminf}
\newcommand{\eps}{\varepsilon}

\DeclareMathOperator*{\aplims}{{\rm ap\,-}\lims}

\newcommand{\md}{{\sf md}}
\newcommand{\D}{{\sf D}}
\newcommand{\G}{{\sf G}}
\newcommand{\nn}{{\mathfrak n}}
\newcommand{\dou}{{\sf Doub}}

\newcommand{\E}{{\sf E}} % Funzionale per gradient flow
\newcommand{\fr}{\penalty-20\null\hfill$\blacksquare$}         %quadratino nero alla fine del remark

\DeclareMathOperator*{\esssup}{\rm ess-sup}

\newcommand{\ks}{{\sf KS}}
\newcommand{\hs}{{\sf HS}}
\newcommand{\kse}{{\sf ks}}

\newcommand{\la}{\langle}
\newcommand{\ra}{\rangle}
\newcommand{\CAT}{{\sf CAT}}
\newcommand{\sn}{{\sf sn}}
\newcommand{\ud}{\underline\d}
\newcommand{\vv}{{\sf v}}

\begin{document}

\title{Korevaar-Schoen's energy on  strongly rectifiable spaces}

\author{Nicola Gigli}
\address{SISSA, Via Bonomea 265, 34136 Trieste,  Italy}
\email{ngigli@sissa.it}

\author{Alexander Tyulenev}
\address{Steklov Mathematical Institute of Russian Academy of Sciences, 8 Gubkina St., Moscow 119991, Russia}
\email{tyulenev-math@yandex.ru}

\maketitle

\begin{abstract}
We extend Korevaar-Schoen's theory of metric valued Sobolev maps to cover the case of the source space being an $\RCD$ space. 

In this situation it appears that no version of the `subpartition lemma' holds: to obtain both existence of the limit of the approximated energies and the lower semicontinuity of the limit energy we shall rely on:
\begin{itemize}
\item[-] the fact that such spaces are `strongly rectifiable' a notion which is first-order in nature (as opposed to measure-contraction-like properties, which are of second order). This fact is particularly useful in combination with Kirchheim's metric differentiability theorem, as it allows to obtain an approximate metric differentiability result which in turn quickly provides a representation for the energy density,
\item[-] the differential calculus developed by the first author which allows, thanks to a representation formula for the energy that we prove here, to obtain the desired lower semicontinuity from the closure of the abstract differential.
\end{itemize}
When the target space is $\CAT(0)$ we can also identify the energy density as the Hilbert-Schmidt norm of the differential, in line with the smooth situation.
\end{abstract}

\tableofcontents

\section{Introduction}

A seminal paper in the study of the regularity of harmonic maps between Riemannian manifolds is the one \cite{ES64} by Eells and Sampson. A crucial result that they obtain is a local Lipschitz estimate in terms of a lower bound on the Ricci curvature of the source manifold and an upper bound on its dimension under the assumption that the target manifold has non-positive sectional curvature and is simply connected.

An interesting feature of their statement is that it does not rely on the smoothness of the manifolds but only on  the stated curvature bounds. It is therefore natural to wonder whether the same Lipschitz estimates can be obtained in the non-smooth setting under the appropriate weak curvature condition: we refer to \cite{GS92}, \cite{KS93}, \cite{Jost1997}, \cite{Jost98}, \cite{Gregori98}, \cite{Sturm97}, \cite{Sturm99}, \cite{KS03}, \cite{ZZ18}, \cite{Guo17}  for a non-exhaustive list of papers studying this issue at various levels of generality. In connection with the kind of program developed here, we mention also the recent \cite{Freidin19}, where the Bochner-Eells-Sampson formula has been established for harmonic maps from a smooth Riemannian manifold to a $\CAT(0)$ space.

One of the first tasks to accomplish in order to move from the smooth to the non-smooth setting is that of finding the appropriate replacement for the notion of energy that is minimized by harmonic maps. Recall that for smooth maps $u$ between smooth Riemannian manifolds such energy is given by the formula
\begin{equation}
\label{eq:defenintro}
\E(u):=\int |\d u|^2_{\sf HS}\,\d{\rm vol}.
\end{equation}
It is not clear a priori how to adapt this to the case where either the source or the target space are non-smooth (and actually even in this case some thoughts are required to handle the case of $u$ Sobolev): a turning point of the theory has been the paper \cite{KS93} by Korevaar and Schoen where the energy of maps from a smooth manifold $M$ to a general metric space $\Y$ has been defined. The idea - that here we briefly recall with some simplifications -  is that given such a map $u$ and a positive `scale' $r$ one defines first the `energy density at scale $r$' $\kse_{r}[u](x)$ of $u$ at $x$ by putting
\[
\kse_{r}[u](x):=\sqrt{\fint_{B_r(x)} \frac{\sfd_\Y(u(x),u(y))^2}{r^2}\,\d {\rm vol}(y)},
\]
then the total energy $\E_r(u)$ at scale $r$ as $\E_r(u):=\int \kse^2_{r}[u] \,\d{\rm vol} $ and finally the energy as
\begin{equation}
\label{eq:limenintro}
\E(u):=\lim_{r\downarrow0}\E_r(u).
\end{equation}
While this procedure indeed recovers the energy \eqref{eq:defenintro} in the case of smooth maps between smooth manifolds, it is non-trivial to check that such an energy is well defined in the generality of \cite{KS93}: to check that this is the case one should prove that the limit in \eqref{eq:limenintro} exists and that $\E$ is lower semicontinuous w.r.t.\ $L^2$-convergence of maps.

Both these fact are deduced in \cite{KS93} as a consequence of the so-called   \emph{subpartition lemma}, which roughly saying can be formulated as
\begin{equation}
\label{eq:suble}
\sqrt{\E_r(u)}\leq\sum_i\lambda_i\sqrt{\E_{\lambda_ir}(u)}+{\rm Err}(r,u)\qquad\text{ for }\lambda_i\geq 0,\ \sum_i\lambda_i=1,
\end{equation}
where the `error term' ${\rm Err}(r,u)$ goes to 0 as $r\downarrow0$. This inequality grants approximate monotonicity in $r$ of $\E_r(u)$, and in turn this implies  at once both the existence of the limit in \eqref{eq:limenintro}   and - since the energies at positive scale  are trivially $L^2$-continuous -  $L^2$-lower semicontinuity the limit energy $\E$.

In \cite{KS93}, the inequality \eqref{eq:suble} is obtained by relying, for the most part, on the fact that thanks to  the smoothness of the source space, the Ricci curvature is bounded from below. While this approach does not directly work in the non-smooth context, the basic idea  does and the argument can be stretched to cover much more general situation: this kind of work has been carried out in \cite{KS03}, where the notion of space with the strong measure contraction property of Bishop-Gromov type (SMCPBG-spaces, in short) has been introduced and it has been proved that a suitable version of \eqref{eq:suble} holds on SMCPBG-spaces. In particular, on these spaces the approximated energies converge and the limit energy is lower semicontinuous. Notable example of non-smooth SMCPBG-spaces     are finite dimensional Alexandrov spaces with curvature bounded from below. In this direction we remark that in the recent paper  \cite{ZZ18} it has been proved that $\CAT(0)$-valued harmonic maps on Alexandrov spaces are locally Lipschitz, thus greatly extending the original result by Eells-Sampson  \cite{ES64} and in particular providing the first extension of their Lipschitz estimates to the case where both the source  and target spaces are non-smooth. 

\bigskip

As said, the result by Eells-Sampson provides local Lipschitz estimates in terms of a lower Ricci and an upper dimension bound on the source space, therefore the natural non-smooth setting where to expect it to hold is that of maps on $\RCD(K,N)$ spaces (introduced in \cite{Gigli12} - see also \cite{AmbrosioGigliSavare11-2} for a previous contribution in the infinite dimensional case and \cite{Lott-Villani09}, \cite{Sturm06I,Sturm06II} for the original works on the $\CD$ condition via optimal transport) for $K\in\R$ and $N\in[1,\infty)$. Unfortunately, $\RCD$ spaces are not SMCPBG-spaces in general: informally speaking, this is due to the fact that the SMCPBG condition asks for the measure of balls to increase at a given power both  at small and at large scales. In the context of lower Ricci bounds, this kind of behaviour is related to a \emph{non-collapsing} condition (see \cite{Cheeger-Colding97I} for the original definition in the setting of Ricci-limit spaces and \cite{GDP17} for the adaptation in the synthetic setting). 

If one has the goal of generalizing Eells-Sampson's result to the $\RCD$ setting, it would be unnatural to impose such a non-collapsing assumption. Indeed, the typical example of \emph{collapsed} and smooth space is that of a weighted Riemannian manifold, i.e.\ of a smooth Riemannian manifold equipped with a measure different from the volume one. In this setting the relevant notion of Ricci curvature tensor is the so-called $N$-Bakry-\'Emery-Ricci tensor and it turns out that, by closely following Eells-Sampson's argument, Lipschitz estimates for harmonic maps can be obtained in  terms of  lower bounds on such tensor.  Alternatively, to see that it is not natural to impose a non-collapsing condition one can notice that Eells-Sampson's estimates pass to the limit under measured-Gromov-Hausdorff convergence even in the collapsing case.

\bigskip

Aim of this paper is to adapt  the constructions in \cite{KS93} to the case of $\RCD(K,N)$ spaces. Our main results can be described as follows: under suitable assumptions on the source space $\X$ (see \eqref{eq:assintro} below) which cover the $\RCD(K,N)$ case and for arbitrary complete spaces $\Y$ as target spaces we have:
\begin{itemize}
\item[i)] The energy $\E$ of a map is well defined by the formula \eqref{eq:limenintro}, i.e.\ the limit exists (Theorem \ref{thm:exlim}). Also, we improve, w.r.t.\ what previously known, the convergence results for the energy densities at positive scale to the limit energy density.
\item[ii)] The energy $\E$ is lower semicontinuous w.r.t.\ $L^2$ convergence of maps (Theorem \ref{thm:lscks}),
\item[iii)] A formula like \eqref{eq:defenintro} holds, i.e.\ the energy $\E$ can be written as
\begin{equation}
\label{eq:repintro}
\E(u)=\int S_2(\d u)\,\d\mm,
\end{equation}
where $S_2(\d u)$ is  a natural replacement of the squared Hilbert-Schmidt norm of the abstract differential of the given map (see Theorem \ref{thm:repr} and formula \eqref{eq:endensintro2} below for the simplified case $\X=\R^d$).  We remark that in the case $\X=\R^d$ this formula was already established in \cite{LS12} (and our proof read in the Euclidean case reduces to that of  \cite{LS12}).
\end{itemize}
Once this is done, following standard ideas in the field we can
\begin{itemize}
\item[iv)] Define the energy of a map from an open subset $\Omega$ of $\X$ and show that in this setting it is still possible to prescribe the value at the boundary (Definitions \ref{def:ksloc}, \ref{def:boundval}). In the case of $\CAT(0)$ spaces as target, we also show that the problem of minimizing the energy among maps with given boundary value has unique solution  (Theorem \ref{thm:unicat}).
\end{itemize}
We remark that since, as said, we cannot rely on the monotonicity granted by the subpartition lemma, we shall obtain  existence of the limit and  lower semicontinuity of the energy via two different means. 

Let us illustrate our strategy in the simplified case $\X=\R^d$. As starting point we recall the known fact that if $u:\R^d\to\Y$ is such that $\limi_{r\downarrow 0}\E_r(u)<\infty$ then, using only the fact that $\R^d$ is doubling and supports a local Poincar\'e inequality, for some $G\in L^2$ it holds
\begin{equation}
\label{eq:domintro}
\sfd_\Y(u(x),u(y))\leq |x-y|\big(G(x)+G(y)\big)\quad\forall x,y\in A
\end{equation}
for some Borel set $A\subset\R^d$ with negligible complement. In particular, this shows  that   $u$ has the Lusin-Lipschitz property. We couple this information with (a simplified version of) a result by Kirchheim \cite{Kir94} which says:  for $u:\R^d\to\Y$ Lipschitz we have that for $\mathcal L^d$-a.e.\ $x\in\R^d$ there exists a seminorm $\md_x(u)$, called metric differential of $u$ at $x$, such that 
\begin{equation}
\label{eq:mdintro}
\sfd_\Y(u(y),u(x))=\md_x(u)(y-x)+o(|y-x|).
\end{equation}
It is then possible to see (as done in \cite{Karma07}) that for maps having the Lusin-Lipschitz property, an appropriate approximate (in the measure theoretic sense) version of \eqref{eq:mdintro} holds and this fact coupled with the domination \eqref{eq:domintro} easily gives that 
\[
\kse^2_{r}[u](x)\quad\to\quad \fint_{B_1(0)}\md_x^2(u)(z)\,\d z\quad\text{as $r\downarrow0$ for a.e.\ $x$}
\]
and that the limit of $\E_r(u)$ as $r\downarrow0$ exists. The argument also gives the explicit expression  for the energy density 
\begin{equation}
\label{eq:endensintro}
\e_2^2[u](x)=\fint_{B_1(0)}\md_x^2(u)(v)\,\d v\qquad a.e.\ x
\end{equation}
and convergence in $L^2$ of $\kse_{r}[u]$ to it. This settles $(i)$. Then we turn to $(iii)$ and recall that the notion of differential for a Sobolev and metric-valued map has been defined in \cite{GPS18} by building up on the theory developed in \cite{Gigli14}. We won't enter into technicalities here and refer instead to Section \ref{se:remdif} for all the details; for the moment we just recall that in \cite{GPS18}  the following natural link between such abstract differential $\d u$ and the metric differential $\md_\cdot(u)$ has been established, at least for Lipschitz maps: for any $v\in \R^d$ it holds
\begin{equation}
\label{eq:linkintro}
|\d u(v)|(x)=\md_x(u)(v)\qquad a.e.\ x\in\R^d.
\end{equation}
Using the Lusin-Lipschitz property of Sobolev maps that we already mentioned it is not hard to extend this to the Sobolev case and thus to obtain from \eqref{eq:endensintro} the representation formula
\begin{equation}
\label{eq:endensintro2}
\e_2^2[u](x)=\fint_{B_1(0)}|\d u(v)|^2(x)\,\d v=:S_2(\d u)(x)\qquad a.e.\ x\in\R^d.
\end{equation}
This gives $(iii)$. The advantage of having formula \eqref{eq:endensintro2} at disposal in place of \eqref{eq:endensintro} is that we can rely on the closure properties of the differential to deduce the desired lower-semicontinuity. Specifically, one starts from the duality formula
\begin{equation}
\label{eq:dualintro}
|\d u(v)|=\esssup_{f:\Y\to\R\atop\Lip(f)\leq 1}\d(f\circ u)(v)\qquad\forall v\in\R^d
\end{equation}
and  uses the closure of the differential of scalar valued maps to deduce that: if   $u_n\to u$ in $L^2(\R^d,\Y)$ and $\sup_n\E(u_n)<\infty$ then 
\[
\d (f\circ u_n)(v)\quad\to \quad \d (f\circ u)(v)\qquad\text{ in the weak topology of $L^2$}
\]
 for any $f:\Y\to\R$ Lipschitz and $v\in\R^d$. From this and \eqref{eq:dualintro} it is not hard to check that under the same assumptions it holds
\[
|\d u(v)|\leq g\quad\mathcal L^d-a.e.\qquad \text{ for any weak $L^2$-limit $g$ of $(|\d u_n(v)|)$}
\]
which together with the representation formula \eqref{eq:endensintro2} easily gives the desired lower semicontinuity of the energy, thus obtaining  $(ii)$.
 
\bigskip

All this in the case $\X=\R^d$. The  observation that allows to extend the results to the non-smooth setting is that all the arguments that we used are first-order in nature. Thus for instance to obtain the same conclusions  in the case of $\X$ being a Riemannian manifold it is sufficient to notice that for every $\eps>0$ we have that $\X$ can be covered by open sets which are $(1+\eps)$-biLipschitz to open sets in $\R^d$. Then, roughly said, we can run the above arguments by locally replacing the metric in each of these open sets with the Euclidean one, thus committing errors of order $\eps$, and then let $\eps\downarrow0$.

A technically more involved - but conceptually similar - argument allows to extend the above line of thought  to metric measure spaces $(\X,\sfd,\mm)$ which are
\begin{equation}
\label{eq:assintro}
\text{(uniformly locally) doubling, support a Poincar\'e inequality and  strongly rectifiable,}
\end{equation}
the latter meaning that: there is $d\in\N$ such that $\mm$ is absolutely continuous w.r.t.\ the $d$-dimensional Hausdorff measure $\mathcal H^d$ and for every $\eps>0$ we can cover $\mm$-almost all $\X$ by Borel sets which are $(1+\eps)$-biLipschitz to Borel subsets of  $\R^d$ (see Section \ref{se:srs} for more on this). 
For the purpose of the original problem of studying harmonic maps from $\RCD(K,N)$ to $\CAT(0)$ spaces it is important to remark that  $\RCD(K,N)$ spaces are known to satisfy the assumptions \eqref{eq:assintro} as: they are uniformly locally doubling (\cite{Lott-Villani07}, \cite{Sturm06I,Sturm06II}), support a local Poincar\'e inequality (\cite{Lott-Villani07}, \cite{Rajala12}) and to  be strongly rectifiable (\cite{Mondino-Naber14}, \cite{MK16}, \cite{GP16-2}, \cite{BS18}), i.e.\ they satisfy the assumptions in \eqref{eq:assintro}.

As said, our set of assumptions is of first order in nature, but while they cover the case of the original paper \cite{KS93} and the one of Lipschitz manifolds studied in \cite{Gregori98},  they do not cover the one studied in \cite{KS03}, even if the hypotheses therein, being related to measure-contraction-like properties, are of second-order in spirit.

We also point out that there is nothing really special about the exponent $p=2$ here: everything can be generalized to arbitrary $p\in(1,\infty)$. Still, for simplicity for the most part of the manuscript we shall stick to the case $p=2$, see Remark \ref{re:depp} for more about this.

\bigskip

We conclude with a comment about the quantity $S_2(\d u)$ appearing in \eqref{eq:repintro}. The fact that in general something different from the Hilbert-Schmidt norm must appear is easily seen by considering the case of a smooth map from a smooth Riemannian manifold to a smooth Finsler manifold: in this case the differential of such map at any given point is a linear operator from a Hilbert to a Banach space and as such its Hilbert-Schmidt norm is not well defined. 

The quantity $S_2$, that we call 2-size, serves as replacement of the Hilbert-Schmidt norm. It can then be seen that whenever the target space has the appropriate kind of Hilbert-like behaviour on small scales - the relevant concept is that of `universally infinitesimally Hilbertian metric spaces' - then $S_2(\d u)$ coincides, up to a dimensional constant, with the squared Hilbert-Schmidt norm $|\d u|_{\hs}^2$ of $\d u$, as expected. 

For our case this is interesting because in \cite{DMGSP18} it has been proved that $\CAT(0)$ spaces (and more generally spaces that are locally $\CAT(\kappa)$) are universally infinitesimally Hilbertian (see Theorem \ref{thm:uih} for the rigorous meaning of this) and thus the energy of a Sobolev map $u$ from an $\RCD(K,N)$ space to a $\CAT(0)$ space can be written as
\[
\E(u)=c(d)\int|\d u|^2_\hs\,\d\mm,
\]
thus providing a close analogue of the defining formula \eqref{eq:defenintro},  where here $c(d)$ is a dimensional constant and $d$ the dimension of the source space when seen as a strongly rectifiable space. As mentioned, the above formula is valid for targets that are locally $\CAT(\kappa)$ spaces, but in Section \ref{se:cattarg} we shall stick to the case of $\CAT(0)$ targets because it is in this setting that we are able to prove existence and uniqueness of harmonic functions (Theorem \ref{thm:unicat}). K.T.\ Sturm pointed out to us that the same is expected to hold for target spaces that are $\CAT(1)$ and with diameter $<\pi$, but investigating in this direction is outside the scope of this paper.

\bigskip

Finally, we mention that building on top of the content of this paper, in \cite{GN20} it has been defined a suitable notion of `Laplacian' for maps from (open subsets of) $\RCD(K,N)$ to $\CAT(0)$ spaces.

\bigskip

{\bf Acknowledgement.} This research has been supported by the MIUR SIR-grant `Nonsmooth Differential Geometry' (RBSI147UG4). The work of A.\ I.\ Tyulenev was in part performed at the Steklov International Mathematical Center and supported by the Ministry of Science and Higher Education of the Russian Federation (agreement no. 075-15-2019-1614).

\section{Preliminaries}
\subsection{Doubling spaces, Poincar\'e inequalities and metric-valued $L^p$ spaces}
Throughout this paper by \emph{metric measure space} we will always mean a triple $(\X,\sfd,\mm)$ where $(\X,\sfd)$ is a complete and separable metric space and $\mm$ is  a non-negative and non-zero Borel measure giving finite mass to bounded sets. 

Given such a space and a pointed complete metric space $(\Y,\sfd_\Y,{\bar y})$ we denote by $L^0(\X,\Y)$ the collection of all equivalence classes up to $\mm$-a.e.\ equality of Borel maps from $\X$ to $\Y$ with separable range. Then for any $p\in(1,\infty)$  we put
\[
L^p(\X,\Y_{\bar y}):=\Big\{u\in L^0(\X,\Y)\ :\  \int\sfd_\Y^p(u(x),\bar y)\,\d\mm(x)<\infty\Big\}.
\]
Similar definitions can be given for maps defined only on some Borel subset $E$ of $\X$, leading to the spaces $L^0(E,\Y)$ and $L^p(E,\Y_{\bar y})$. If $\Y$ is a Banach space, we shall always pick ${\bar y}=0$ and avoid explicitly referring to such point. Notice that if $\mm(E)<\infty$ then the particular choice of $\bar y$ is irrelevant for the definition of  $L^p(\X,\Y_{\bar y})$.

It is easy to see that the distance 
\[
\sfd_{L^p}(u,v):=\bigg|\int\sfd_\Y^p\big(u(x),v(x)\big)\,\d\mm(x)\bigg|^{\frac1p}
\]
makes $L^p(E,\Y_{\bar y})$ a complete metric space. Notice that  if $\iota:\Y\to\Z$ is an isometric embedding, then $f\mapsto\iota\circ f$ is an isometric embedding of $L^p(\X,\Y_{\bar y})$ into $L^p(\X,\Z_{\iota(\bar y)})$. We shall occasionally use such embedding when it is convenient to deal with a Banach space target, a situation to which we can always reduce thanks to the Kuratowski's embedding that we now recall. Given  a set $\Y$, the Banach space $\ell^\infty(\Y)$ consists of all real valued bounded maps on $\Y$ endowed with the norm
\[
\|f\|_{\ell^\infty}:=\sup_{y\in\Y}|f(y)|.
\]
Then the following is well known and easy to prove:
\begin{lemma}[Kuratowski's embedding]\label{le:kur}
Let $(\Y,\sfd_\Y,{\bar y})$ be a pointed metric space. Then the map $\iota:\Y\to\ell^\infty(\Y)$ given by
\[
\iota_y(z):=\sfd_\Y(z,y)-\sfd_\Y(z,{\bar y})\qquad\forall z\in\Y
\]
is an isometry of $\Y$ with its image sending ${\bar y}$ to 0.
\end{lemma}
In what follows, given $E\subset\X$ we shall denote by $\nchi_E:\X\to\{0,1\}$ the function equal to 1 on $E$ and 0 outside.

A simple application of the above lemma is in the following density-like result:
\begin{lemma}[`Density' of continuous functions]\label{le:condenselp} Let $(\X,\sfd,\mm)$ be a metric measure space and $(\Y,\sfd_\Y,{\bar y})$ a pointed complete space. 

Then there exists another pointed complete space $(\Z,\sfd_\Z,{\bar z})$ and a pointed (i.e.\ sending ${\bar y}$ to ${\bar z}$) isometric immersion $\iota:\Y\to\Z$ such that the image of $L^p(\X,\Y_{\bar y})$ under the isometry $f\mapsto\iota\circ f$ is contained in the $L^p(\X,\Z_{\bar z})$-closure of $C_b(\X,\Z)\cap L^p(\X,\Z_{\bar z})$.
\end{lemma}
\begin{proof}
We pick $\Z:=\ell^\infty(\Y)$ and $\iota:\Y\to\Z$ the Kuratowski embedding. Clearly, it is sufficient to prove that $C_b(\X,\Z_0)\cap L^p(\X,\Z_0)$ is dense  in $L^p(\X,\Z_0)$. To this aim we notice that our definition of $L^p(\X,\Z_0)$ reduces to the case of the Lebesgue-Bochner space $L^p(\X,\Z)$ and in particular by well-known approximation procedures we know that the space of functions attaining only  a finite number of values is dense in $L^p(\X,\Z)$. 

By linearity, it is now sufficient to prove that any function of the form $\nchi_E v$ for $E\subset\X$ Borel and $v\in\Z$ belongs to the $L^p$-closure of $C_b(\X,\Z)\cap L^p(\X,\Z)$. To see this, just pick $(g_n)\subset C_b(\X,\R)\cap L^p(\X)$ be converging to $\nchi_E$ in $L^p(\X)$ and notice that $(g_n v)\subset C_b(\X,\Z)\cap L^p(\X,\Z)$ converges to $\nchi_Ev$ in $L^p(\X,\Z)$.
\end{proof}

\bigskip

The space $(\X,\sfd,\mm)$ is said to be \emph{uniformly locally  doubling} provided for any $R>0$ there is a constant $\dou(R)>0$ such that
\[
\mm(B_{2r}(x))\leq \dou(R)\,\mm(B_r(x))\qquad\forall x\in\X,\ r\in(0,R).
\]
On such spaces we shall occasionally consider the maximal function $M_R(f)$ of a function $f\in L^1_{loc}(\X)$ defined, for any given $R>0$, as 
\[
M_R(f)(x):=\sup_{r\in(0,R)}\fint_{B_r(x)}|f|\,\d\mm.
\]
It is well known that the doubling condition coupled with Vitali's covering lemma gives the following estimates:
\begin{proposition}\label{prop:maxest}
Let $(\X,\sfd,\mm)$ be a uniformly locally doubling space and $p\in(1,\infty)$. Then for every $R>0$ there is a constant $C(R,p)>0$ depending on $p$ and $\dou(R)$ only such that for any $f\in L^p(\X)$ it holds
\begin{equation}
\label{eq:maxest}
\|M_R(f)\|_{L^p(\mm)}\leq C(R,p)\|f\|_{L^p(\mm)}.
\end{equation}
\end{proposition}
A direct consequence of such estimates is the validity of the Lebesgue differentiation theorem. In particular, for any $E\subset\X$ Borel we have that $\mm$-a.e.\ point $x\in E$ is a density point for $E$, i.e.\ such that $\lim_{r\downarrow0}\frac{\mm(B_r(x)\cap E)}{\mm(B_r(x))}=1$.  Also, the set of density points of a Borel set is Borel itself.

We shall also use the fact that
\begin{equation}
\label{eq:denspor}
\text{if $x$ is a density point of $E$ and $x_n\to x$, $x_n\neq x$ then there is $(y_n)\subset E$ with }\frac{\sfd(x_n,y_n)}{\sfd(x_n,x)}\to0.
\end{equation}
Indeed, if not up to pass to a subsequence we could find $\alpha\in(0,1)$ such that $B_{\alpha \sfd(x,x_n)}(x_n)\cap E=\emptyset$. Then putting $r_n:=\sfd(x,x_n)$ the doubling condition grants the existence of $c>0$ such that
\begin{equation}
\label{eq:por}
\mm(B_{\alpha r_n}(x_n))\geq c\,\mm(B_{4r_n}(x_n))\geq c\,\mm(B_{2r_n}(x))
\end{equation}
for $n>>1$ and thus taking into account the inclusion  $B_{\alpha r_n}(x_n)\subset B_{2r_n}(x)$ we obtain
\[
\mm(B_{2r_n}(x)\cap E)\leq \mm(B_{2r_n}(x)\setminus B_{\alpha r_n}(x_n))=\mm(B_{2r_n}(x))-\mm( B_{\alpha r_n}(x_n))\stackrel{\eqref{eq:por}}\leq (1-c)\mm(B_{2r_n}(x)),
\]
which contradicts the fact that $x$ is a density point of $E$.

\bigskip

Another basic property of doubling spaces that we shall use is  the following simple and known result about partitions of unity (see also \cite[Section 4.1]{HKST15}):
\begin{lemma}\label{le:partun}
Let $(\X,\sfd,\mm)$ be a uniformly locally  doubling space. Then there exists a constant $C>0$ depending only on $\dou(1)$ such that for any $r\in(0,1/4)$ the following holds. 

There is an at most countable cover of $\X$ made of balls $B_i$ of radius $r$ such that each point $x\in\X$ belongs to at most $C$ balls, i.e.\ $\sum_i\nchi_{B_i}\leq C$. Moreover, there are functions $\varphi_i:\X\to[0,1]$ with $\supp(\varphi_i)\subset B_i$, $\sum_i\varphi_i=1$ and with $\Lip(\varphi_i)\leq \frac Cr$ for every $i\in\N$. The collection of these $\varphi_i$'s is called partition of  unity subordinate to $(B_i)$.
\end{lemma}
\begin{proof} Put for brevity ${\sf D}:=\dou(1)$.
Fix $r\in(0,1/8)$ and let $(x_n)\subset\X$ be countable and dense. Define an at most countable set  $(y_n)$ by putting $y_0:=x_0$ and then recursively putting $y_n:=x_{k}$ where $k$ is the least index $i\in\N$ such that $x_i\notin \cup_{j<n}B_r(y_j)$. If no such $k$ exists, we do not define $y_n$ (in other words, we built a maximal $r$-separated set). The definition and the density of $(x_n)$ ensure that the balls $B_i:=B_{2r}(y_i)$ cover $\X$. 

Now we claim that 
\begin{equation}
\label{eq:overl}
\text{For every $x\in\X$ the ball $B_{2r}(x)$ meets at most ${\sf D}^5$ balls $B_i$.}
\end{equation}
Indeed, if $B_{2r}(x)\cap B_i\neq\emptyset$ then $B_{2^5\cdot r/2}(y_i)=B_{16r}(y_i)\supset B_{8r}(x)$ and thus taking into account  the doubling condition we get 
\begin{equation}
\label{eq:percov}
\mm(B_{8r}(x))\leq{\sf D}^5\mm(B_{r/2}(y_i))\qquad\forall i\ s.t.\ B_r(x)\cap B_i\neq\emptyset .
\end{equation}
On the other hand, by construction the balls $B_{r/2}(y_i)$ are all disjoint (because  $\sfd(y_i,y_j)\geq r$ for any $i\neq j$) and if $B_{2r}(x)\cap B_i\neq\emptyset$ then  $B_i\subset B_{8r}(x)$. Thus if ${i_1},\ldots,{i_N}$ are such that  $B_{2r}(x)\cap B_{i_j}\neq\emptyset$ for every $j=1,\ldots,N$, we have
\[
N\sum_{j=1}^N\mm(B_{r/2}(y_{i_j}))\leq N\,\mm(B_{8r}(x))\stackrel{\eqref{eq:percov}}\leq{\sf D}^5 \sum_{j=1}^N\mm(B_{r/2}(y_{i_j})),
\]
which gives \eqref{eq:overl}.

Now let $\psi_i:=(\frac32r-\sfd(\cdot,y_i))^+$, where $(z)^+$ denotes the positive part of the real number $z$, and notice that $\supp(\psi_i)\subset B_i$ and that $\Lip(\psi_i)\leq 1$. Therefore by \eqref{eq:overl} we deduce $\Lip(\sum_j\psi_j\restr{B_i})\leq {\sf D}^5$ for every $i$.  Also, since by construction every $x\in \X$ is at distance $\leq r$ from some of the $y_i$'s, we have $\sum_j\psi_j\geq \frac r2$ on $\X$.  Hence from the trivial bound $\Lip(\frac1f)\leq \frac{\Lip(f)}{|\inf f|^2}$ we deduce $\Lip(\frac{1}{\sum_j\psi_j}\restr{B_i})\leq4\frac{{\sf D}^5}{r^2}$ for every $i$.

To conclude put $\varphi_i:=\frac{\psi_i}{\sum_j\psi_j}$. It is clear that $\supp(\varphi_i)\subset B_i$, $\varphi_i\geq 0$ and $\sum_i\varphi_i=1$. Thus we also have $\varphi_i\leq 1$ everywhere for any $i$. Let us now bound from above
\[
\Lip(\varphi_i)=\sup_{x,y\in\X}\frac{|\frac{\psi_i(y)}{\sum_j\psi_j(y)}-\frac{\psi_i(x)}{\sum_j\psi_j(x)}|}{\sfd(x,y)}.
\]
For $x,y\notin \supp(\psi_i)$ the expression at the right hand side is 0. For $x,y\in B_i$ we can use   the trivial bound $\Lip(fg)\leq \sup|g|\Lip(f)+\sup|f|\Lip(g)$ to obtain
\[
\begin{split}
\Lip(\varphi_i\restr{B_i})&\leq \sup|\psi_i|\Lip\Big(\frac{1}{\sum_j\psi_j}\restr{B_i}\Big)+\sup\Big|\frac1{\sum_j\psi_j}\Big|\Lip(\psi_i)\leq 2r\frac{4{\sf D}^5}{r^2}+\frac2r\cdot 1 =\frac{8{\sf D}^5+2}r.
\end{split}
\]
Finally, if $x\in \supp(\psi_i)$ and $y\notin B_i$ we have $\sfd(x,y)\geq\frac r2$ and since $|\psi_i|\leq \frac r2$ and  $\frac{1}{\sum_j\psi_j}\leq\frac2r$ we obtain
\[
\begin{split}
\frac{|\frac{\psi_i(x)}{\sum_j\psi_j(x)}|}{\sfd(x,y)}\leq \frac2r
\end{split}
\]
and the conclusion follows.
\end{proof}
Recall   that given a Borel function $u:\X\to\R\cup\{\pm\infty\}$ and $x\in\X$, the approximate $\lims$ of $u$ at $x$ is defined as 
\[
\aplims_{y\to x}u(y):=\inf\{\lambda\in\R\cup\{+\infty\}\ :\ x\text{ is a density point of }\{u\leq\lambda\}\}
\]
and it is easy to verify that
\begin{equation}
\label{eq:altroaplim}
\aplims_{y\to x}u(y)=\inf_U\lims_{y\to x\atop y\in U}u(y)=\inf_U\inf_{r>0}\sup_{B_r(x)\cap U}u,
\end{equation}
the $\inf$ being made among Borel sets $U$ for which $x$ is a density point.  Also, if $x$ is a density point of $U\subset \X$, then the value of $u$ outside the set $U$ is irrelevant for what concerns the value of $\aplims u(x)$, as seen by the very definition of this latter object. Therefore in this case it makes sense to define the quantity $\aplims u(x)$ as $\aplims v(x)$ for any Borel extension $v$ of $u$ to the whole $\X$: what just said ensures that the result does not depend on the chosen extension. From \eqref{eq:altroaplim} we also see that in this case it holds
\begin{equation}
\label{eq:altroaplim2}
\aplims_{y\to x}u(y)=\inf\lims_{y\to x\atop y\in V\cap U}u(y),
\end{equation}
the $\inf$ being made among Borel sets $V$ for which $x$ is a density point.

Using the notion of approximate $\lims$ we can introduce the one of \emph{approximate local Lipschitz constant} $\aplip(u):\X\to[0,\infty]$ of a Borel function $u:\X\to\Y$ as
\[
\aplip(u)(x):=\aplims_{y\to x}\frac{\sfd_\Y(u(y),u(x))}{\sfd(y,x)},\qquad\forall x\in\X.
\]
This notion should be compared with that of  \emph{local Lipschitz constant} $\lip(u):\X\to[0,\infty]$ defined as
\[
\lip (u)(x):=\lims_{y\to x}\frac{\sfd_\Y(u(y),u(x)))}{\sfd(x,y)}
\]
and with that of \emph{asymptotic Lipschitz constant}  $\lip_a(u):\X\to[0,\infty]$ defined as
\[
\lip_a (u)(x):=\lims_{y,z\to x}\frac{\sfd_\Y(u(y),u(z)))}{\sfd(y,z)}=\inf_{r>0}\Lip(u\restr{B_r(x)}).
\]
All these notions are intended to be 0 if $x$ is an isolated point. We shall denote by $\Lip(\X)$ (resp.\ $\Lip_{bs}(\X)$) the space of Lipschitz (resp.\ Lipschitz and with bounded support) real-valued functions on $\X$.

On uniformly locally  doubling spaces, for Lipschitz functions we have $\aplip(u)=\lip(u)$ as we are going to show now:
\begin{proposition}\label{prop:liplocap}
Let $(\X,\sfd,\mm)$ be a uniformly locally  doubling space, $(\Y,\sfd_\Y)$ a complete space, $U\subset \X$ Borel and $u:U\to\Y$ a Lipschitz map. Then for every $x\in U$ density point we have
\[
\lip(u)(x)=\aplip(u)(x).
\]
In particular,  if $u$ is defined on the whole $\X$ then such identity holds for every $x\in\X$.
\end{proposition}
\begin{proof} The inequality $\geq$ is obvious, so we turn to the other one. Fix $x\in\X$ and notice that if $\mm(x)>0$ then the doubling property (and the fact that $\mm$ gives finite mass to bounded sets) forces $\X=\{x\}$ and in this case the conclusion is obvious. Thus we may assume that $\mm(x)=0$, so that the fact that it is a density point of $U$ implies that it is not an isolated point of $U$. Then pick $V\subset \X$ Borel having $x$ as density point and let $(x_n)\subset U$ be an  arbitrary sequence converging to $x$, with $x_n\neq x$ for every $n\in\N$. Since, trivially, $x$ is a density point of $V\cap U$,  by \eqref{eq:denspor} there is a sequence $(y_n)\subset V\cap U$ such that $\frac{\sfd(x_n,y_n)}{\sfd(x_n,x)}\to 0$ and therefore
\[
\begin{split}
\lims_{n\to\infty}\frac{|u(x_n)-u(x)|}{\sfd(x_n,x)}&\leq \lims_{n\to\infty}\frac{|u(y_n)-u(x)|}{\sfd(x_n,x)}+\lims_{n\to\infty}\Lip(u) \frac{\sfd(x_n,y_n)}{\sfd(x_n,x)}\leq  \lims_{y\to x\atop y\in  V\cap U}\frac{|u(y)-u(x)|}{\sfd(y,x)},
\end{split}
\]
so that the claim follows from the arbitrariness of $(x_n)$, the definition of $\lip(u)$ and the characterization \eqref{eq:altroaplim2} of the approximate-$\lims$.
\end{proof}
We shall be mainly interested in approximate local Lipschitz constants for maps $u:\X\to\Y$ having the \emph{Lusin Lipschitz property}, i.e.\ such that  we can find Borel sets $N,U_n$, $n\in\N$, with $\X=N\cup(\cup_nU_n)$,  $N$ which is $\mm$-negligible and for which $u\restr{U_n}$ is Lipschitz for every $n\in\N$. 

Notice that a trivial consequence of the definition and of Proposition \ref{prop:liplocap} above is that
\begin{equation}
\label{eq:aplipfin}
\aplip (u)<\infty\quad\mm-a.e.\quad\text{ if $u$ has the Lusin Lipschitz property and $\X$ is unif.loc.doubling}.
\end{equation}

We conclude this section recalling that $(\X,\sfd,\mm)$ is said to support a  (weak, local, 1-1) \emph{Poincar\'e inequality} provided for any $R>0$ there are constants $C(R),\lambda (R)>0$ such that for any Lipschitz function $f:\X\to\R$ it holds
\begin{equation}
\label{eq:poincarelip}
\fint_{B_r(x)}|f-f_{B_r(x)}|\,\d\mm\leq C(R)r\fint_{ B_{\lambda r}(x)} \lip f\,\d\mm\qquad\forall x\in\X,\ r\in(0,R),
\end{equation}
where $f_B:=\fint_B f\,\d\mm$. Notice that in the literature this inequality is typically required to hold for continuous functions and upper gradients: our formulation is equivalent to that one, see \cite{AmbrosioGigliSavare11-3}. Also, for our results regarding the Korevaar-Schoen space $\ks^{1,p}(\X,\Y_{\bar y})$ it would be sufficient to require the (weaker) $1-p$ Poincar\'e inequality, see Remark \ref{re:poip} for further comments in this direction, but given that the main application that we have in mind is that of $\X$ being a $\RCD(K,N)$ space, where \eqref{eq:poincarelip} holds (see \cite{Lott-Villani07} and \cite{Rajala12}), for simplicity we preferred to deal just with it.

\subsection{Sobolev functions in the non-smooth setting} In this section we recall the concept of Sobolev function over a metric measure space with both real and metric target. For what concerns the real valued case, we shall mostly focus on the approach based on \emph{test plans} introduced in \cite{AmbrosioGigliSavare11}, but we recall (\cite{AmbrosioGigliSavare11,AmbrosioGigliSavare11-3}) that this is equivalent to the original definition given in \cite{Cheeger00} and thus also to the variant proposed in \cite{Shanmugalingam00}. Both for this and for more detailed references for the metric valued case we refer to \cite{HKST15},

Let us fix a metric measure space $(\X,\sfd,\mm)$ and $p,q\in(1,\infty)$ with $\frac1p+\frac1q=1$. We shall denote by $C([0,1],\X])$ the (complete and separable) space of continuous curves in $\X$ defined on $[0,1]$ equipped with the `$\sup$' distance. For $t\in[0,1]$ the evaluation map $\e_t:C([0,1],\X])\to\X$ sends $\gamma$ to $\gamma_t$. A curve $\gamma\in C([0,1],\X])$ is said to be absolutely continuous provided there is $f\in L^1(0,1)$ such that
\begin{equation}
\label{eq:defms}
\sfd(\gamma_t,\gamma_s)\leq \int_t^s f(r)\,\d r\qquad\forall t\leq s,\ t,s\in[0,1].
\end{equation}
The least - in the a.e.\ sense - such $f$ is called \emph{metric speed} of $\gamma$ and denoted $|\dot\gamma_t|$. In what follows, when writing $\int_0^1|\dot\gamma_t|\,\d t$ it will be intended that such integral is $+\infty$ by definition if $\gamma$ is not absolutely continuous.   We shall also define the \emph{metric speed} functional  ${\rm ms}: C([0,1],\X)\times[0,1]\to [0,+\infty]$ by putting
\[
{\rm ms}(\gamma,t):=\lim_{h\to 0}\frac{\sfd(\gamma_{t+h},\gamma_t)}{|h|}
\]
provided $\gamma$ is absolutely continuous and the limit exists, ${\rm ms}(\gamma,t):=\infty$ otherwise. It can be proved, see for instance \cite[Theorem 1.1.2]{AmbrosioGigliSavare08}, that for any absolutely continuous curve $\gamma$ it holds ${\rm ms}(\gamma,t)=|\dot\gamma_t|$ for a.e.\ $t$.

The notion of Sobolev function is given in duality with that of test plan:
\begin{definition}[$q$-test plan]
A $q$-test plan on $\X$ is a Borel probability measure $\ppi$ on $C([0,1],\X)$ such that
\[
\begin{split}
\int_0^1\int|\dot\gamma_t|^q\,\d\ppi(\gamma)\,\d t&<\infty,\\
(\e_t)_*\ppi&\leq C\mm\qquad\forall t\in[0,1],
\end{split}
\]
for some $C>0$. 
\end{definition}
\begin{definition}[Sobolev functions]\label{def:sobf}
Let $p\in(1,\infty)$ and $f\in L^0(\X)$. We say that $f$ belongs to the Sobolev class $S^p(\X)$ provided there is $G\in L^p(\X)$, $G\geq 0$ such that
\[
\int |f(\gamma_1)-f(\gamma_0)|\,\d\ppi(\gamma)\leq\iint_0^1G(\gamma_t)|\dot\gamma_t|\,\d t\,\d\ppi(\gamma)\qquad\forall \ppi\ \text{$q$-test plan}.
\]
Any such $G$ is called $p$-weak upper gradient of $f$.

We define $W^{1,p}(\X):=L^p(\X)\cap S^p(\X)$.
\end{definition}
It is possible to check that for $f\in S^p(\X)$ there is a minimal, in the $\mm$-a.e.\ sense,  $p$-weak upper gradient $G$: it will be denoted by $|Df|$. Notice that in principle $|Df|$ depends on $p$, but in what follows we shall omit to insist on such dependence, see also Remark \ref{re:depp}.

We shall equip $W^{1,p}(\X)$ with the norm
\[
\|f\|_{W^{1,p}}^p:=\|f\|^p_{L^p}+\||Df|\|^p_{L^p}
\]
and recall that $W^{1,p}(\X)$, which is easily seen to be a vector space, is a Banach space when equipped with this norm.

The following theorem collects some important properties of real-valued Sobolev functions on a metric measure space:
\begin{theorem}\label{thm:basesob}
Let $(\X,\sfd,\mm)$ be a metric measure space and $p\in(1,\infty)$. Then the following hold:
\begin{itemize}
\item[i)] Let $(f_n)\subset L^0(\X)$ and $(G_n)\subset  L^p(\X)$. Assume that $f_n\to f$ in $L^0(\X)$, that $G_n\weakto G$ in $L^p(\X)$ and that $G_n$ is a $p$-weak upper gradient of $f_n$ for every $n$. Then $G$ is a $p$-weak upper gradient of $f$.
\item[ii)] For any $f,g\in S^p(\X)$ we have
\[
|D f|=|Dg|\qquad\mm-a.e.\ on\  \{f=g\}.
\]
\item[iii)] Let $f\in W^{1,p}(\X)$. Then there is a sequence $(f_n)\subset \Lip_{bs}(\X)$ such that $(f_n),(\lip_a(f_n))$ converge to $f,|Df|$ in $L^p(\X)$ as $n\to\infty$.
\item[iv)] Let $f\in L^0(\X)$ and $G\in L^p(\X)$, $G\geq 0$. Then $G$ is a $p$-weak upper gradient for $f$ if and only if for any $q$-test plan $\ppi$ the following holds: for $\ppi$-a.e.\ $\gamma$ the function $f\circ\gamma$ belongs to $W^{1,1}(0,1)$ and
\[
|(f\circ \gamma)'|(t)\leq G(\gamma_t)\,|\dot\gamma_t|\qquad \ppi\times\mathcal L^1\restr{[0,1]}-a.e.\ (\gamma,t).
\]
\item[v)] Suppose that $(\X,\sfd,\mm)$ is uniformly locally  doubling. Then $W^{1,p}(\X)$ is reflexive.
\item[vi)] Suppose that $(\X,\sfd,\mm)$ supports a Poincar\'e inequality. Then for every $f\in W^{1,p}(\X)$   it holds
\begin{equation}
\label{eq:poincaresob}
\fint_{B_r(x)}|f-f_{B_r(x)}|\,\d\mm\leq C(R)r\fint_{ B_{\lambda r}(x)} |Df|\,\d\mm\qquad\forall x\in\X,\ r\in(0,R),
\end{equation}
where $C(R),\lambda$ are the same constants appearing in \eqref{eq:poincarelip}.
\end{itemize}
\end{theorem}
\begin{proof}
For $(i),(ii)$ see \cite{AmbrosioGigliSavare11}, for $(iii)$ see \cite{AmbrosioGigliSavare11} and \cite{AmbrosioGigliSavare11-3}. $(iv)$ is proved - by slightly modifying arguments in \cite{AmbrosioGigliSavare11} - in \cite[Theorem B.4]{Gigli12}. $(v)$ has been obtained in \cite{ACM14} under a global doubling assumption, but the argument works without modifications even under our assumption. Finally, $(vi)$ follows trivially from $(i)$ and the fact that for $f\in\Lip_{bs}(\X)$ the local Lipschitz constant $\lip(f)$ is a $p$-weak upper gradient (see also \cite{AmbrosioGigliSavare11-3}).
\end{proof}

The definition of Sobolev function can  be  adapted to the case of metric valued functions via a post-composition procedure (as proposed first by Ambrosio in \cite{Ambr90} for the case of BV functions and then by Reshetnyak in \cite{Resh97} for the Sobolev case - see \cite{HKST15} for more on the topic and detailed bibliography):
\begin{definition}\label{def:sobmetr}
Let  $(\X,\sfd,\mm)$ be a metric measure space, $(\Y,\sfd_\Y,{\bar y})$ a pointed  complete space, $p\in(1,\infty)$ and $u\in L^p(\X,\Y_{\bar y})$. We say that $u\in W^{1,p}(\X,\Y_{\bar y})$ provided there is $G\in L^p(X)$, $G\geq 0$ such that for every $\varphi:\Y\to\R$ 1-Lipschitz it holds $\varphi\circ u\in S^{p}(\X)$ with $|D (\varphi\circ u)|\leq G$ $\mm$-a.e..  Any such $G$ is called $p$-weak upper gradient of $u$.
\end{definition}
Fix $p\in(1,\infty)$. It is clear that the essential supremum of $|D(\varphi\circ u)|$ as $\varphi$ varies among 1-Lipschitz functions from $\Y$ to $\R$ is a  $p$-weak upper gradient of $u$ and that is the minimal one in the $\mm$-a.e. sense: such function is called minimal weak upper gradient of $u$ and denoted $|Du|$ (we will omit the dependence on $p$ of such object from our notation, as we shall only work with one fixed $p$ at the time).  We remark that in the smooth setting $|Du|$ would be the operator norm of the differential of $u$.

Some basic properties of metric-valued Sobolev functions are collected in the following proposition:
\begin{proposition}\label{prop:basesobmet}
Let $(\X,\sfd,\mm)$ be a metric measure space, $(\Y,\sfd_\Y,{\bar y})$ a pointed complete space and $p,q\in(1,\infty)$ with $\frac1p+\frac1q=1$. Then:
\begin{itemize}
\item[i)] Let $u\in L^p(\X,\Y_{\bar y})$ and $G\in L^p(\X)$, $G\geq 0$. Then the following are equivalent:
\begin{itemize}
\item[a)] $u\in W^{1,p}(\X,\Y_{\bar y})$ and $G$ is a $p$-weak upper gradient of $u$.
\item[b)] For every $q$-test plan $\ppi$ on $\X$ it holds
\[
\int\sfd_\Y(u(\gamma_1),u(\gamma_0))\,\d\ppi(\gamma)\leq\iint_0^1G(\gamma_t)|\dot\gamma_t|\,\d t\,\d\ppi(\gamma)
\]
\item[c)] For every $q$-test plan $\ppi$ on $\X$ the following holds. For $\ppi$-a.e.\ $\gamma$ the curve $[0,1]\ni t\mapsto u(\gamma_t)\in\Y$ has an absolutely continuous representative $u_\gamma$ and the bound\linebreak  ${\rm ms}(u_\gamma,t)\leq G(\gamma_t)|\dot\gamma_t|$ holds for $\ppi\times\mathcal L^1\restr{[0,1]}$-a.e.\ $(\gamma,t)$.
\end{itemize}
\item[ii)] Let $u\in L^p(\X,\Y_{\bar y})$, $(\Z,\sfd_\Z)$ a complete space and $\iota:\Y\to \Z$ be an isometric embedding. Then $u\in W^{1,p}(\X,\Y_{\bar y})$ if and only if $\iota\circ u\in W^{1,p}(\X,\Z_{\iota(\bar y)})$ and in this case $|D u|=|D(\iota\circ u)|$ $\mm$-a.e.. 
\item[iii)] Let $u_n\in L^p(\X,\Y_{\bar y})$ for every $n\in\N$ be such that $u_n\to u$ in $L^p(\X,\Y_{\bar y})$. Assume that $u_n\in W^{1,p}(\X,\Y_{\bar y})$ for every $n\in \N$ and that for some $G\in L^p(\X)$ we have $|Du_n|\weakto G$ in $L^p(\X)$. Then $u\in W^{1,p}(\X,\Y_{\bar y})$ as well with $|Du|\leq G$ $\mm$-a.e..
\item[iv)] For any $u,v\in W^{1,p}(\X,\Y_{\bar y})$ we have
\[
|Du|=|Dv|\qquad\mm-a.e.\ on\ \{u=v\}.
\]
\end{itemize}
\end{proposition}
\begin{proof} \ \\
\noindent{(i)} Up to modify $u$ in a negligible set we can, and will, assume that it takes values in a separable subset of $\Y$. Let $(y_n)\subset \Y$ be dense in such subset and put $\varphi_n:=\sfd_\Y(\cdot,y_n)$ for every $n\in\N$.

\noindent{$(b)\Rightarrow(a)$} If $\varphi:\Y\to\R$ is 1-Lipschitz we have 
\[
|\varphi(u(\gamma_1))-\varphi(u(\gamma_0))|\leq \sfd_Y(u(\gamma_1),u(\gamma_0))
\]
for any curve $\gamma$, thus the conclusion follows from our assumption by a direct verification of Definition \ref{def:sobf}

\noindent{$(a)\Rightarrow(c)$} Fix a $q$-test plan $\ppi$ and use $(iv)$ of Theorem \ref{thm:basesob} and the well-known existence of absolutely continuous representatives for real valued Sobolev functions to deduce: for $\ppi$-a.e.\ $\gamma$ there is a Borel negligible set $N_\gamma\subset[0,1]$ such that
\[
|\varphi_n(u(\gamma_s))-\varphi_n(u(\gamma_t))|\leq \int_t^s G(\gamma_r)|\dot\gamma_r|\,\d r\qquad\forall t,s\in[0,1]\setminus N_\gamma,\ \forall n\in\N.
\]
Taking the supremum in $n\in\N$ we deduce that
\[
\sfd_\Y\big(u(\gamma_s), u(\gamma_t)\big)\leq \int_t^s G(\gamma_r)|\dot\gamma_r|\,\d r\qquad\forall t,s\in[0,1]\setminus N_\gamma,
\]
which in particular grants that the restriction of $u\circ\gamma$ to $[0,1]\setminus N_\gamma$ is uniformly continuous. It is then clear that its continuous extension $u_\gamma$ is absolutely continuous and that  the desired bound on its metric speed comes from the characterization of the latter as least function $f$ for which \eqref{eq:defms} holds.

\noindent{$(c)\Rightarrow(b)$}  We know that for $\ppi$-a.e.\ $\gamma$ it holds
\[
\sfd_\Y(u(\gamma_s),u(\gamma_t))\leq \int_t^sG(\gamma_r)|\dot\gamma_r|\,\d r\qquad a.e.\ t,s\in[0,1],\ t<s
\]
and thus integrating w.r.t.\ $\ppi$ and using Fubini's theorem we deduce
\begin{equation}
\label{eq:econt}
\int \sfd_\Y(u(\gamma_s),u(\gamma_t))\,\d\ppi(\gamma)\leq \iint_t^sG(\gamma_r)|\dot\gamma_r|\,\d r\,\d\ppi(\gamma)\qquad a.e.\ t,s\in[0,1],\ t<s.
\end{equation}
Now, observe that the right hand side is continuous in $t,s$ and thus to conclude it is then sufficient to prove that the left hand side is also continuous in $t,s$. Use Lemma \ref{le:condenselp} to find $\iota:\Y\to\Z$ as in the statement and  recall the defining property of a test plan to obtain that for any $v_1,v_2\in L^p(\X,\Z_{\bar z})$ it holds
\begin{equation}
\label{eq:unifcont}
\begin{split}
\int| \sfd_\Z(v_1(\gamma_s),v_1(\gamma_t))- \sfd_\Z(v_2(\gamma_s),v_2(\gamma_t))|\,\d\ppi(\gamma)&\leq\int \sfd_\Z(v_1(\gamma_s),v_2(\gamma_s))+ \sfd_\Z(v_1(\gamma_t),v_2(\gamma_t))\,\d\ppi(\gamma)\\
&\leq 2C^{\frac1p}\sfd_{L^p}(v_1,v_2).
\end{split}
\end{equation}
Now fix $u\in L^p(\X,\Y_{\bar y})$ and   find  $(v_n)\subset C_b(\X,\Z)\cap L^p(\X,\Z_{\bar z})$ converging to $\iota\circ u$ in  $L^p(\X,\Z_{\bar z})$. Since $v_n\in C_b(\X,\Z_{\bar z})$ it is easy to check that  the quantity $\int  \sfd_\Z(v_n(\gamma_s),v_n(\gamma_t))\,\d\ppi(\gamma)$ is continuous in $t,s$. Then from \eqref{eq:unifcont} it follows that the left hand side of \eqref{eq:econt} is continuous in $t,s$, being the uniform limit of continuous functions.

\noindent{(ii)} Assume $u\in W^{1,p}(\X,\Y_{\bar y})$ and let $\varphi:\Z\to\R$ be 1-Lipschitz. Then $\varphi\circ\iota:\Y\to\R$ is 1-Lipschitz and thus our assumption and the defining property of $|Du|$ ensure that $\varphi\circ\iota\circ u\in W^{1,p}(\X)$ with $|D(\varphi\circ\iota\circ u)|\leq |Du|$. The arbitrariness of $\varphi$ then ensures that $\iota\circ u\in W^{1,p}(\X,\Z_{\iota(\bar y)})$ with $|D(\iota\circ u)|\leq |D u|$.

Now assume $\iota\circ u\in W^{1,p}(\X,\Z_{\iota(\bar y)})$ and let $\psi:\Y\to\R$ be 1-Lipschitz. Then there exists (e.g.\ as a consequence of McShane extension lemma) a 1-Lipschitz function $\varphi:\Z\to\R$ such that $\psi=\varphi\circ\iota$ on $\Y$.  Therefore our assumption grants that $\psi\circ u=\varphi\circ\iota\circ u$ belongs to $W^{1,p}(\X)$ with $|D(\psi\circ u)|=|D(\varphi\circ\iota\circ u)|\leq |D(\iota\circ u)|$ and the conclusion comes from the arbitrariness of $\psi$.

\noindent{(iii)} Let $\varphi:\Y\to\R$ be 1-Lipschitz and notice that $\varphi\circ u_n\to\varphi\circ u$ in $L^p(\X)$ and that since $|D(\varphi\circ u_n)|\leq |Du_n|$ the sequence $(|D(\varphi\circ u_n)|)$ is bounded in $L^p(\X)$. Letting $g$ be any  $L^p$-weak limit of a subsequence we clearly have $g\leq G$ and thus, by $(i)$ of Theorem \ref{thm:basesob}, we conclude  that $\varphi\circ u\in S^{p}(\X)$ with $|D(\varphi\circ u)|\leq G$.

The conclusion follows by the arbitrariness of $\varphi$.

\noindent{(iv)} Direct consequence of the analogous property in the real valued case.
\end{proof}
In some circumstances it is convenient to operate a cut-off procedure for metric-valued Sobolev maps. While this seems hard to do for arbitrary target spaces, at least in the case of Banach-valued maps the situation resembles that of real-valued function, thanks to the `differential' characterization of Sobolev maps given in point $(i-c)$ above:
\begin{lemma}\label{le:metcutoff}
Let $u\in W^{1,p}(\X,\Y)$ with $\Y$ being a Banach space. Let $\eta:\X\to\R$ be Lipschitz and with bounded support. Then $\eta u\in W^{1,p}(\X,\Y)$ with
\[
|D(\eta u)|\leq (\sup|\eta|)|Du|+\Lip(\eta)|u|\qquad\mm-a.e..
\]
\end{lemma}
\begin{proof} Let $\ppi$ be a $q$-test plan on $\X$. Then keep in mind point $(i-c)$ of Proposition  \ref{prop:basesobmet} and notice that for $\ppi$-a.e.\ $\gamma$ the curve $t\mapsto\eta(\gamma_t)u(\gamma_t)$ is a.e.\ equal to the absolutely continuous one $(\eta\circ\gamma )\,u_\gamma$ whose metric speed is - by direct computation - bounded by
\[
{\rm ms}((\eta\circ\gamma) \,u_\gamma,t)\leq |\dot\gamma_t|\big((\sup|\eta|)|D u|(\gamma_t)+\Lip(\eta)u(\gamma_t)\big)\quad a.e.\ t\in[0,1].
\]
The conclusion follows from  point $(i-c)$ of Proposition  \ref{prop:basesobmet}  again.
\end{proof}

\subsection{The Hajlasz-Sobolev space $\hs^{1,p}(\X,\Y_{\bar y})$ }
Here we recall Hajlasz's  definition (see \cite{Hajlasz2009}) of Sobolev functions in the non-smooth setting and its links with the  $W^{1,p}$ spaces that we have just seen.

\begin{definition}[Real-valued Hajlasz-Sobolev space]
Let $(\X,\sfd,\mm)$ be a metric measure space and $p\in(1,\infty)$. Then  $\hs^{1,p}(\X)$ is the subspace of $L^p(\X)$ made of functions $f$ such that there is $A\subset\X$ Borel of full measure and, for every $R>0$, a function $G_R\in L^p(\X)$ such that
\begin{equation}
\label{eq:defhs}
|f(y)-f(x)|\leq \sfd(x,y)\big(G_R(x)+G_R(y)\big)\qquad\forall x,y\in A \text{ such that }\sfd(x,y)\leq R.
\end{equation}
 Any such function $G_R$ is called $p$-Hajlasz gradient of $u$, or simply Hajlasz gradient if it is clear the Sobolev exponent we are working with, at scale $R$.
\end{definition}
\begin{remark}{\rm The standard definition of the Hajlasz-Sobolev space asks for a single function $G$ to satisfy \eqref{eq:defhs} for any $R>0$ (see \cite{Haj03}, \cite{HKST15}). The choice of our variant is motivated by the fact that we shall not work with (globally) doubling spaces but only with a (uniform) local doubling condition and in this case our phrasing is better linked to the notion of $W^{1,p}(\X)$ space, see Proposition \ref{prop:whs}.
}\fr\end{remark}

Much like the case of functions in $S^p(\X)$, a natural metric-valued variant of the notion of Hajlasz-Sobolev map can be obtained via post-composition with 1-Lipschitz maps:
\begin{definition}[Metric-valued Hajlasz-Sobolev space] Let  $(\X,\sfd,\mm)$ be a metric measure space, $(\Y,\sfd_\Y,{\bar y})$ a pointed complete space, $p\in(1,\infty)$ and $u\in L^p(\X,\Y_{\bar y})$. We say that  $u\in\hs^{1,p}(\X,\Y_{\bar y})$ if for any $R>0$ there is $G_R\in L^p(\X)$ which is an Hajlasz upper gradient at scale $R$ of $\varphi\circ u$ for any 1-Lipschitz function $\varphi:\Y\to\R$ such that $\varphi(\bar y)=0$. 
Any such $G_R$ is called Hajlasz gradient of $u$ at scale $R$.
\end{definition}
In other words, we require that for any 1-Lipschitz function $\varphi:\Y\to\R$ with $\varphi(\bar y)=0$ (but this requirement is in fact irrelevant) it holds
\[
|\varphi(u(y))-\varphi(u(x))|\leq  \sfd(x,y)\big(G_R(x)+G_R(y)\big)\qquad\forall x,y\in A_\varphi\text{ s.t. }\sfd(x,y)\leq R,
\]
where $A_\varphi\subset\X$ is some Borel set of full measure. Picking $\varphi_n:=\sfd_\Y(\cdot,y_n)-\sfd_\Y(\bar y,y_n)$ where $(y_n)\subset\Y$ is countable and dense in the essential image of $u$, and putting $A:=\cap_nA_{\varphi_n}$ we see that this is the same as requiring that
\[
\sfd_\Y(u(x),u(y))\leq \sfd(x,y)\big(G_R(x)+G_R(y)\big)\qquad\forall x,y\in A \text{ s.t. }\sfd(x,y)\leq R.
\]
From this bound it directly follows that the restriction of $u$ to $B_{R/2}(x)\cap A\cap \{G_R\leq c\}$ is Lipschitz for any $x\in\X$ and $c,R\geq 0$. In particular
\begin{equation}
\label{eq:hsluslip}
\text{every map in $\hs^{1,p}(\X,\Y_{\bar y})$ has the Lusin-Lipschitz property.}
\end{equation}
Let us now discuss the relation between $\hs^{1,p}(\X,\Y_{\bar y})$ and $W^{1,p}(\X,\Y_{\bar y})$. The inclusion $\hs^{1,p}(\X,\Y_{\bar y})\subset W^{1,p}(\X,\Y_{\bar y})$ holds without any assumption on the source space:
\begin{proposition}
Let $(\X,\sfd,\mm)$ be a metric measure space , $(\Y,\sfd_\Y,\bar y)$ a pointed complete space and $p\in(1,\infty)$. Then $\hs^{1,p}(\X,\Y_{\bar y})\subset W^{1,p}(\X,\Y_{\bar y})$ and for every $f\in \hs^{1,p}(\X,\Y_{\bar y})$, $R>0$ and Hajlasz-upper gradient $G_R$ at scale $R$ it holds
\[
|D f|\leq 2G_R\qquad\mm-a.e..
\]
\end{proposition}
\begin{proof} By the very definitions of $\hs^{1,p}(\X,\Y_{\bar y})$ and  $W^{1,p}(\X,\Y_{\bar y})$ it is sufficient to deal with the real-valued case.

 Fix $R>0$, let $q\in(1,\infty)$ be such that $\frac1p+\frac1q=1$, fix a $q$-test plan $\ppi$ and put $\hat\ppi:=\ppi\times\mathcal L^1\restr{[0,1]}$. For every $n\in\N$, $n> 0$,   let $\Gamma_n\subset C([0,1],\X)$ be the set of curves $\gamma$ such that $\sfd(\gamma_t,\gamma_{t+h})\leq R$ for every $t\in[0,1-\frac1n]$, $h\in[0,\frac1n]$, and observe that $\Gamma_n\subset\Gamma_{n+1}$ and $\cup_n\Gamma_n=C([0,1],\X)$.

Then define the `incremental ratio' functional ${\rm IR}_n:C([0,1],\X)\times[0,1]\to\R^+$ as
\[
{\rm IR}_n(\gamma,t):=\left\{\begin{array}{ll}
n\,\sfd(\gamma_\frac in,\gamma_{\frac{i+1}n}),&\qquad\text{ if }\gamma\in \Gamma_n\text{ and }t\in[\frac in,\frac{i+1}n)\text{ for some }i=0,\ldots,n-1,\\
0,&\qquad\text{ otherwise.}
\end{array}
\right.
\]
It is well known and easy to check (see for instance the arguments in \cite{Lisini07} and \cite{GT18}) that if $\gamma\in C([0,1],\X)$ is absolutely continuous with $|\dot\gamma_t|\in L^q(0,1)$, then ${\rm IR}_n(\gamma,\cdot)\to {\rm ms}(\gamma,\cdot) $ in $L^q(0,1)$ - we omit the proof of this fact. Also, from the trivial bound $\sfd(\gamma_t,\gamma_{t+\frac1n})\leq\int_t^{t+\frac1n}|\dot\gamma_s|\,\d s$  it follows that 
\[
\|{\rm IR}_n(\gamma,\cdot)\|_{L^q(0,1)}\leq \|{\rm ms}(\gamma,\cdot) \|_{L^q(0,1)}\qquad\forall \gamma\in C([0,1],\X),
\]
therefore using the fact that ${\rm ms}$ belongs to $L^q(\hat \ppi)$ (which follows from the fact that $\ppi$ is a $q$-test plan), by the dominated convergence theorem applied to the Lebesgue-Bochner space $L^q(C([0,1],\X),L^q([0,1];\mathcal L^1\restr{[0,1]});\ppi)\sim L^q(C([0,1],\X)\times[0,1];\hat\ppi)$ we conclude that 
\begin{equation}
\label{eq:convlq}
{\rm IR}_n\to {\rm ms}\quad{\rm in }\ L^q(C([0,1],\X)\times[0,1];\hat\ppi).
\end{equation}
Now consider $G\in L^p(\X)$ and for $n\in\N$, $n>0$, define $\hat G\in L^p(C([0,1],\X)\times[0,1];\hat\ppi)$ as
\[
\hat G(\gamma,t):=G(\gamma_t)\qquad\forall \gamma\in C([0,1],\X]),\ t\in[0,1],
\]
and then $\tilde G_n\in L^p(C([0,1],\X)\times[0,1];\hat\ppi)$ as
\[
\tilde G_n(\gamma,t)
:=\left\{\begin{array}{ll}
\hat G(\gamma, {\frac in})+\hat G(\gamma,{\frac{i+1}n}),&\qquad\text{ if }\gamma\in \Gamma_n\text{ and }t\in[\frac in,\frac{i+1}n]\text{ for some }i=0,\ldots,n-1,\\
0,&\qquad\text{ otherwise.}
\end{array}
\right.
\]
Notice that %the trivial bound $\|\hat G\|_{L^q(\ppi\mathcal L^1\restr{0,1})}\leq (\comp(\ppi))^{\frac1q}\|G\|_{L^q(\mm)}$ ensures that $\|\tilde  G_n\|_{L^q(\ppi\mathcal L^1\restr{0,1})}\leq 2(\comp(\ppi))^{\frac1q}\|G\|_{L^q(\mm)}$ for every $n>0$. Then 
the continuity of $t\mapsto\hat G(\cdot,t)\in L^p(\ppi)$ (which is well-known and easy to establish - see also the proof of the implication $(c)\Rightarrow(b)$ in Proposition \ref{prop:basesobmet}) gives
\begin{equation}
\label{eq:convlp}
\tilde G_n\to2\hat G\quad {\rm in }\ L^p(C([0,1],\X),L^p([0,1];\mathcal L^1\restr{[0,1]});\ppi)\sim L^p(C([0,1],\X)\times[0,1];\hat\ppi)
\end{equation}
as $n\to\infty$. To conclude the proof, let $f\in \hs^{1,p}(\X)$, $G\in L^p(\X)$ an Hajlasz-upper gradient at scale $R$ and $m\geq n>0$. Then a simple telescopic argument shows that
\[
\int_{\Gamma_n}|f(\gamma_1)-f(\gamma_0)|\,\d\ppi(\gamma)\leq \iint_0^1\tilde G_m\,{\rm IR}_m\,\d\hat \ppi
\]
and thus passing to the limit first as $m\to\infty$ recalling \eqref{eq:convlq} and \eqref{eq:convlp} and then as $n\to\infty$ we deduce that
\[
\int |f(\gamma_1)-f(\gamma_0)|\,\d\ppi(\gamma)\leq 2\iint_0^1\hat G\,{\rm ms} \,\d\hat \ppi,
\]
which, by the arbitrariness of the $q$-test plan $\ppi$, is the conclusion.
\end{proof}

We have already seen that  the inclusion $\hs^{1,p}(\X,\Y_{\bar y})\subset W^{1,p}(\X,\Y_{\bar y})$ always holds. The converse one is false in general, as shown in the following simple example:
\begin{example}{\rm Let $a_n\downarrow0$ be a sequence to be fixed later and $\X:=\{0\}\cup\{a_n:n\in\N\}$ equipped with the Euclidean distance and the measure  $\mm:=\delta_0+\sum_{n>0}2^{-n}\delta_{a_n}$.  From the fact that the space is totally disconnected it easily follows that any test plan is concentrated on constant curves, and thus that any $L^p(\X)$ function is actually in $W^{1,p}(\X)$ with null minimal upper gradient. Now consider the function $f:\X\to\R^+$ defined as $f(0):=0$ and $f(a_n):=n$ for every $n\in\N$. If $G$ is an Hajlasz gradient we must have $G(a_n)\geq \frac n{a_n}-G(0)$ and thus choosing the $a_n$'s small enough we see that $G$ does not belong to $L^p(\X)$. 
}\fr\end{example} 
Nevertheless, the isomorphism of $\hs^{1,p}(\X,\Y_{\bar y})$ and  $W^{1,p}(\X,\Y_{\bar y})$ as Banach spaces is true under assumptions that we shall often make in this manuscript (the proof we report is taken from \cite{HKST15}):
\begin{proposition}\label{prop:whs}
Let $(\X,\sfd,\mm)$ be locally uniformly doubling and supporting a Poincar\'e inequality, $(\Y,\sfd_\Y,{\bar y})$ a pointed complete space and  $p\in(1,\infty)$. Then $W^{1,p}(\X,\Y_{\bar y})\subset \hs^{1,p}(\X,\Y_{\bar y})$ and for every $u\in W^{1,p}(\X,\Y_{\bar y})$ a choice of Hajlasz-gradient for $u$ at scale $R$ is given by 
\[
G_R:=C(R)M_{2\lambda R}(|D u|),
\] 
where $\lambda$ is the constant appearing in the Poincar\'e inequality \eqref{eq:poincarelip} and the constant $C(R)>0$ depends only on the doubling and Poincar\'e constants of $\X$ and the chosen $R>0$.
\end{proposition}
Note that the fact that $G_R\in L^p(\X)$ follows from Proposition \ref{prop:maxest}
\begin{proof} By the very definitions of  $W^{1,p}(\X,\Y_{\bar y})$ and  $\hs^{1,p}(\X,\Y_{\bar y})$ it is sufficient to consider the real-valued case.
 
In the course of the proof we shall denote by $C(R)$ a positive constant depending on a parameter $R>0$ and the doubling and Poincar\'e constants only, whose value may change from line to line. 
%We start claiming that
%\begin{equation}
%\label{eq:claim1}
%|u(x)-u_{B_r(x)}|\leq Cr M_{|D u|}\qquad\mm-a.e.\ x\in\X,\ r\in(0,1).
%\end{equation}
Let  $x\in\X$ be such that  $u_{B_r(x)}\to u(x)$ as $r\downarrow0$ ($\mm$-a.e.\ $x\in \X$ has this property),   $R>0$, put $B_i:=B_{2^{-i}R}(x)$ and notice that
\begin{equation}
\label{eq:stimax}
\begin{split}
|u(x)-u_{B_R(x)}|&\leq\sum_{i=0}^\infty|u_{B_i}-u_{B_{i+1}}|\leq \sum_{i=0}^\infty\fint_{B_{i+1}}|u-u_{B_{i}}|\,\d\mm\\
&\leq C(R)\sum_{i=0}^\infty\fint_{B_{i}}|u-u_{B_{i}}|\,\d\mm\stackrel{\eqref{eq:poincaresob}}\leq C(R)\sum_{i=0}^\infty 2^{-i}R\fint_{\lambda B_{i}}|Du|\,\d\mm\\
&\leq C(R)RM_{\lambda R}(|Du|)(x).
\end{split}
\end{equation}
Now observe that if $R:=\sfd(x,y) $, from similar arguments and the inclusion $B_R(y)\subset B_{2R}(x)$ we get
\[
\begin{split}
|u_{B_R(x)}-u_{B_R(y)}|&\leq |u_{B_R(x)}-u_{B_{2R}(x)}|+|u_{B_{2R}(x)}-u_{B_R(y)}|\\
&\leq \fint_{B_R(x)}|u-u_{B_{2R}(x)}|\,\d\mm+\fint_{B_R(y)}|u-u_{B_{2R}(x)}|\,\d\mm \\
&\leq 2C(R)\fint_{B_{2R}(x)}|u-u_{B_{2R}(x)}|\,\d\mm\stackrel{\eqref{eq:poincaresob}}\leq 2C(R)\fint_{B_{2\lambda R}(x)}|Du|\,\d\mm\\
&\leq C(R)RM_{2\lambda R}(|Du|)(x).
\end{split}
\]
The conclusion comes combining  this bound and \eqref{eq:stimax} written for both $x$ and $y$.
\end{proof}

\subsection{Strongly rectifiable spaces}\label{se:srs}

Here we recall the notion of strongly rectifiable space by slightly modifying the original approach given in \cite{GP16}.

\begin{definition}[Strongly rectifiable spaces and aligned set of atlases]\label{def:srs} We say that a metric measure space $(\X,\sfd,\mm)$ is strongly rectifiable provided there is  $d\in\N$, called dimension of $\X$, such that  for every $\eps>0$ there exists a collection $\mathcal A^\eps:=\{(U^\eps_i,\varphi^\eps_i): i\in\N\}$, called $\eps$-atlas, such that:
\begin{itemize}
\item[i)] $U^\eps_i$ is a Borel subset of $\X$ for every $i$ and the $U^\eps_i$'s form a partition of $\X\setminus N$ for some $\mm$-negligible Borel set $N$,
\item[ii)] $\varphi^\eps_i$ is a $(1+\eps)$-biLipschitz map from $U^\eps_i$ to $\varphi^\eps_i(U^\eps_i)\subset \R^d$,
\item[iii)] it holds
\begin{equation}
\label{eq:measchart}
c_i\mathcal L^d\restr{\varphi^\eps_i(U^\eps_i)}\leq \mm_i\leq (1+\eps)c_i\mathcal L^d\restr{\varphi^\eps_i(U^\eps_i)}\quad \text{for some $c_i>0$, where\quad}\mm_i:=(\varphi^\eps_i)_*\mm\restr{U^\eps_i}.%\quad\text{and }c_i:=\essinf_{\varphi_i(U_i)}\frac{\d\mm_i}{\d\mathcal L^d}
\end{equation}
\end{itemize}
Given $\eps_n\downarrow0$, a family $\{\mathcal A^{\eps_n}\}_{n\in\N}$ of atlases is said \emph{aligned} provided:
\begin{itemize}
\item[iv)] for any $n,m\in\N $ and $(U^{\eps_n}_i,\varphi^{\eps_n}_i)\in\mathcal A^{\eps_n}$, $(U^{\eps_m}_j,\varphi^{\eps_m}_j)\in\mathcal A^{\eps_m}$ we have that
\begin{equation}
\label{eq:peralign}
\text{the map}\quad \varphi^{\eps_n}_i-\varphi^{\eps_m}_j:U^{\eps_n}_i\cap U^{\eps_m}_j\to \R^d\quad\text{ is $\eps_n+\eps_m$-Lipschitz}.
\end{equation}
\end{itemize}
\end{definition}
A relevant class of strongly rectifiable spaces is that of $\RCD(K,N)$ spaces:
\begin{theorem}
Let $K\in\R$, $N\in[1,\infty)$ and $(\X,\sfd,\mm)$ a $\RCD(K,N)$ space. Then  $(\X,\sfd,\mm)$  is strongly rectifiable.
\end{theorem}
\begin{proof}
The existence of $(1+\eps)$-biLipschitz charts has been proved in \cite{Mondino-Naber14}. The fact that the measure is absolutely continuous w.r.t.\ the Hausdorff measure of relevant dimension (which is easily seen to be equivalent to \eqref{eq:measchart} - see also the discussion below) has been obtained independently in \cite{MK16} and \cite{GP16-2}. Finally, the fact that the dimension of the target Euclidean space is independent on the particular chart is the main result of \cite{BS18}.
\end{proof}

Let us compare the definition just given with the one appeared in \cite{GP16}, there called strong $m$-rectifiability. A first difference is in the fact that here we imposed the space to have a given dimension $d$, while in \cite{GP16} the possibility of it being the union of parts with different dimensions was allowed. Strictly speaking, even for the theory developed in this manuscript we could deal with such more general situation, but that would only be an unnecessary complication. In fact, both here and in \cite{GP16} the main class of spaces we have in mind to work with is that of $\RCD(K,N)$ spaces and, as just recalled,  it is now known that they have constant dimension (a fact which was not clear at the time of \cite{GP16}).

Beside this, here  we require \eqref{eq:measchart} in place of the apparently weaker $(\varphi_i)_*\mm\restr{U_i}\ll\mathcal L^d$, but it is clear that up to refine the partition in such a way that the density $\frac{\d (\varphi_i)_*\mm\restr{U_i}}{\d\mathcal L^d}$ has small oscillations on $\varphi_i(U_i)$ and including the regions where such density is 0 in the negligible set $N$, the two approaches are equivalent.

Concerning the aligned family of atlases, in \cite{GP16} the condition \eqref{eq:peralign}  is replaced by
\begin{equation}
\label{eq:peralign2}
\|\d(\varphi^{\eps_n}_i\circ(\varphi^{\eps_m}_j)^{-1}-{\rm Id})(z)\|\leq \delta_{n,m}\quad\mathcal L^d-a.e.\ z\in \varphi_j(U^{\eps_n}_i\cap U^{\eps_m}_j)
\end{equation}
for every $ i,j\in\N$, where $\delta_{n,m}\to 0$ as $n,m\to\infty$. It is obvious that \eqref{eq:peralign} implies \eqref{eq:peralign2}. In fact, also the  converse implication holds, indeed, up to a relabeling of the atlases in the sequence and recalling the fact that $\varphi_j$ is $1+\eps_m$-biLipschitz, by refining the charts to conclude it is sufficient to show that
\[
\begin{split}
&\text{if $f:K\subset\R^d\to\R^d$ is a Lipschitz function with $\|\d f(z)\|< c$ for $\mathcal L^d$-a.e.\ $z\in K$, then for }\\
&\text{some sequence $(K_i)$ of Borel sets with $\mathcal L^d(K\setminus\cup_iK_i)=0$,  we have $\Lip(f\restr{K_i})\leq c$ $\forall i\in\N$.}
\end{split}
\]
To see this let 
\[
K_n:=\{z\in K\ :\ |f(w)-f(z)|<c|w-z|\quad\forall w\in B_{1/n}(z)\}
\]
and notice that these are Borel sets and that our assumption gives  $\mathcal L^d(K\setminus \cup_nK_n)=0$. To conclude just write $K_n=\cup_mK_{n,m}$ with the $K_{n,m}$'s Borel and with diameter $\leq \frac1n$ and notice that by construction $\Lip(f\restr{K_{n,m}})\leq c$ for every $n,m\in\N$.

\bigskip

In particular, by \cite[Theorem 3.9]{GP16} we know that on a strongly rectifiable space, for any sequence $\eps_n\downarrow0$  an aligned family of atlases $(\mathcal A^{\eps_n})$ exists. We shall often use this fact in what comes next without further notification.

\bigskip

From the assumptions on the strongly rectifiable space $\X$ it follows that $\mm\ll \mathcal H^d$, where  $\mathcal H^d$ is the $d$-dimensional Hausdorff measure. The Radon-Nikodym density can be computed via differentiation as discussed in the following well-known result (see e.g.\ \cite[Theorem 2.13]{GDP17} for the proof):
\begin{theorem}[Density w.r.t.\ the Hausdorff measure]\label{thm:denshaus}
Let $(\X,\sfd,\mm)$ be a strongly rectifiable space of dimension $d\in\N$. Then the function  $\vartheta^d:\X\to[0,\infty]$ defined by
\[
\vartheta^d(x)=\lim_{r\downarrow0}\frac{\mm(B_r(x))}{\omega_dr^d}\quad\text{provided the limit exists and $0$ otherwise,}
\]
is a Borel representative of the Radon-Nikodym density $\frac{\d\mm}{\d\mathcal H^d}$.
\end{theorem}
Finally, we recall that when the space under consideration is the Euclidean one $\R^d$, it is well known that the Hausdorff measure $\mathcal H^d$ coincides with the Lebesgue measure $\mathcal L^d$; to emphasize the fact that we are working on $\R^d$ we shall speak about the Lebesgue measure $\mathcal L^d$ in this context.

\section{The Korevaar-Schoen space}

\subsection{Approximate metric differentiability on strongly rectifiable spaces} We shall denote by $\sn^d$ the set of seminorms on $\R^d$ and equip it  with the complete and separable distance $\D $ defined by
\[
\D(\nn_1,\nn_2):=\sup_{z:|z|\leq 1}|\nn_1(z)-\nn_2(z)|=\Lip(\nn_1-\nn_2)=\lip(\nn_1-\nn_2)(0),
\]
where here and below by $|\cdot|$ we intend the classical Euclidean norm. We shall also put 
\[
|||\nn|||:=\D(\nn,0)=\sup_{z:|z|\leq 1}|\nn(z)|=\Lip(\nn)=\lip(\nn)(0).
\]

We start recalling the following result, due to Kirchheim \cite{Kir94}:
\begin{theorem}[Kirchheim's metric differential]\label{thm:kir}
Let $(\Y,\sfd_\Y)$ be a complete space $U\subset\R^d$  Borel and $u:U\to\Y$ a Lipschitz map. Then for $\mathcal L^d$-a.e.\ $x\in U$ there is a seminorm $\md_x(u)$ on $\R^d$, called metric differential of $u$ at $x$, such that
\begin{equation}
\label{eq:md}
\lims_{y\to x\atop y\in U}\frac{|\sfd_\Y(u(y),u(x))-\md_x(u)(y-x)|}{|y-x|}=0.
\end{equation}
Moreover, the $\mathcal L^d$-a.e.\ defined map $x\mapsto \md_x(u)\in\sn^d$ is Borel.
\end{theorem}
\begin{proof}
In \cite{Kir94} the existence of the metric differential is given for functions defined in the whole $\R^d$. This variant is easily obtainable by considering a Kuratowski embedding $\iota:\Y\to\ell^\infty(\Y)$ (Lemma \ref{le:kur}), a Lipschitz extension $v:\R^d\to\ell^\infty(\Y)$ of $\iota\circ u$ (Lemma \ref{le:ext} below), applying the original statement to such function $v$ and noticing that since $\iota$ is an isometry, the metric differential of $v$ at $x$ coincides with that of $u$ at $x$ for any $x\in U$.

This argument also shows that to prove the stated Borel regularity it is sufficient to deal with the case of maps $u$ defined on $\R^d$. Also, from the identity
\[
\D(\nn_1,\nn_2)=\sup_n\nn_1(v_n)-\nn_2(v_n)\qquad\text{where $(v_n)\subset B_1(0)$ is dense}
\]
we see that it is sufficient to prove that for any $v\in\R^d$ the $\mathcal L^d$-a.e.\ defined map $x\mapsto \md_x(u)(v)$ is Borel. Then from the identity
\[
\md_x(u)(v)=\lim_{n\to\infty}n\,\md_x(u)((x+n^{-1}v)-x)\stackrel{\eqref{eq:md}}=\lim_{n\to\infty}n\,{\sfd_\Y(u(x+n^{-1}v),u(x))}
\]
valid for $\mathcal L^d$-a.e.\ $x$ the conclusion easily follows.
\end{proof}
The well known McShane extension lemma can easily be adapted to the case of $\ell^\infty$-valued maps:
\begin{lemma}[Lipschitz extension]\label{le:ext}
Let $(\X,\sfd)$ be a metric space, $\Y$ a set, $U\subset \X$ and $f:U\to \ell^\infty(\Y)$ a Lipschitz function. Then there exists an extension $F$ of $f$ to the whole $\X$ with the same Lipschitz constant, i.e.\ a map $F:\X\to  \ell^\infty(\Y)$ such that $F\restr U=f$ and $\Lip(F)=\Lip(f)$.
\end{lemma}
\begin{proof}
For every $y\in\Y$ define 
\[
F_y(x):=\inf_{z\in U}f_y(z)+\sfd(x,z)\Lip(f)\qquad\forall x\in\X.
\]
It is readily verified that this function has the required properties.
\end{proof}

\bigskip

The main goal of this section is to extend Kirchheim's result to maps defined on strongly rectifiable spaces. The basic idea is simple: we use the charts to reduce the differentiability problem to a problem on $\R^d$ for which we can apply the known result; if the charts are $(1+\eps)$-biLipschitz, in doing so we will make an error of order $\eps$ and if we consider a different atlas we shall obtain a metric differential close to the previous one provided the charts are somehow aligned (the relevant notion being the one introduced in Definition \ref{def:srs}).  Then the conclusion follows by considering an aligned family of atlases $(\mathcal A^{\eps_n})$ and passing to the limit as $n\to\infty$.

\bigskip

We now turn to the rigorous definition of `approximate metric differentiability' on strongly rectifiable space. It  is based on, for given $(U_i,\varphi_i)$ belonging to some atlas, thinking at   the map $y\mapsto\varphi_i(y)-\varphi_i(x)\in\R^d$ as a sort of `$\eps$-inverse of the exponential map at $x\in U_i$' (see also \cite[Theorem 6.6]{GP16} for more about the interpretation of $\R^d$ as the tangent space at a given point of a strongly rectifiable space).
\begin{definition}[Approximate metric differentiability]\label{def:apmd}
Let $(\X,\sfd,\mm)$ be a strongly rectifiable  space, $\eps_n\downarrow0$ and $\{\mathcal A^{\eps_n}\}$ an aligned family of atlases. Also, let $(\Y,\sfd_\Y)$ be a metric space and $u:\X\to\Y$. We say that $u$ is approximately metrically differentiable at $x\in\X$ relatively to $(\mathcal A^{\eps_n})$ provided:
\begin{itemize}
\item[i)] For every $n\in\N$ there is $i(x,n)\in\N$ such that $x$ belongs to $U_{i(x,n)}^{\eps_n}$, is a density point of such set and $\varphi^{\eps_n}_{i(x,n)}(x)$ is a density point of $\varphi^{\eps_n}_{i(x,n)}(U^{\eps_n}_{i(x,n)})$.
\item[ii)] There is a seminorm $\md_x(u)$ on $\R^d$, called metric differential of $u$ at $x$, such that
\[
\lims_{n\to\infty}\aplims_{y\to x\atop y\in U_{i}^{n}}\frac{\big|\sfd_\Y(u(y),u(x))-\md_x(u)(\varphi^{n}_{i}(y)-\varphi^{n}_{i}(x))\big|}{\sfd(y,x)}=0,
\]
where for brevity we wrote $U_i^{n},\varphi_i^{n}$ in place of  $U_{i(x,n)}^{\eps_n},\varphi_{i(x,n)}^{\eps_n}$.
\end{itemize}
\end{definition}
For smooth maps on $\R^d$ it is trivial to check that the norm of the differential coincides with the local Lipschitz constant. A similar link exists  between approximate metric differential and approximate local Lipschitz constant:
\begin{lemma}\label{le:normmd} Let $(\X,\sfd,\mm)$ be a strongly rectifiable space,  $\eps_n\downarrow0$ and $(\mathcal A^{\eps_n})$ an aligned family of atlases. Also, let $(\Y,\sfd_{\Y})$ be a complete space and $u:\X\to\Y$ a map which is approximately metrically differentiable at $x\in\X$ relatively to $(\mathcal A^{\eps_n})$. Then
\begin{equation}
\label{eq:normmd}
\aplip u(x)=|||\md_x(u)|||.
\end{equation}
\end{lemma}
\begin{proof} Fix $n\in\N$, let $i\in\N$ be such that $x\in U^n_i$ and notice that since $\varphi^n_i:U^n_i\to\varphi^n_i(U^n_i)$ is $(1+\eps_n)$-Lipschitz we have
\begin{equation}
\label{eq:s1}
\frac{\md_x(u)(\varphi^n_i(y)-\varphi^n_i(x))}{\sfd(x,y)}\leq(1+\eps_n) \frac{\md_x(u)(\varphi^n_i(y)-\varphi^n_i(x))}{|\varphi^n_i(y)-\varphi^n_i(x)|}.
\end{equation}
Now recall that $x,\varphi^n_i(x)$ are density points of $U^n_i,\varphi^n_i(U^n_i)$ respectively and notice that the properties of $\varphi^n_i$ ensure that 
\begin{equation}
\label{eq:claimperchangevar}
\text{$V\subset U^n_i$ has $x$ as density point if and only if $\varphi^n_i(V)$ has $\varphi^n_i(x)$ as density point.}
\end{equation} 
To see this notice that
\begin{equation}
\label{eq:stimadens}
\frac{\mathcal L^d\big(B^{\R^d}_r(\varphi^n_i(x))\setminus \varphi^n_i(V)\big)}{r^d}=\frac{\mathcal L^d\big(B^{\R^d}_r(\varphi^n_i(x))\setminus \varphi^n_i(U^n_i)\big)}{r^d}+\frac{\mathcal L^d\restr{\varphi^n_i(U^n_i)}\big(B^{\R^d}_r(\varphi^n_i(x))\setminus \varphi^n_i(V)\big)}{r^d}
\end{equation}
and the first addend on the right goes to 0 because $\varphi^n_i(x)$ is a density point of $\varphi^n_i(U^n_i)$. Recalling that $\varphi^n_i$ is $(1+\eps_n)$-Lipschitz and the bound \eqref{eq:measchart} we can estimate from above the second one with
\[
\frac{\mathcal L^d\restr{\varphi^n_i(U^n_i)}\big(B^{\R^d}_r(\varphi^n_i(x)) \setminus \varphi^n_i(V)\big)}{r^d}\leq \frac{\mm(B^\X_{r(1+\eps_n)}(x)\setminus V)}{c_ir^d}.
\]
Hence if $x$ is a density point of $V$ we have that the right hand side in the above goes to 0 as $r\downarrow0 $ and thus \eqref{eq:stimadens} shows that $\varphi^n_i(x)$ is a density point of $\varphi^n_i(V^n_i)$. The opposite implication is proven analogously. From \eqref{eq:claimperchangevar} and \eqref{eq:altroaplim} we deduce the `change of variable formula'
\begin{equation}
\label{eq:changevar}
\aplims_{y\to x\atop y\in U^n_i}\frac{\md_x(u)(\varphi^n_i(y)-\varphi^n_i(x))}{|\varphi^n_i(y)-\varphi^n_i(x)|}=\aplims_{w\to \varphi^n_i(x)\atop w\in \varphi^n_i(U^n_i)}\frac{\md_x(u)(w-\varphi^n_i(x))}{|w-\varphi^n_i(x)|}
\end{equation}
which together with \eqref{eq:s1} and Proposition \ref{prop:liplocap} gives
\[
\aplims_{y\to x\atop y\in U^n_i}\frac{\md_x(u)(\varphi^n_i(y)-\varphi^n_i(x))}{\sfd(x,y)}\leq(1+\eps_n)\lip(\md_x(u))(0)=(1+\eps_n)|||\md_x(u)|||.
\]
Using the fact that $(\varphi^n_i)^{-1}$ is also $(1+\eps_n)$-Lipschitz and similar arguments we obtain the lower bound 
\[
\aplims_{y\to x\atop y\in U^n_i}\frac{\md_x(u)(\varphi^n_i(y)-\varphi^n_i(x))}{\sfd(x,y)}\geq (1+\eps_n)^{-1}|||\md_x(u)|||,
\]
so that the conclusion  follows from the very definition of metric differential.
\end{proof}

\begin{remark}{\rm
The conclusion of the above lemma fails if one does not insist on $\varphi^n_{i(x,n)}(x)$ to be a density point of $\varphi^n_{i(x,n)}(U^n_{i(x,n)})$ in the definition of metric differentiability. This can be easily seen by considering $\X:=[0,1]^2$ (equipped with the restriction of the Euclidean distance and Lebesgue measure), $U^n_0=\X$ and $\varphi^n_0$ to be the natural embedding in $\R^2$. Then $x:=(0,0)\in\X$ is a density point of $U^n_0$ and the function $\X\ni (x_1,x_2)\mapsto u(x_1,x_2):=x_1-x_2$ is metrically differentiable at $x$, with metric differential given by $\md_x(u)(v_1,v_2)=|v_1-v_2|$, so that $|||\md_x(u)|||=\sup_{(v_1,v_2)\in\R^2}\frac{|v_1-v_2|}{\sqrt{|v_1|^2+|v_2|^2}}=\sqrt 2$. On the other hand we have $\aplip u(x)\leq\lip u(x)$ (in fact equality occurs by Proposition \ref{prop:liplocap}) and we have
\[
\lip\, u(x)=\lims_{(y_1,y_2)\in\X\atop (y_1,y_2)\to x}\frac{|u(y_1,y_2)-u(0,0)|}{\sfd_\X((y_1,y_2),(0,0))}=\lims_{y_1,y_2\geq 0\atop y_1,y_2\downarrow0 }\frac{|y_1-y_2|}{\sqrt{|y_1|^2+|y_2|^2}}=1,
\]
thus showing that the strict inequality $<$ holds in  \eqref{eq:normmd}.
}\fr\end{remark}

We now turn to the main result of this section:
\begin{proposition}\label{prop:metrdiff}
Let $(\X,\sfd,\mm)$ be a  uniformly locally  doubling and strongly rectifiable  space, $(\Y,\sfd_\Y)$ a metric space and $u:\X\to\Y$ a Borel function with the Lusin-Lipschitz property. Also, let $\eps_n\downarrow0$ and $(\mathcal A^{\eps_n})$ an aligned family of atlases.

Then $u$ is approximately metrically differentiable at $\mm$-a.e.\ $x\in\X$, relatively to $(\mathcal A^{\eps_n})$, and the $\mm$-a.e.\ defined map $x\mapsto\md_x(u)\in\sn^d$ is Borel.

More precisely:
\begin{itemize}
\item[i)] for every $n\in\N$ the map $\X\ni x\mapsto\nn^n_x\in\sn^d$ is a $\mm$-a.e.\ well defined Borel map by the formula
\[
\nn^n_x:=\md_{\varphi^n_i(x)}(u\circ(\varphi^n_i)^{-1})\qquad\mm-a.e.\ on\ U^n_i,
\]
\item[ii)] for $\mm$-a.e.\ $x\in\X$ the sequence $(\nn^n_x)\subset \sn^d$ admits a limit $\nn_x$,
\item[ii)] $\nn_x$ is the metric differential of $u$ at $x$ for $\mm$-a.e.\ $x\in\X$.
\end{itemize}
\end{proposition}
\begin{proof} Notice that up to a  refining the charts,  we can assume that $u\restr{U}$ is Lipschitz for every $U$ chart of some of the given atlases. Also, in the course of the proof we shall frequently use the following observation: if $U\subset\R^d$ is Borel and $v:U\to \Y$ is metrically differentiable at a density point  $x\in U$ in the sense of Theorem \ref{thm:kir}, then it is also approximately metrically differentiable, and with the same metric differential, in the sense of Definition \ref{def:apmd}, where here we pick  $\X:=\bar U$ equipped with the Euclidean distance and $\mm:=\mathcal L^d\restr U$, where the charts are given by the inclusion  $\X\hookrightarrow\R^d$.

\noindent{$\mathbf{(i)}$} Fix $n\in\N$ and define $\mm$-a.e.\ the map $\nn^n:\X\to\sn^d$ as follows: for every $(U^n_i,\varphi^n_i)\in\mathcal A^{\eps_n}$ consider the Lipschitz map $v^n_i:\varphi^n_i(U^n_i)\to\Y$ given by $v^n_i:=u\circ(\varphi^n_i)^{-1}$ and use Kirchheim's theorem \ref{thm:kir}  to obtain that $\nn^n_x:=\md_{\varphi^n_i(x)}(v^n_i)$ is well-defined $\mm$-a.e.\ and Borel on $U^n_i$. By the arbitrariness of $i\in\N$ this defines $\nn^n_x$ for $\mm$-a.e.\ $x$. Now we apply \eqref{eq:normmd} to the function $v^n_i$ defined on $\R^d$ (which is a strongly rectifiable space with the identity as chart - here we extend $v^n_i$ to be 0 outside $\varphi^n_i(U^n_i)$) to deduce that  for $\mm$-a.e.\ $x\in U^n_i$ we have $|||\nn^n_x|||=\aplip(v^n_i)(\varphi^n_i(x))$. Then arguing as for \eqref{eq:changevar} to relate approximate limits in different spaces we see that $\aplip(v^n_i)(\varphi^n_i(x))\leq (1+\eps_n)\aplip(u)(x)$ holds for $\mm$-a.e.\ $x\in U^n_i$  and thus  assuming, without loss of generality, that $\eps_n\leq 1$ for every $n\in\N$, we have
\begin{equation}
\label{eq:boundnnx}
|||\nn^n_x|||\leq 2\aplip(u)(x)\stackrel{\eqref{eq:aplipfin}}<\infty\qquad\mm-a.e\ x\in\X,\ \forall n\in\N.
\end{equation}
\noindent{$\mathbf{(ii)}$} We claim that for every $n,m\in\N$ it holds
\begin{equation}
\label{eq:distnorm}
\D(\nn_x^n,\nn^m_x)\leq 2\aplip(u)(x) (\eps_n+\eps_m)(1+\eps_n)\qquad\mm-a.e.\ x
\end{equation}
and to prove this we fix $n,m,i,j\in\N$ such that $\mm(U^n_i\cap U^m_j)>0$ and pick $x\in U^n_i\cap U^m_j$ so that $x,\varphi^n_i(x),\varphi^m_j(x)$ are density points of $U^n_i\cap U^m_j,\varphi^n_i(U^n_i),\varphi^m_j(U^m_j)$ respectively  and so that $v^n_i$ (resp.\ $v^m_j$) is metrically differentiable at $\varphi^n_i(x)$ (resp.\ $\varphi^m_j(x)$).

Then we have 
\begin{equation}
\label{eq:anchedopo}
\begin{split}
\sfd_\Y\big(v^n_i(\varphi^n_i(y)),v^n_i(\varphi^n_i(x))\big)&=\nn^n_x(\varphi^n_i(y)-\varphi^n_i(x))+o(|\varphi^n_i(y)-\varphi^n_i(x)|)\\
&=\nn^n_x(\varphi^n_i(y)-\varphi^n_i(x))+o(\sfd(x,y)),
\end{split}
\end{equation}
having used the fact that $\varphi^n_i$ is  biLipschitz. Since a similar identity holds for  $v^m_j$ and since $v^n_i\circ\varphi^n_i=u=v^m_j\circ\varphi^m_j$ on $U^n_i\cap U^m_j$, we deduce that
\[
\begin{split}
o(\sfd(x,y))&\stackrel{\phantom{\eqref{eq:boundnnx}}}=|\nn^n_x(\varphi^n_i(y)-\varphi^n_i(x))-\nn^m_x(\varphi^m_j(y)-\varphi^m_j(x))|\\
&\stackrel{\phantom{\eqref{eq:boundnnx}}}\geq|(\nn^n_x-\nn^m_x)(\varphi^n_i(y)-\varphi^n_i(x))|-\big|\nn^m_x\big((\varphi^n_i(y)-\varphi^n_i(x))-(\varphi^m_j(y)-\varphi^m_j(x))\big)\big|\\
&\stackrel{\eqref{eq:boundnnx}}\geq|(\nn^n_x-\nn^m_x)(\varphi^n_i(y)-\varphi^n_i(x))|-2\aplip(u)(x)|(\varphi^n_i(y)-\varphi^m_j(y))-(\varphi^n_i(x)-\varphi^m_j(x))|\\
&\stackrel{\eqref{eq:peralign}}\geq|(\nn^n_x-\nn^m_x)(\varphi^n_i(y)-\varphi^n_i(x))|-2(\eps_n+\eps_m)\aplip(u)(x)\sfd(x,y)\\
&\stackrel{\phantom{\eqref{eq:boundnnx}}}\geq|(\nn^n_x-\nn^m_x)(\varphi^n_i(y)-\varphi^n_i(x))|-2(\eps_n+\eps_m)\aplip(u)(x)(1+\eps_n)|\varphi^n_i(y)-\varphi^n_i(x)|,
\end{split}
\]
having used the fact that  $\varphi^n_i$ is $(1+\eps_n)$-biLipschitz in the last step. We can rewrite what we obtained as
\[
|(\nn^n_x-\nn^m_x)(\varphi^n_i(y)-\varphi^n_i(x))|\leq 2(\eps_n+\eps_m)(1+\eps_n)\aplip(u)(x)|\varphi^n_i(y)-\varphi^n_i(x)|+o(\sfd(x,y)),
\]
so that \eqref{eq:distnorm} follows from Proposition \ref{prop:liplocap} applied to  $\X:=\R^d$, $\Y:=\R$, the Borel set $\varphi^n_i(U^n_i)$, its density point $\varphi^n_i(x)$ and the Lipschitz function $\nn^n_x-\nn^m_x$.

A direct consequence of  \eqref{eq:distnorm} (and \eqref{eq:aplipfin}) is the fact that $n\mapsto\nn_x^n\in\sn^d$ is a Cauchy sequence for $\mm$-a.e.\ $x$. We denote its limit by $\nn_x$.

\noindent{$\mathbf{(iii)}$} We claim that $\nn_x$ is the approximate metric differential of $u$ at $x$ for $\mm$-a.e.\ $x$. Indeed from the identity $u=v^n_i\circ\varphi^n_i$  and the bound $\sfd(x,y)\geq \frac{|\varphi^n_i(y)-\varphi^n_i(x)|}{1+\eps_n}$ valid on $U^n_i$ we obtain
\[
\begin{split}
\frac{\big|\sfd_\Y(u(y),u(x))-\nn_x(\varphi^{\eps_n}_{i}(y)-\varphi^{\eps_n}_{i}(x))\big|}{\sfd(y,x)}&\leq
\frac{\big|\sfd_\Y(v^n_i(\varphi^n_i(y)),v^n_i(\varphi^n_i(x)))-\nn^n_x(\varphi^{\eps_n}_{i}(y)-\varphi^{\eps_n}_{i}(x))\big|}{\sfd(y,x)}\\
&\quad\quad\quad+(1+\eps_n)\D(\nn^n_x,\nn_x),
\end{split}
\]
so that the claim follows from \eqref{eq:anchedopo} and the fact that $\nn^n_x\to\nn_x$ as $n\to\infty$ for $\mm$-a.e.\ $x$. The fact that $x\mapsto\nn_x\in\sn^d$ is Borel follows from the Borel regularity of $x\mapsto\nn^x_n$ and pointwise convergence on a Borel set of full measure.
\end{proof}

\subsection{Definition and basic properties of the Korevaar-Schoen space} Here we introduce the main object of study of this manuscript, namely the Korevaar-Schoen-Sobolev space of metric valued maps. Let us fix a source metric measure space $(\X,\sfd,\mm)$ and a target pointed complete space $(\Y,\sfd_\Y,{\bar y})$.

Let $p\in(1,\infty)$ and $u\in L^p(\X,\Y_{\bar y})$. The $p$-energy density $\kse_{p,r}[u]:\X\to\R^+$ of $u$ at scale $r>0$  in the sense of Korevaar-Schoen is given by
\begin{equation}
\label{eq:kse}
\kse_{p,r}[u](x):=\Big|\fint_{B_r(x)} \frac{\sfd^p_\Y(u(x),u(y))}{r^p}\,\d \mm(y) \Big|^{1/p}
\end{equation}
and the (total) energy $\E_p(u)\in[0,\infty]$ is defined as
\begin{equation}
\label{eq:defeks}
\E_p(u):=\limi_{r\downarrow0}\int\kse_{p,r}^p[u]\,\d\mm.
\end{equation}
Then the Korevaar-Schoen space is introduced as:
\begin{definition}[Korevaar-Schoen space]\label{def:kss}
Let $u\in L^p(\X,\Y_{\bar y})$. We say that $u\in \ks^{1,p}(\X,\Y_{\bar y})$ provided $\E_p(u)<\infty$.
\end{definition}
\begin{remark}{\rm
The typical definition of the Korevaar-Schoen space requires the $\lims$, rather than the $\limi$, to be finite in \eqref{eq:defeks}. As we shall see soon in Corollary \ref{cor:ksh} the two conditions are equivalent under assumptions on $\X$ that we are very willing to make: we chose the approach with the $\limi$ because it is more natural in view of Proposition \ref{prop:linksob} below.
}\fr\end{remark}

\begin{proposition}\label{prop:linksob} Let $(\X,\sfd,\mm)$ be a uniformly locally  doubling space,  $(\Y,\sfd_\Y,{\bar y})$ a pointed complete space and $p\in(1,\infty)$. Then $\ks^{1,p}(\X,\Y_{\bar y})\subset W^{1,p}(\X,\Y_{\bar y})$ and there is a constant $C>0$ depending only on $\inf_{r>0}\dou(r)$  such that for any $u\in \ks^{1,p}(\X,\Y_{\bar y})$ it holds
\[
|Du|\leq C\,G,\qquad\mm-a.e.,
\]
where $G$ is any weak $L^p(\X)$-limit of $\kse_{p,\eps_n}[u]$ along some sequence $\eps_n\downarrow0$ (the fact that at least one such $G$ exists follows from the definition of $\ks^{1,p}(\X,\Y_{\bar y})$ and the reflexivity of $L^p(\X)$).
\end{proposition}
\begin{proof} \ \\
\noindent\emph{Step 1: the case $\Y=\R$.} Fix $u\in \ks^{1,p}(\X,\R)$, $r>0$ and for $\eps\in(0,r/4)$ apply Lemma \ref{le:partun} to obtain a cover of $\X$ made of balls $(B_i)_{i\in I}$ of radius $\eps$ and a partition of  unity $(\varphi_i)_{i\in I}$ subordinate to such cover as in the statement of such Lemma. Define $u_\eps:\X\to\R$ as
\[
u_\eps(x):=\sum_{i\in I}\varphi_iu_{B_i}=\sum_{i\in I} \varphi_i\fint_{B_i}u\,\d\mm.
\]
We claim that $u_\eps$ is a locally Lipschitz function in $L^p(\X)$ and that for some constant $C>0$ depending only on  $\dou(r)$ it holds
\begin{equation}
\label{eq:ksw}
\begin{split}
\|u_\eps-u\|_{L^p(\X)}&\leq C\eps \|\kse_{p,8\eps}[u]\|_{L^p(\X)},\\
\lip(u_\eps)&\leq C\kse_{p,8\eps}[u].
\end{split}
\end{equation}
These two properties easily imply the conclusion by the arbitrariness of $r>0$, point $(i)$ of Theorem \ref{thm:basesob} and the (trivial) fact that for a locally Lipschitz function the local Lipschitz constant is a $p$-weak upper gradient for any $p\in(1,\infty)$. In the computations below the value of the positive constant $C$ may change from line to line, but in any case it only depends on $\dou(r)$.

 The fact that $u_\eps$ is locally Lipschitz is obvious. Now notice that from
\[
\begin{split}
|u_\eps(x)-u(x)|^p&\leq\Big| \sum_{i:x\in B_i} |u_{B_i}-u(x)|\Big|^p\leq C\sup_{i:x\in B_i}|u_{B_i}-u(x)|^p\\
&\leq C\sup_{i:x\in B_i}\fint_{B_i}|u(y)-u(x)|^p\,\d\mm(y)\leq C\fint_{B_{8\eps}(x)}|u(y)-u(x) |^p\,\d\mm(y)
\end{split}
\]
the first in \eqref{eq:ksw} easily follows. For the second, let $x,y\in \X$ and $j\in I$ and notice that
\[
|u_{\eps}(y)-u_\eps(x)|=\Big|\sum_{i\in I}(\varphi_i(y)-\varphi_i(x))u_{B_i}\Big|=\Big|\sum_{i\in I}(\varphi_i(y)-\varphi_i(x))(u_{B_i}-u_{B_j})\Big|
\]
having used the fact that $\sum_i\varphi_i\equiv1$. Now pick $j\in I$ so that $x\in B_j$ and let $y\in B_\eps(x)$. With these choices we have that if  $y\in B_i$ then $x\in 2B_i$, thus the above gives
\begin{equation}
\label{eq:uepslip}
\begin{split}
|u_{\eps}(y)-u_\eps(x)|&\leq \sum_{i:x\in 2B_i}|\varphi_i(y)-\varphi_i(x)||u_{B_i}-u_{B_j}|\\
&\leq\sfd(x,y)\frac{C}{\eps}\sum_{i:x\in 2B_i}|u_{B_i}-u_{B_j}|\leq \sfd(x,y)\frac{C}{\eps}\sum_{i:x\in 2B_i}|u_{B_i}-u_{4B_i}|+|u_{4B_i}-u_{B_j}|.
\end{split}
\end{equation}
Now observe that $y\in B_i\cap B_\eps(x)$ and $x\in B_j\cap 2B_i$ imply $B_j\subset 4B_i$ and $4B_i\subset B_{8\eps}(x)$, thus
\[
\begin{split}
|u_{4B_i}-u_{B_j}|&\leq C\fint_{4B_i}\fint_{4B_i}|u(z)-u(w)|\,\d\mm(z)\,\d\mm(w)\\
&\leq C\fint_{4B_i}|u(z)-u(x)|\,\d\mm(z)\leq C\fint_{B_{8\eps}(x)}|u(z)-u(x)|\,\d\mm(z)\leq C\eps\, \kse_{p,8\eps}[u](x)
\end{split}
\]
and since a similar estimate holds for $|u_{4B_i}-u_{B_i}|$, the second in \eqref{eq:ksw} follows from \eqref{eq:uepslip}.
 
\noindent\emph{Step 2: the general case.} Let $\varphi:\Y\to\R$ be 1-Lipschitz and notice that from the trivial inequality $\kse_{p,\eps}[\varphi\circ u]\leq \kse_{p,\eps}[u]$ it follows that $\varphi\circ u\in \ks^{1,p}(\X,\R)$ and that if $\kse_{p,\eps_n}[u]\weakto G$ in $L^p(\X)$, then up to pass to a subsequence it also holds $\kse_{p,\eps_n}[\varphi\circ u]\weakto G_\varphi\leq G$ in $L^p(\X)$. 

Hence what already proved ensures that $\varphi\circ u\in W^{1,p}(\X)$ with $|D(\varphi\circ u)|\leq C\,G_\varphi\leq C\,G$ for some constant $C$ not depending on $u,\varphi$. The conclusion then follows from the arbitrariness of $\varphi$ and the definition of $W^{1,p}(\X,\Y_{\bar y})$.
\end{proof}
It is now easy to see that if assume not only a doubling condition, but also a Poincar\'e inequality, then the Korevaar-Schoen space coincides - as set - with the other notions of metric-valued Sobolev spaces that we have encountered:
\begin{corollary}\label{cor:ksh} Let $(\X,\sfd,\mm)$ be uniformly locally  doubling and supporting a Poincar\'e inequality, $(\Y,\sfd_\Y,\bar y)$  a pointed complete space and  $p\in(1,\infty)$. Then $\ks^{1,p}(\X,\Y_{\bar y})=W^{1,p}(\X,\Y_{\bar y})=\hs^{1,p}(\X,\Y_{\bar y})$ and for any $u\in L^p(\X,\Y_{\bar y})$ we have
\begin{equation}
\label{eq:finito}
\limi_{r\downarrow0}\int\kse_{p,r}^p[u]\,\d\mm<\infty\qquad\Leftrightarrow\qquad\lims_{r\downarrow0}\int\kse_{p,r}^p[u]\,\d\mm<\infty.
\end{equation}
Also, for $u\in \ks^{1,p}(\X,\Y_{\bar y})$, $R>0$ and Hajlasz upper gradient $G_R$ at scale $R$ it holds the inequality
\begin{equation}
\label{eq:bdkse}
\kse_{p,r}^p[u](x)\leq c(p)\Big( G_R^p(x)+\fint_{B_r(x)}G_R^p(y)\,\d\mm(y)\Big)\qquad\mm-a.e.\ x\in\X,\ \forall r\in(0,R),
\end{equation}
where $c(p)>0$ is a constant depending only on $p$.
\end{corollary}
\begin{proof}
Propositions \ref{prop:linksob} and \ref{prop:whs} give $\ks^{1,p}(\X,\Y_{\bar y})\subset W^{1,p}(\X,\Y_{\bar y})\subset \hs^{1,p}(\X,\Y_{\bar y})$. Now assume $u\in \hs^{1,p}(\X,\Y_{\bar y})$, let $G_R\in L^p(\X)$ be an Hajlasz upper gradient at scale $R$ and notice that for any $r\leq R$ it holds
\[
\frac{\sfd_\Y^p(u(y),u(x))}{r^p}\leq c(p)(G_R^p(x)+G_R^p(y))\quad\mm\times\mm\ a.e.\ (x,y)\text{ s.t. }\sfd(x,y)\leq r.
\] 
The bound \eqref{eq:bdkse} follows and since the simple Lemma \ref{le:convconv} below  ensures that the right hand side in  \eqref{eq:bdkse} is bounded in $L^1(\X)$, we obtained at once \eqref{eq:finito} and the inclusion $\hs^{1,p}(\X,\Y_{\bar y})\subset \ks^{1,p}(\X,\Y_{\bar y})$.
\end{proof}
\begin{lemma}\label{le:convconv}
Let $(\X,\sfd,\mm)$ be a uniformly locally  doubling space and $g\in L^1(\X)$. For $r>0$ define $g_r(x):=\fint_{B_r(x)}g\,\d\mm$. Then $g_r\to g$ in $L^1(\X)$ as $r\downarrow0$. 
\end{lemma}
\begin{proof} The  claim is trivial if $g\in C_b\cap L^1(\X)$, thus the general case follows by approximation if we show that the operators sending $g$ to $g_r$ are uniformly bounded in $L^1$ for $r\in(0,1)$. For this  notice that 
\[
\int\fint_{B_r(x)}g(y)\,\d\mm(y)\,\d\mm(x)=\int g(y)\fint_{B_r(y)}\frac{\mm(B_r(y))}{\mm(B_r(x))}\,\d\mm(x)\,\d\mm(y)\leq \dou(1)\int g(y)\,\d\mm(y),
\]
and the conclusion follows.
\end{proof}
We emphasize that at this stage we have been able to deduce the important property \eqref{eq:finito}, but we have not proved that the limit of $\int\kse_{p,r}^p[u]\,\d\mm$ as $r\downarrow0$ exists. It is unclear to us whether at this level of generality this really holds: we shall obtain the existence of such limit in the next section under the further assumption that the source space $\X$ is strongly rectifiable.

\subsection{The energy density}
In order to characterize the limit of  $\int\kse_{p,r}^p[u]\,\d\mm$ the following notion will turn out to be useful:
\begin{definition}[$p$-size of a seminorm]\label{def:ps1} Let $\|\cdot\|$ be a seminorm on $\R^d$ and $p\in(1,\infty)$. Its $p$-size $S_p(\|\cdot\|)$ is defined as
\[
S_p(\|\cdot\|):=\Big|\fint_{B_1(0)}{\|w\|^p}\,\d \mathcal L^d(w)\Big|^{\frac1p}= \Big|\fint_{B_r(0)}\frac{{\|z\|^p}}{r^p}\,\d \mathcal L^d(z)\Big|^{\frac1p}\qquad\forall r>0,
\]
where the balls $B_1(0),B_r(0)$ are intended in  the Euclidean norm.
\end{definition}
We then have the following result identifying the limit of $\int\kse_{p,r}^p[u]\,\d\mm$. Notice a difference with respect to the terminology in \cite{KS93}: what here we call \emph{energy density}, in \cite{KS93} was the $p$-th root of the energy density. 
\begin{theorem}[Energy density]\label{thm:exlim} Let $(\X,\sfd,\mm)$ be a strongly rectifiable metric measure space with locally uniformly doubling measure and supporting a Poincar\'e inequality and $(\Y,\sfd_\Y,{\bar y})$ a pointed complete space.  Also, let $p\in(1,\infty)$.

Then for every   $u\in\ks^{1,p}(\X,\Y_{\bar y})$  there is a function $\e_p[u]\in L^p(\X)$, called \emph{$p$-energy density} of $u$, such that
\begin{equation}
\label{eq:defendens}
\kse_{p,r}[u]\quad\to\quad \e_p[u]\quad\text{$\mm$-a.e.\ and in $L^p(\X)$ as $r\downarrow0$.}
\end{equation}
More explicitly, for any $\eps_n\downarrow0$ and aligned family of atlases $(\mathcal A^{\eps_n})$, the function $\e_p[u]$  is given by the formula
\begin{equation}
\label{eq:endensform}
\e_p[u](x)=S_p(\md_x(u))\qquad\mm-a.e.\ x\in\X,
\end{equation}
where $\md_\cdot(u)$ is the metric differential of $u$ relative to $(\mathcal A^{\eps_n})$.

In particular, the  limit in \eqref{eq:defeks} exists for any $u\in L^p(\X,\Y_{\bar y})$.
\end{theorem}
\begin{proof} The last claim is trivial if $u\notin \ks^{1,p}(\X,\Y_{\bar y})$ by the very definition of such space. If instead $u\in \ks^{1,p}(\X,\Y_{\bar y})$, then such claim follows from \eqref{eq:defendens}. Fix $\eps_n\downarrow0$ and aligned family of atlases $(\mathcal A^{\eps_n})$.
Also,  fix $u\in\ks^{1,p}(\X,\Y_{\bar y})$ and suppose we have already proved  $\mm$-a.e.\ convergence of $\kse_{p,r}[u]$ to $S_p(\md_\cdot(u))$. Then the bound \eqref{eq:bdkse} ensures that we can apply the simple Lemma \ref{le:convconv2} below to $f_r:=\kse_{p,r}^p[u]$ and $g_r$ given by the right hand side of \eqref{eq:bdkse}: Lemma \ref{le:convconv} ensures that $(g_r)$ has a limit in $L^1(\X)$ as $r\downarrow0$ and thus Lemma \ref{le:convconv2} yields  that    $\kse_{p,r}^p[u]\to S_p^p(\md_\cdot(u))$ in $L^1(\X)$ or equivalently that    $\kse_{p,r}[u]\to S_p(\md_\cdot(u))$ in $L^p(\X)$.  Therefore to conclude it is sufficient to prove that $\kse_{p,r}[u](\cdot)\to S_p(\md_\cdot(u))$ $\mm$-a.e.. Recalling the identity $\ks^{1,p}(\X,\Y_{\bar y})=\hs^{1,p}(\X,\Y_{\bar y})$ - proved in Corollary \ref{cor:ksh} - and \eqref{eq:hsluslip}, up to restrict the charts it is  not restrictive  to assume that $u\restr{U^{\eps_n}_i}$ is Lipschitz for every $n,i\in\N$.

To this aim let $R>0$ and start observing that if $x\in\X$ is a Lebesgue point of $G_R^p$ and a density point of $U\subset \X$, by passing to the limit in the trivial bound
\[
\frac1{\mm(B_r(x))}\int_{B_r(x)\setminus U}G_R^p\,\d\mm\leq 
\frac1{\mm(B_r(x))}\int_{B_r(x)\setminus U}|G_R^p(y)-G_R^p(x)|\,\d\mm(y)+ G_R^p(x)\frac{\mm(B_r(x)\setminus U)}{\mm(B_r(x))}
\]
we deduce that
\[
\lim_{r\downarrow0}\frac1{\mm(B_r(x))}\int_{B_r(x)\setminus U}G_R^p\,\d\mm=0.
\]
Using this and \eqref{eq:bdkse} it is easy to see that for any $U\subset\X$ Borel it holds
\begin{equation}
\label{eq:c1}
\lim_{r\downarrow0}\bigg|\kse_{p,r}^p[u](x)-\fint_{B_r(x)\cap U}\frac{\sfd_\Y^p(u(y),u(x))}{r^p}\,\d\mm(y)\bigg|=0\qquad  \mm-a.e.\ x\in U.
\end{equation}
Now use again the fact that $\ks^{1,p}(\X,\Y_{\bar y})=\hs^{1,p}(\X,\Y_{\bar y})$ (Corollary \ref{cor:ksh}) and  \eqref{eq:hsluslip}   to apply Proposition \ref{prop:metrdiff} and deduce that $u$ is $\mm$-a.e.\ approximately metrically differentiable relatively to $(\mathcal A^{\eps_n})$. Then putting for brevity $\nn_x:=\md_x(u)$ it is easy to see that the metric differentiability gives that
\begin{equation}
\label{eq:aplim}
\lim_{n\to\infty}\aplims_{y\to x\atop y\in U_{i(n,x)}^{\eps_n}}\frac{|\sfd^p_\Y(u(y),u(x))-\nn_x^p(\varphi^n_i(y)-\varphi^n_i(x))|}{\sfd^p(x,y)}=0\qquad\mm-a.e.\ x\in\X,
\end{equation}
where $i(n,x)\in\N$ is such that $x\in U^{\eps_n} _{i(n,x)}$. We are now in position to apply, for every $n,i,\in\N$, the trivial Lemma \ref{le:aplim} below to the set $K:=U^{\eps_n}_{i(n,x)}$ (recall that we assumed $u\restr{U^{\eps_n}_{i(n,x)}}$ to be Lipschitz) and to $\mm$-a.e.\ $x\in U^{\eps_n}_{i(n,x)}$ to deduce that
\begin{equation}
\label{eq:c2}
\lim_{n\to\infty}\lim_{r\downarrow0}\fint_{B_r(x)\cap U^{\eps_n}_{i(n,x)}}\frac{|\sfd_\Y^p(u(y),u(x))-\nn_x^p(\varphi^n_i(y)-\varphi^n_i(x))|}{r^p}\,\d\mm(y)\stackrel{\eqref{eq:aplim}}=0\qquad\mm-a.e.\ x\in \X.
\end{equation}
We now claim that 
\begin{equation}
\label{eq:c3}
\lim_{n\to\infty}\lim_{r\downarrow0}\fint_{B_r(x)\cap U^{\eps_n}_{i(n,x)}}\frac{\nn_x^p(\varphi^n_i(y)-\varphi^n_i(x))}{r^p}\,\d\mm(y)=S^p_p(\nn_x)\qquad\mm-a.e.\ x\in\X
\end{equation}
and observe that this identity together with \eqref{eq:c1} and \eqref{eq:c2} gives the conclusion. 

Fix $i,n\in\N$, put $U^n_i:=U^{\eps_n}_i$ for brevity and let $x\in U^n_i$ be a  point where $u$ is metrically differentiable.  Then we have
\begin{equation}
\label{eq:perc3}
\fint_{B_r(x)\cap U^n_{i}}\frac{\nn_x^p(\varphi^n_i(y)-\varphi^n_i(x))}{r^p}\,\d\mm(y)= \frac{1}{\mm(B_r(x)\cap U^n_i)}\int \frac{\nn_x^p(w-\varphi^n_i(x))}{r^p}\,\d(\varphi^n_i)_*\mm\restr{B_r(x)\cap U^n_{i}}(w).
\end{equation}
Now notice that since $\varphi^n_i:U^n_i\to \varphi^n_i(U^n_i)\subset \R^d$ is $(1+\eps_n)$-biLipschitz we have
\begin{subequations}
\label{eq:boundmisure}
\begin{align}
\label{eq:boundmisure1}
 \nchi_{B^{\R^d}_{r/(1+\eps_n)}(\varphi^n_i(x))}(\varphi^n_i)_*(\mm\restr{ U^n_{i}})&
\leq (\varphi^n_i)_*(\mm\restr{B^\X_r(x)\cap U^n_{i}})\leq
 \nchi_{B^{\R^d}_{r(1+\eps_n)}(\varphi^n_i(x))}(\varphi^n_i)_*(\mm\restr{ U^n_{i}})\\
\label{eq:boundmisure2}
 (1+\eps_n)^{-d}\frac{\d\mm}{\d\mathcal H^d}&
 \leq \frac{\d(\varphi^n_i)_*\mm\restr{U^n_i}}{\d\mathcal L^d}\circ\varphi^n_i\leq (1+\eps_n)^d\frac{\d\mm}{\d\mathcal H^d}\qquad\mm-a.e.\ on\ U^n_i,
\end{align}
\end{subequations}
therefore for any $f:\R^d\to\R^+$ Borel, using the inequalities on the right in the above we have
\[
\begin{split}
\frac{1}{\mm(B_r(x)\cap U^n_i)}&\int f(w)\,\d(\varphi^n_i)_*\mm\restr{B_r(x)\cap U^n_{i}}(w)\\
&\stackrel{\eqref{eq:boundmisure1}}\leq \frac{1}{\mm(B_r(x)\cap U^n_i)}\int_{B^{\R^d}_{r(1+\eps_n)}(\varphi^n_i(x))} f(w)\,\d(\varphi^n_i)_*(\mm\restr{U^n_i})(w)\\
&\stackrel{\eqref{eq:boundmisure2}}
\leq \frac{(1+\eps_n)^d}{\mm(B_r(x)\cap U^n_i)}\int_{B^{\R^d}_{r(1+\eps_n)}(\varphi^n_i(x)) } f(w)\nchi_{\varphi^n_i(U^n_i)}\frac{\d\mm}{\d\mathcal H^d}((\varphi^n_i)^{-1}(w))\,\d\mathcal L^d(w)\\
&\stackrel{\phantom{\eqref{eq:boundmisure1}}}= \frac{(1+\eps_n)^d}{\mm(B_r(x)\cap U^n_i)}\int_{B^{\R^d}_{r(1+\eps_n)}(\varphi^n_i(x)) } f(w)\rho(w)\,\d\mathcal L^d(w)\\
&\stackrel{\phantom{\eqref{eq:boundmisure1}}}= \frac{(1+\eps_n)^{2d}r^d}{\mm(B_r(x)\cap U^n_i)}\int_{B^{\R^d}_{1}(0) }( f\rho)\big(\varphi^n_i(x)+r(1+\eps_n)z\big)\,\d\mathcal L^d(z),
\end{split}
\]
where we put $\rho:=\nchi_{\varphi^n_i(U^n_i)}\frac{\d\mm}{\d\mathcal H^d}\circ(\varphi^n_i)^{-1}$. 
We shall apply this bound to $f:=\frac{\nn_x^p(\cdot-\varphi^n_i(x))}{r^p}$ to obtain
\begin{equation}
\label{eq:primalim}
\begin{split}
 \frac{1}{\mm(B_r(x)\cap U^n_i)}\int& \frac{\nn_x^p(w-\varphi^n_i(x))}{r^p}\,\d(\varphi^n_i)_*\mm\restr{B_r(x)\cap U^n_{i}}(w)\\
&\leq  \frac{(1+\eps_n)^{2d+p}r^d}{\mm(B_r(x)\cap U^n_i)}\int_{B^{\R^d}_{1}(0) } \nn_x^p(z)\rho\big(\varphi^n_i(x)+r(1+\eps_n)z\big)\,\d\mathcal L^d(z).
\end{split}
\end{equation}
Now recall that since $x$ is a point of approximate metric differentiability, we know that $x,\varphi^n_i(x)$ are density points of $U^n_i,\varphi^n_i(U^n_i)$ respectively. Assume also that $x$ is a Lebesgue point of the density $\vartheta^d(y)=\lim_{r\downarrow0}\frac{\mm(B_r(y))}{\omega_dr^d}$ (recall Theorem \ref{thm:denshaus}), that the density itself exists positive at $x$ and finally assume that $\varphi^n_i(x)$ is  a Lebesgue point of $\frac{\d\mm}{\d\mathcal H^d}\circ(\varphi^n_i)^{-1}=\vartheta^d\circ(\varphi^n_i)^{-1}$ (notice that all these conditions are satisfied for $\mm$-a.e.\ $x\in U^n_i$). Then passing to the limit in \eqref{eq:primalim} and using the continuity of $\nn_x$ and \eqref{eq:perc3} we obtain
\[
\begin{split}
\lims_{r\downarrow0}\fint_{B_r(x)\cap U^n_{i}}\frac{\nn_x^p(\varphi^n_i(y)-\varphi^n_i(x))}{r^p}\,\d\mm(y)&\leq  \frac{(1+\eps_n)^{2d+p}}{\vartheta^d(x)}\fint_{B^{\R^d}_{1}(0) } \nn_x^p(z)\vartheta^d(x)\,\d\mathcal L^d(z)\\
&=(1+\eps_n)^{2d+p}\,S_p^p(\nn_x).
\end{split}
\]
Since the lower bound $
\limi_{r\downarrow0}\fint_{B_r(x)\cap U^n_{i}}\frac{\nn_x^p(\varphi^n_i(y)-\varphi^n_i(x))}{r^p}\,\d\mm(y)\geq(1+\eps_n)^{-2d-p}S_p^p(\nn_x) $ can be obtained in a similar way by using the inequalities on the left in \eqref{eq:boundmisure}, the claim \eqref{eq:c3} and the conclusion follow.
\end{proof}
Notice that the existence of the energy density provides  the representation formula
\begin{equation}
\label{eq:repr22}
\E_p(u)
=\left\{\begin{array}{ll}
\displaystyle{\int\e_p^p[u]\,\d\mm},&\qquad\text{ if }u\in\ks^{1,p}(\X,\Y_{\bar y}),\\
+\infty,&\qquad\text{ otherwise}.
\end{array}
\right.
\end{equation}
\begin{lemma}\label{le:convconv2}
Let $(\X,\sfd,\mm)$ be a metric measure space and $(f_r),(g_r)\subset L^1(\X)$ be non-negative with 
\[
\begin{split}
f_r&\leq g_r\qquad\mm-a.e.\ \forall r>0,\\
g_r&\to g\qquad \text{in }L^1(\X),\\
f_r&\to f\qquad \mm-a.e.\ \text{ as }r\downarrow0,
\end{split}
\]
for some Borel functions $f,g$. Then $f_r\to f$ in $L^1(\X)$ as $r\downarrow0$.
\end{lemma}
\begin{proof}Let $r_n\downarrow0$ be arbitrary and use the assumption $g_r\to g$ in $L^1(\X)$ and classical properties of the space $L^1(\X)$ to find a subsequence, not relabeled, such that $g_{r_n}\leq \tilde g$ $\mm$-a.e.\ for some $\tilde g\in L^1(\X)$. Then $f_{r_n}\leq \tilde g$ $\mm$-a.e.\ for every $n$ and an application of the dominated convergence theorem gives that $f_{r_n}\to f$ in $L^1(\X)$. The conclusion  follows from the arbitrariness of the  sequence $r_n\downarrow0$.
\end{proof}

\begin{lemma}\label{le:aplim} Let $(\X,\sfd,\mm)$ be a metric measure space,  $K\subset\X$ be a Borel set, $f:K\to\R$ be a Borel and bounded function and  $x\in K$ a density point.

 Then
\[
\lims_{r\downarrow0}\fint_{B_r(x)\cap K}f\,\d\mm\leq \aplims_{y\to x\atop y\in K}f(y).
\] 
\end{lemma}
\begin{proof}
Let $\delta>0$ and $A_\delta\subset K$ be the set of $z$'s such that $f(z)>\delta+\aplims_{y\to x\atop y\in K}f(y)$. Then by definition of $\aplims$ we know that $\lim_{r\downarrow0}\frac{\mm(B_r(x)\cap A_\delta)}{\mm(B_r(x))}=0$. Since  for any $r>0$ we have
\[
\int_{B_r(x)\cap K}f\,\d\mm=\int_{B_r\cap A_\delta}f\,\d\mm+\int_{B_r\cap K\setminus A_\delta}f\,\d\mm\leq \|f\|_{L^\infty}\mm(B_r\cap A_\delta)+\mm(B_r(x))\big(\delta+\aplims_{y\to x\atop y\in K}f(y)\big),
\]
the conclusion follows by the arbitrariness of $\delta>0$.
\end{proof}
\begin{corollary}[Locality of the energy density]\label{cor:loced}
Let $(\X,\sfd,\mm)$ be a strongly rectifiable  space with uniformly locally doubling measure and supporting a Poincar\'e inequality and $(\Y,\sfd_\Y,{\bar y})$ a pointed complete space. 

Let $p\in(1,\infty)$ and $u,v\in \ks^{1,p}(\X,\Y_{\bar y})$. Then
\[
\e_p[u]=\e_p[v]\qquad\mm-a.e.\ on\ \{u=v\}.
\]
\end{corollary}
\begin{proof} Let $\eps_n\downarrow0$ be arbitrary and $(\mathcal A^{\eps_n})$ an aligned family of atlases.
By the very definition of approximate metric differentiability and the fact that $\mm$-a.e.\ $x\in\{u=v\}$ is a density point for such set (by the doubling assumption) we see that
\[
\md_x(u)=\md_x(v)\qquad\mm-a.e.\ x\in\{u=v\}.
\]
The conclusion then follows from the identity \eqref{eq:endensform}.
\end{proof}

\begin{remark}[About the Poincar\'e inequality]\label{re:poip}{\rm
For a given $p\in(1,\infty)$, the same conclusions of Theorem \ref{thm:exlim}  hold if in place of \eqref{eq:poincarelip} we assume the weaker $p$-Poincar\e' inequality:
\[
\fint_{B_r(x)}|f-f_{B_r(x)}|\,\d\mm\leq C(R)r\Big|\fint_{ B_{\lambda r}(x)} \lip^p f\,\d\mm\Big|^{\frac1p}\qquad\forall x\in\X,\ r\in(0,R)\qquad\forall f\in\Lip_{bs}(\X).
\]
Indeed:
\begin{itemize}
\item[i)] by the celebrated result by Keith-Zhong \cite{KZ08} in this case it also holds the $p'$-Poincar\'e inequality for some $p'<p$,
\item[ii)] then in the estimate \eqref{eq:stimax} (and similarly in the one below it) one can replace $M_{\lambda R}(|Du|)$ with  $|M_{\lambda R}(|Du|^{p'})|^{\frac{1}{p'}}$
\item[iii)] and since $|Du|^{p'}\in L^{\frac{p}{p'}}$, Proposition \ref{prop:maxest} grants that $M_{\lambda R}(|Du|^{p'})|\in L^{\frac{p}{p'}}(\X)$ as well, i.e.\ $|M_{\lambda R}(|Du|^{p'})|^{\frac{1}{p'}}\in L^p(\X)$.
\end{itemize}
In other words, the conclusion of Proposition \ref{prop:whs} are in place with $G_R:=C(R) |M_{2\lambda R}(|Du|^{p'})|^{\frac{1}{p'}}$ and once the equivalence of $W^{1,p}(\X,\Y_{\bar y})$ and $\hs^{1,p}(\X,\Y_{\bar y})$ is obtained, the other arguments, which do not use the Poincar\'e inequality, can be reproduced.

Similarly, from the next section on we are going to deal with the case $p=2$ and prove, among other things, the lower semicontinuity of the energy. Poincar\'e inequality will not be used, beside its application in Proposition \ref{prop:whs}, so that for the reasons just explained all the stated results would work as well assuming only the validity of the 2-Poincar\'e inequality.
}\fr\end{remark}

\section{Energy density and differential calculus}
\subsection{On the notion of differential in the non-smooth setting}

\subsubsection{Reminders about  differentials of metric-valued maps}\label{se:remdif}
In this short section we are going to recall the definition of differential of a map $u\in W^{1,2}(\X,\Y_{\bar y})$ as given in \cite{GPS18} by building up on the theory developed in \cite{Gigli14}. To keep the presentation short we assume the reader familiar with the language of $L^0$-normed modules. 

Recall that the differential of real valued Sobolev functions and the cotangent module are defined by the following:
\begin{thmdef}\label{thm:cotmod}
Let $(\X,\sfd,\mm)$ be a metric measure space. Then there exists a unique couple $(L^0(T^*\X),\d)$ where $L^0(T^*\X)$ is a $L^0(\mm)$-normed module and $\d:W^{1,2}(\X)\to L^0(T^*\X)$ is linear and satisfies
\begin{itemize}
\item[i)] for any $f\in W^{1,2}(\X)$ it holds $|\d f|=|Df|$ $\mm$-a.e.,
\item[ii)] the space $\{\d f\ :\ f\in W^{1,2}(\X)\}$ generates $L^0(T^*\X)$.
\end{itemize}
Uniqueness is intended up to unique isomorphism, i.e.\ if $(\mathscr M,\tilde \d)$ has the same properties, then there is a unique isomorphism $\Phi:L^0(T^*\X)\to\mathscr M$ such that $\tilde \d=\Phi\circ \d$.
\end{thmdef}
As studied in \cite{GP16}, in the setting of strongly rectifiable spaces, the cotangent module (and its dual the tangent module $L^0(T\X)$) are tightly linked to the geometry of the underlying space. In particular, the following holds:
\begin{theorem}[Dimension of (co)tangent module]\label{thm:dimcot}
Let $(\X,\sfd,\mm)$ be a strongly rectifiable space of dimension $d$ which is locally doubling and supporting a Poincar\'e inequality. 

Then the modules $L^0(T^*\X)$ and $L^0(T\X)$ have dimension $d$.
\end{theorem}
For general metric measure structures the structure of the (co)tangent module can be rather complicated, but at least if the metric structure is the Euclidean one some link between such abstract notions and more concrete ones can be established: the following result has been obtained in \cite{GP16-2}, see also \cite{LP19} for more recent development on the topic.
\begin{theorem}[Tangent module in the Euclidean setting]\label{thm:tangrd} Let $d\in\N$ and consider the space $\R^d$ equipped with the Euclidean distance and a non-negative and non-zero Radon measure $\mu$. 

Then there is a canonical embedding $\iota$ of the tangent module $L^0(T\R^d)$ into the space $L^0(\R^d,\R^d;\mu)$ of Borel vector fields on $\R^d$ identified up to $\mu$-a.e.\ equality. In particular the dimension of $L^0(T\R^d)$ is bounded above by $d$.
\end{theorem}

We now turn to the definition of differential for a metric valued map $u:\X\to\Y$. This is given in terms of Sobolev functions on both the source and the target space, where the latter will typically be equipped with the measure  $\mu:=u_*(|Du|^2\mm)$. In order to emphasize the dependence of such structure on the choice of the measure $\mu$ (and thus on the function $u$) we shall denote by $\d_\mu$ the differential operator on $(\Y,\sfd_\Y,\mu)$ and by $L^0_\mu(T^*\Y),L^0_\mu(T\Y)$ the corresponding cotangent and tangent modules.

The definition of differential of $u$ is given by duality with that of pullback of Sobolev functions. This latter operation is the one studied in the following lemma: 
\begin{lemma}[Pullback of Sobolev functions]\label{le:compo} Let $(\X,\sfd,\mm)$ be a metric measure space and $(\Y,\sfd_\Y,{\bar y})$ a pointed complete space.

Let  $u\in W^{1,2}(\X,\Y_{\bar y})$, put $\mu:=u_\ast(|Du|^2\mm)$ and let $f\in W^{1,2}(\Y,\sfd_\Y,\mu)$. Then there is $g\in S^2(\X)$ such that $g=f\circ u$ $\mm$-a.e.\ on $\{|Du|>0\}$ and 
\begin{equation}
\label{eq:chain}
|\d g|\leq |\d_\mu f|\circ u|D u|\qquad\mm-a.e..
\end{equation}
More precisely, there is  $g\in S^2(\X)$ and a sequence $(f_n)\subset \LIP_{bs}(\Y)$ such that
\begin{equation}
\label{eq:conv}
\begin{array}{rllrll}
f_n& \to\ f\qquad& \mu-a.e.&\qquad\qquad\lip_a(f_n)& \to\  |\d_\mu f|\qquad &\text{\rm in }L^2(\mu),\\
f_n\circ u& \to\ g\qquad& \mm-a.e.&\qquad\qquad\lip_a(f_n)\circ u|Du|& \to\  |\d_\mu f|\circ u|D u|&\text{\rm in }L^2(\mm).
\end{array}
\end{equation}
Moreover, if $f$ is also bounded, then the $f_n$'s can be taken to be equibounded.
\end{lemma}
\begin{proof} See \cite[Proposition 3.3]{GPS18}. The last claim is trivial by truncation, as already noticed in the course of the proof of  \cite[Proposition 3.3]{GPS18}.
\end{proof}
Ideally, we would like to define $\d u$ as $L^0(\mm)$-linear map from $L^0(T\X)$ to the pullback $u^*L^0_\mu(T\Y)$ of $L^0_\mu(T\Y)$ via $u$. A technical issue in doing so is that  $u^*L^0_\mu(T\Y)$ is not  a $L^0(\mm)$-normed module, but  a $L^0(|Du|^2\mm)$-normed module, or equivalently a $L^0(\nchi_{\{|Du|>0\}}\mm)$-normed module. Yet, it is clear that $\d u$ should be 0 on $\{|Du|=0\}$ so that what we should do is to produce a $L^0(\mm)$-normed module which `coincides with $u^*L^0_\mu(T\Y)$ on $\{|Du|>0\}$' and the easiest way to do so is to require that such module `contains only the 0 element on $\{|Du|=0\}$'.

This procedure is done by the extension functor that we now describe.  Let $E\subset\X$ be Borel and notice that we have a natural projection/restriction operator ${\rm proj}:L^0(\mm)\to L^0(\mm\restr E)$ given by passage to the quotient up to equality $\mm$-a.e.\ on $E$\ and a natural `extension' operator ${\rm ext}:L^0(\mm\restr E)\to L^0(\mm)$ which sends $f\in L^0(\mm\restr E)$ to the function equal to $f$ $\mm$-a.e.\ on $E$ and to $0$ on $\X\setminus E$. Now let $\mathscr M$ be a $L^0(\mm\restr E)$-normed module. The \emph{extension} of $\mathscr M$ is the $L^0(\mm)$-normed module  ${\rm Ext}(\mathscr M)$ defined as a set by ${\rm Ext}(\mathscr M):=\mathscr M$, equipped with the  multiplication of $v\in {\rm Ext}(\mathscr M) $ by $f\in L^0(\mm)$ given by   ${\rm proj}(f)v\in \mathscr M={\rm Ext}(\mathscr M)$ and with the pointwise norm defined as ${\rm ext}(|v|)\in L^0(\mm)$. We shall denote by ${\rm ext}:\mathscr M\to {\rm Ext}(\mathscr M)$ the identity map.
\begin{definition}\label{def:diff} Let $(\X,\sfd,\mm)$ be a metric measure space,  $(\Y,\sfd_\Y,{\bar y})$ a pointed complete space and $u\in W^{1,2}(\X,\Y_{\bar y})$. 
	Then the \emph{differential} $\d u$ of $u$ is the operator 
	\[ 
	\d u:\, L^0(T\X)\to {\rm Ext}\big((u^\ast L^0_\mu(T^\ast\Y))^\ast \big)
	\] given as follows. 
	For $\vv\in L^0(T\X)$, the object $\d u(\vv)\in {\rm Ext}\big((u^\ast L^0_\mu(T^\ast\Y))^\ast\big)$ is characterized by the property: for every $f\in W^{1,2}(\Y,\sfd_\Y,\mu)$ 
	and every $g\in S^2 (\X,\sfd_\X,\mm)$ as in Lemma \ref{le:compo} we have 
\begin{equation}\label{eq:def_du}
{\rm ext}\big([u^\ast \d_\mu f]\big)\big(\d u(\vv)\big)=\d g(\vv)\quad \mm-a.e..
\end{equation}
\end{definition}
The fact that such definition is well posed is the content of the next proposition, see \cite[Proposition 3.5]{GPS18}, which also provides the compatibility \eqref{eq:samenorm2} between two natural notions of `norm of the differential'.
\begin{proposition}[Well posedness of the definition]\label{prop:basic}  Let $(\X,\sfd,\mm)$ be a metric measure space, $(\Y,\sfd_\Y,{\bar y})$ a pointed complete space and $u\in W^{1,2}(\X,\Y_{\bar y})$. Then the differential $\d u$ of $u$ in Definition \ref{def:diff} is well-defined and the map  $\displaystyle \d u:\,L^0(T\X)\to {\rm Ext}\big((u^\ast L^0_\mu(T^\ast\Y))^\ast \big)$ is $L^0(\mm)$-linear and continuous. Moreover, it holds that
\begin{equation}
\label{eq:samenorm2}
|\d u|=|Du|\qquad\mm-a.e..
\end{equation}
\end{proposition}

\subsubsection{Differential and metric differential}
In the last section we have seen the definition of differential for a metric valued Sobolev map and in Theorem \ref{thm:kir} we have seen the one of metric differential for a metric valued Lipschitz map on $\R^d$. It is natural to wonder whether the two concepts are compatible: the positive answer is given in the following result, proved in  \cite[Theorem 4.7]{GPS18}:

\begin{theorem}\label{thm:diffmetr} Let $(\Y,\sfd_\Y,{\bar y})$ be a pointed complete space, $u:\R^d\to\Y$ be a Lipschitz map which is also in $W^{1,2}(\R^d,\Y_{\bar y})$ and ${\sf v}\in\R^d\sim T_0\R^d$. Denote by $\bar{\sf v}\in L^0(T\R^d)$ the vector field constantly equal to ${\sf v}$. Then
\[
|\d u(\bar{\sf v})|(x)=\md_x(u)({\sf v})\qquad\mathcal L^d-a.e.\ x\in\R^d.
\]
\end{theorem}
In this  section we will  extend Theorem \ref{thm:diffmetr} to the case of strongly rectifiable spaces. In order to do so, we need to recall the link between the `abstract and analytic' tangent module and the `concrete and geometric' bundle obtained by `gluing  one copy of $\R^d$ for each point of $\X$'. Such link has been established in \cite{GP16}: to recall it we need some intermediate definition and result.

First of all, we define the \emph{geometric tangent bundle} of the strongly rectifiable space $\X$ of dimension $d$ as
\[
T_{\rm GH}\X:=\X\times\R^d,
\]
and then we define the space of its Borel sections up to $\mm$-a.e.\ equality in the natural way as
\[
L^0(T_{\rm GH}\X):=\big\{\text{Borel maps from $\X$ to $\R^d$ identified up to $\mm$-a.e.\ equality}\big\}.
\]
Of course, this definition alone does not make much sense: what is relevant is the way $L^0(T_{\rm GH}\X)$ is related to $X$ and to the calculus on it: such link is recalled in Theorem \ref{thm:isotang} below and is established via the use of an aligned family of atlases (in the same spirit as in Proposition \ref{prop:metrdiff}).

The following lemma is useful as it defines the differential of a coordinate map (which in our axiomatisation is only defined on a Borel set), see  \cite[Theorem 2.5]{GP16} for the proof:
\begin{lemma}
Let $(\X,\sfd,\mm)$ be such that $W^{1,2}(\X)$ is reflexive.  Let $U\subset \X$ be Borel and $\varphi:U\to\R^d$ be such that for some constants $L,C>0$ it holds
\[
\begin{split}
\varphi:U\to\varphi(U)\text{ is $L$-biLipschitz},\\
C^{-1}\mathcal L^d\restr{\varphi(U)}\leq \varphi_*(\mm\restr U)\leq  C\mathcal L^d\restr{\varphi(U)}.
\end{split}
\]
Then there is a unique linear and continuous operator $\ud\varphi:L^0(T\X)\restr U\to L^0(\varphi(U),\R^d)$, called differential of $\varphi$,  that for any ${\sf v}\in L^0(T\X)\restr U$ satisfies:
\begin{equation}
\label{eq:defud}
\begin{split}
\d g(\ud\varphi(\vv))&=\d(g\circ\bar\varphi)(\vv)\circ\varphi^{-1}\qquad\forall g\in \Lip_{bs}(\R^d),\\
\ud\varphi(f\vv)&=f\circ\varphi^{-1}\ud\varphi(\vv)\qquad\forall f\in L^0(\mm),
\end{split}
\end{equation}
where $\bar\varphi:\X\to\R^d$ is any Lipschitz extension of  $\varphi$. Moreover, $\ud\varphi$ satisfies
\[
L^{-1}|\vv|\circ\varphi^{-1}\leq |\ud\varphi(\vv)|\leq L|\vv|\circ\varphi^{-1}\qquad\mathcal L^d-a.e.\ on\ \varphi(U).
\]
\end{lemma}
We then have the following result:
\begin{theorem}[Abstract and concrete tangent modules]\label{thm:isotang} Let   $(\X,\sfd,\mm)$ be a strongly rectifiable space such that  $W^{1,2}(\X)$ is reflexive. Let $\eps_n\downarrow0$ be a given sequence and $(\mathcal A^{\eps_n})$ an aligned family of atlases.

Then:
\begin{itemize}
\item[i)] for every $n\in\N$ the $L^0(\mm)$-linear and continuous  map $\mathscr I_n:L^0(T\X)\to L^0(T_{\rm GH}\X)$ is well defined by the formula
\[
\nchi_{U^n_i}\mathscr I_n(\vv)=\ud\varphi^n_i(\nchi_{U^n_i}\vv)\circ\varphi^n_i,\qquad\forall i\in\N,\ \vv\in L^0(T\X),
\]
\item[ii)] the sequence $(\mathscr  I_n)$ is Cauchy in ${\rm Hom}(L^0(T\X), L^0(T_{\rm GH}\X))$,
\item[iii)] the limit map $\mathscr I:L^0(T\X)\to L^0(T_{\rm GH}\X)$ is an isometric isomorphism of modules.
\end{itemize}
\end{theorem}
\begin{proof} The existence of $\mathscr I$ is the content of  \cite[Theorem 5.2]{GP16}, its construction as limit of the maps $\mathscr I_n$ is the content of the proof of such result.
\end{proof}

\begin{remark}[About the dependence on $p$ of the differential calculus]\label{re:depp}{\rm
Theorem/Definition \ref{thm:cotmod} can be stated for any Sobolev exponent $p\in(1,\infty)$ but, without appropriate assumptions on the space, the resulting cotangent module and differentiation operator may \emph{depend} on $p$. One of the problems is in the fact that the minimal $p$-weak upper gradient also may depend on $p$, so that for $f\in W^{1,p}\cap W^{1,p'}(\X)$ its $p$-minimal weak upper gradient and $p'$-minimal weak upper gradient can be different (see e.g.\ \cite{DiMarinoSpeight13}). Then point $(i)$ of the statement shows that $\d_pf$ and $\d_{p'}f$ must be different.

Still, there are circumstances where the space $\X$ is `good enough' so that such differences do not occur. For instance, it can be proved that this is the case for doubling spaces supporting a Poincar\'e inequality (using the results in \cite{Cheeger00}) or $\RCD(K,\infty)$ ones (using the results in \cite{GigliHan14}). 

For what concerns our discussion, a more complicated issue occurs when dealing with metric valued Sobolev maps, because in this case the (co)tangent module on $(\Y,\sfd_\Y,u_*(|Du|^p\mm))$ appears and there is no reasonable regularity assumption one can make on such metric measure structure. The result is that  regardless of the regularity of $\X$, the differential $\d_pu$ of $u\in W^{1,p}(\X,\Y_{\bar y})$ a priori depends on $p$. 

We could have developed the theory presented here even for maps $W^{1,p}(\X,\Y_{\bar y})$ to prove, for instance, that the $p$-energy $\E_p$ is lower-semicontinuous on $L^p(\X,\Y_{\bar y})$, but that would have required to either carry on with the additional notational burden of indicating in some way the dependence on $p$ of the various differentiation operators, or avoiding doing so at the risk of generating confusion when different exponents are compared. For instance, the  representation formula  \eqref{eq:endensform} and  \eqref{eq:repr2} link the 2-differential $\d u$ to the metric object $\md_x(u)$ which is unrelated to Sobolev calculus - this seems to suggest some form of link between different $p$-differentials of the same map $u$.

These kind of discussions are outside the scope of this manuscript, so we preferred to concentrate on the key case $p=2$ only.
}\fr\end{remark}

\subsection{Reproducing formula for the energy density and lower semicontinuity of the energy}
In this section we do three things. The first is of technical nature and will be useful for the other results here: in Theorem \ref{thm:mdd} we generalize Theorem \ref{thm:diffmetr} to the case of maps defined on strongly rectifiable spaces. The second is to provide another representation for the energy density: while formula \eqref{eq:endensform} relates it to the $p$-size of the metric differential, in Theorem \ref{thm:repr} below - using Theorem \ref{thm:mdd} - we relate it to the $p$-size of the differential (still sticking to the case $p=2$, see definition \ref{def:ps2} below). Finally, using this formula we will achieve our third and main goal of this section: we shall prove in Theorem \ref{thm:lscks} that the energy $\E_2$ associated to the Korevaar-Schoen space is lower semicontinuous. This result is  based on the closure properties of the abstract differential that we have encountered in the previous sections.

\bigskip

We start with the following technical lemma:
\begin{lemma}\label{le:dn}
Let  $(\X,\sfd,\mm)$ be such that $W^{1,2}(\X)$ is reflexive, $(\Y,\sfd_\Y,\bar y)$ a pointed complete space and $u\in W^{1,2}(\X,\Y_{\bar y})$.

Then for every $\vv\in L^0(T\X)$ we have
\begin{equation}
\label{eq:iddif}
|\d u(\vv)|=\esssup_{f\in \Lip_{bs}(\Y)\atop\Lip(f)\leq 1,\  f(\bar y)=0}\d (f\circ u)(\vv)\qquad\mm-a.e..
\end{equation}
\end{lemma}
\begin{proof} Put $\mu:=u_*(|D u|^2\mm)$ and notice that by the definition of $\d u(\vv)$ and the fact that $\{{\rm ext}\big([u^\ast \d_\mu f]\big):f\in W^{1,2}\cap L^\infty(\Y,\sfd_\Y,\mu)\}$ generates ${\rm Ext}\big(u^*L^0_\mu(T^*\Y)\big)$ (see \cite{Gigli14}) we see that
\begin{equation}
\label{eq:peraltra}
|\d u(\vv)|=\esssup_{(E_i)}\esssup_{(f_i)\subset W^{1,2}\cap L^\infty(\Y,\sfd_\Y,\mu) \atop |{\rm ext}([u^\ast \d_\mu f_i])|\leq 1\ \mm-a.e.\ on \ E_i}\nchi_{E_i}{\rm ext}\big([u^\ast \d_\mu f_i]\big)(\d u(\vv))
\qquad\mm-a.e.,
\end{equation}
where the first essential supremum is among all Borel partitions $(E_i)$ on $\X$. Let $f:\Y\to\R$ be 1-Lipschitz with bounded support and notice that $f\in  W^{1,2}\cap L^\infty(\Y,\sfd_\Y,\mu)$ with $|{\rm ext}([u^\ast \d_\mu f])|\leq 1$ $\mm$-a.e., and thus recalling Definition \ref{def:diff} we see that
\[
|\d u(\vv)|\geq {\rm ext}([u^\ast \d_\mu f])(\d u(\vv))=\d g(\vv)\qquad\mm-a.e.,
\]
for $g$ given by Lemma \ref{le:compo}. Thus to prove that inequality $\geq$ holds in \eqref{eq:iddif} it is sufficient to show that 
\begin{equation}
\label{eq:perchiu}
\d g(\vv)=\d(f\circ u)(\vv)\qquad \mm-a.e..
\end{equation}
To this aim notice that the locality of the differential  and the fact that $g=f\circ u$ $\mm$-a.e.\ on the set $\{|\d u|>0\}$ proves \eqref{eq:perchiu} on such set. On the other hand, \eqref{eq:chain} and the trivial bound $|\d (f\circ u)|\leq |\d u|$ ensure that $\mm$-a.e.\ on $\{|\d u|=0\}$ both sides of \eqref{eq:perchiu} are 0, thus proving  \eqref{eq:perchiu} and  inequality $\geq$ holds in \eqref{eq:iddif}.

Recalling \eqref{eq:samenorm2} we see that the opposite inequality is trivial on $\{|Du|=0\}$. Also, by  a simple localization argument we can, and will, assume that $|\vv|\in L^\infty(\X)$.

Now fix $E\subset \{|Du|>0\}$ compact and $f\in W^{1,2}\cap L^\infty(\Y,\sfd_\Y,\mu)$ with $|{\rm ext}([u^\ast \d_\mu f])|\leq 1$ $\mm$-a.e.\ on  $E$. Let $F\subset\Y$ be defined up to $\mu$-null sets by $F:=\{\frac{\d u_*(\nchi_E|Du|^2\mm)}{\d u_*(|Du|^2\mm)}>0\}$ and notice that $|\d_\mu f|\leq 1$ $\mu$-a.e.\ on $F$.

Apply Lemma \ref{le:compo} to find $g\in S^2(\X)$ and a sequence $(f_n)\subset \Lip_{bs}(\Y)$ of uniformly bounded functions (by the last claim in the lemma) satisfying \eqref{eq:conv}. Let $\eta:\X\to[0,1]$ be 1-Lipschitz, with bounded support and identically 1 on $E$ and notice that the functions $\eta f_n\circ u$ are uniformly bounded in $L^2(\X)$ (because they are uniformly bounded and with uniformly bounded support) and satisfy
\[
|\d(\eta f_n\circ u)|\leq  \lip_a(f_n)\circ u|Du|+ |f_n|\circ u\,\nchi_{\supp(\eta)}\qquad\mm-a.e..
\]
Therefore recalling the last in \eqref{eq:conv} we   see that these functions are uniformly bounded in $W^{1,2}(\X)$ and thus, since we assumed such space to be reflexive, they have a non-relabeled subsequence weakly converging to some function $\tilde g$ which, by \eqref{eq:conv}, coincides with $g$ on $E$. We now apply Mazur's lemma to the sequence $(\eta f_n\circ u)$ to find a sequence of convex combinations $W^{1,2}(\X)$-strongly converging to $\tilde g$. Clearly, these convex combinations can be written as $\eta \tilde f_n\circ u$ where the $\tilde f_n$'s are convex combinations of the $f_n$'s and it is then easy to see that they belong to $\Lip_{bs}(\Y)$ and satisfy \eqref{eq:conv}.

From $\eta \tilde f_n\circ u\to \tilde g$ in $W^{1,2}(\X)$ and the fact that $\tilde g=g$ on $E$ we deduce that 
\[
\nchi_E\d(f_n\circ u)(\vv)\quad\to\quad\nchi_E\d g(\vv)\qquad\text{ in }L^2(\X,\mm).
\]
Now fix $\eps>0$ and apply Egorov's theorem to find $F_\eps\subset F$ compact with 
\begin{equation}
\label{eq:mispicc}
\mu(F\setminus F_\eps)< \eps\qquad \mm(E\setminus E_\eps)<\eps\qquad\text{  where } E_\eps:= E\cap u^{-1}(F_\eps)
\end{equation}
and a further non-relabeled extraction of subsequence  such that 
\[
\begin{split}
\lip_a(\tilde f_n)&\to |\d_\mu f| \qquad\text{uniformly on }F_\eps\\
\d( \tilde f_n\circ u)(\vv)&\to \d g(\vv) \qquad\text{uniformly on }E_\eps.
\end{split}
\]
In particular, since $|\d_\mu f|\leq 1$ $\mu$-a.e.\ on $F$,  possibly removing another small  set from $F_\eps$  - keeping \eqref{eq:mispicc} valid - we can find  $\bar n\in\N$ such that
\begin{equation}
\label{eq:diffdiff}
\begin{split}
\lip_a(\tilde f_{\bar n})(y)&< 1+\eps\qquad \text{ for every }y\in F_\eps,\\
|\d(\tilde f_{\bar n}\circ u)(\vv)- \d g(\vv)|&<\eps \qquad\mm-a.e.\   on \ E_\eps.
\end{split}
\end{equation}
We thus proved that  any $y\in F_\eps$ has a neighbourhood $U_y$ where $\tilde f_{\bar n}$  is $(1+\eps)$-Lipschitz. Hence restricting $\tilde f_{\bar n}$ to $U_y$  and then applying the McShane extension lemma we find a $1$-Lipschitz function $h_y$ which coincides with $(1+\eps)^{-1}\tilde f_{\bar n}$ on $U_y$. Up to adding a constant, which does not affect the differential, we can also assume that $h_y(\bar y)=0$. By compactness of $F_\eps$ there are $y_1,\ldots,y_N$ such that $F_\eps\subset\cup_iU_{y_i}$. Putting $E_{\eps,i}:=E_\eps\cap u^{-1}(U_{y_i})$, the locality of the differential gives $\d(\tilde f_{\bar n}\circ u)(\vv)= (1+\eps)\d(h_{y_i}\circ u)(\vv)$
and therefore
\[
\begin{split}
{\rm ext}\big([u^\ast \d_\mu f]\big)(\d u(\vv))\stackrel{\eqref{eq:def_du}}=\d g(\vv)\stackrel{\eqref{eq:diffdiff}}\leq \d(\tilde f_{\bar n}\circ u)(\vv)+\eps= (1+\eps)\d(h_{y_i}\circ u)(\vv)+\eps\quad\mm-a.e.\ on \ E_{\eps,i}.
\end{split}
\]
The conclusion follows from this inequality, the arbitrariness of $E$ and $f$ as chosen above, the identity \eqref{eq:peraltra}, the arbitrariness of $\eps>0$ and the bounds \eqref{eq:mispicc}.
\end{proof}

The first application of such lemma is in the proof of the following generalization of Theorem \ref{thm:diffmetr}:

\begin{theorem}\label{thm:mdd}
Let $(\X,\sfd,\mm)$ be a strongly rectifiable space  with uniformly locally doubling measure and supporting a Poincar\'e inequality and $(\Y,\sfd_\Y,\bar y)$ a pointed complete metric space. Also, let $\eps_n\downarrow0$ be a given sequence and $(\mathcal A^{\eps_n})$ an aligned family of atlases and $\mathscr I:L^0(T\X)\to L^0(T_{\rm GH}\X)$ the isomorphism given by Theorem \ref{thm:isotang}.

Then for every $u\in W^{1,2}(\X,\Y_{\bar y})$ and $\vv\in L^0(T\X)$ we have
\[
|\d u(\vv)|(x)=\md_x(u)(\mathscr I(\vv)(x))\qquad\mm-a.e.\ x\in\X,
\]
where $\md_x(u)$ is the approximate metric differential of $u$  relative to $(\mathcal A^{\eps_n})$.
\end{theorem}
\begin{proof} Up to post-compose $u$ with the Kuratowski embedding we can assume that $u$ takes values in $\ell^\infty(\Y)$. 

It is not restrictive to assume that for every $n,i$  the set $U^n_i\subset \X$ is compact and (by \eqref{eq:hsluslip} and Proposition \ref{prop:whs}) that $u\restr{U^n_i}$ is Lipschitz, therefore recalling Lemma \ref{le:ext} for every $i,n\in\N$ there is  $v^n_i\in \Lip(\R^d,\Y)$ such that $v^n_i\circ\varphi^n_i=u$ on $U^n_i$ and with a simple truncation argument we can assume that $v^n_i$ has bounded support. Then letting $\bar\varphi^n_i:\X\to\R^d$ be any Lipschitz extension of $\varphi^n_i$,  for any ${\sf v}\in L^0(T\X)$ and $f:\Y\to\R$ 1-Lipschitz with $f(\bar y)=0$ and bounded support we have
\[
\d(f\circ u)(\vv)=\d(f\circ v^n_i\circ\bar\varphi^n_i)(\vv)\stackrel{\eqref{eq:defud}}=\d(f\circ v^n_i)(\ud\varphi^n_i(\vv) )\circ\varphi^n_i\qquad\mm-a.e.\ on\ U^n_i.
\]
Passing to the essential supremum in $f$ and recalling Lemma \ref{le:dn} (applicable by point $(v)$ in Theorem \ref{thm:basesob}) we deduce that
\[
|\d u(\vv)|=|\d v^n_i(\ud\varphi^n_i(\vv))|\circ \varphi^n_i \qquad\mm-a.e.\ on\ U^n_i.
\]
Then by Theorem \ref{thm:diffmetr}, point $(i)$ of Proposition \ref{prop:metrdiff} and Theorem \ref{thm:isotang} and with the notation introduced there we have that for $\mm$-a.e.\ $x\in U^n_i$ it holds
\[
|\d v^n_i(\ud\varphi^n_i(\vv))|(\varphi^n_i(x))=\md_{\varphi^n_i(x)}(v^n_i)\big(\ud\varphi^n_i(\vv)(\varphi^n_i(x))\big)=\nn^n_x\big(\ud\varphi^n_i(\vv)(\varphi^n_i(x))\big)=\nn^n_x(\mathscr I_n(\vv)(x)).
\]
We thus see that for every $n\in\N$ it holds
\[
|\d u(\vv)|(x)=\nn^n_x(\mathscr I_n(\vv)(x))\qquad\mm-a.e.\ x\in\X.
\]
Now observe that Proposition \ref{prop:metrdiff} ensures that $\nn^n\to\nn$ in $L^0(\X,\sn^d)$ and coupling this information with the convergence of  $\mathscr I_n$ to $\mathscr I$ in  ${\rm Hom}(L^0(T\X), L^0(T_{\rm GH}\X))$ granted by  Theorem \ref{thm:isotang} we easily deduce that $\nn^n_\cdot(\mathscr I_n(\vv)(\cdot))\to\nn_\cdot(\mathscr I(\vv)(\cdot))$ in $L^0(\X)$ thus obtaining the conclusion.
\end{proof}
We now pass to the study of a new representation formula for the energy density. Recall that a $L^0(\mm)$-normed module $\mathscr H$ is said a Hilbert module provided
\[
2(|v|^2+|w|^2)=|v+w|^2+|v-w|^2\qquad\mm-a.e.\quad\forall v,w\in\mathscr H,
\]
and that a Hilbert base is a collection $\e_1,\ldots,\e_d\in\mathscr H$ which generates $\mathscr H$ and satisfies
\[
\la\e_i,\e_j\ra=\delta_{ij}\quad\mm-a.e.\qquad\forall i,j=1,\ldots,d.
\]
If $\mathscr H$ admits such a base made of $d$ elements, we say that $\mathscr H$ has dimension $d$.
\begin{definition}[$2$-size of an operator on a Hilbert module]\label{def:ps2}
Let $\mathscr H$ be a $L^0(\mm)$-Hilbert module of dimension $d$ and $\mathscr M$ a $L^0(\mm)$-normed module. Let $T:\mathscr H\to \mathscr M$ a $L^0(\mm)$-linear and continuous map.

Then the $2$-size $S_2(T)\in L^0(\mm)$ is defined as
\begin{equation}
\label{eq:defs2op}
S_2(T)(x):=\bigg|\fint_{B_1^{\R^d}(0)}\Big|T\big(\sum^d_{i=1}v_i\e_i\big)\Big|_{\mathscr M}^2(x)\,\d\mathcal L^d(v_1,\ldots,v_d)\bigg|^{\frac12},\qquad\mm-a.e.\ x,
\end{equation}
where $\e_1,\ldots,\e_d\in\mathscr H$ is a Hilbert base of $\mathscr H$ and $|\cdot|_{\mathscr M}$ is the pointwise norm on $\mathscr M$.
\end{definition}
Notice that if the target module is also Hilbertian, then the 2-size coincides, up to a multiplicative dimensional constant, with the Hilbert-Schmidt norm of $T$, see Lemma \ref{le:normhs}.

The fact that orthogonal transformations of $\R^d$ preserve the Lebesgue measure grants that for any linear operator $\ell:\R^d\to B$, where $B$ is a Banach space, it holds
\[
\fint_{B_1^{\R^d}(0)}\|\ell(z)\|_B^2\,\d\mathcal L^d(z)=\fint_{B_1^{\R^d}(0)}\|\ell(O(z))\|_B^2\,\d\mathcal L^d(z)\qquad\forall O\in O(d).
\] 
From this identity it is immediate to verify that the definition of $2$-size is well-posed, i.e.\ it does not depend on the chosen base $(\e_i)$. Similarly, it is easy to see that there exists a constant $c(d)>0$ such that for any $\ell$ as above it holds
\begin{equation}
\label{eq:normequiv0}
c(d)\|\ell\|_{\rm op}\leq\bigg|\fint_{B_1^{\R^d}(0)}\|\ell(z)\|_B^2\,\d\mathcal L^d(z)\bigg|^{\frac12}\leq \|\ell\|_{\rm op},
\end{equation}
where $\|\ell\|_{\rm op}:=\sup_{|v|\leq 1}\|\ell(v)\|_B$, and from these bounds, \eqref{eq:defs2op} and the observation that
\begin{equation}
\label{eq:normaopbase}
|T|=\sup_{v\in\R^d\atop |v|\leq 1}|T(\sum_iv_i \e_i)|,\qquad\mm-a.e.
\end{equation}
where $\e_1,\ldots,\e_d$ is any Hilbert base of $\mathscr H$,  it follows that
\begin{equation}
\label{eq:normequiv}
c(d)|T|\leq S_2(T)\leq |T|\qquad\mm-a.e.,
\end{equation}
for any $T$ as in Definition \ref{def:ps2}.

It is now easy to see, as direct consequence of Theorems \ref{thm:exlim},\ref{thm:isotang}, \ref{thm:mdd} and  Definitions \ref{def:ps1}, \ref{def:ps2}, that the following holds:
\begin{theorem}\label{thm:repr}
Let $(\X,\sfd,\mm)$ be a strongly rectifiable space  with uniformly locally doubling measure and supporting a Poincar\'e inequality and $(\Y,\sfd_\Y,\bar y)$ a pointed complete metric space.

Then for every $u\in\ks^{1,2}(\X,\Y_{\bar y})$ we have that
\begin{equation}
\label{eq:repr2}
\e_2[u]=S_2(\d u )\qquad\mm-a.e.
\end{equation}
\end{theorem}
\begin{proof} For $v\in\R^d$ let $\hat v\in L^0(T_{\rm GH}\X)$ be defined as constantly equal to $v$. Then, considering all the integrals below in the Bochner sense in the space $L^1(\X,\mm)$, we have
\[
\begin{split}
S_2^2(\md_\cdot(u))&=\fint_{B_1^{\R^d}(0)}|\md_\cdot u(v)|^2\,\d\mathcal L^d(v)=\fint_{B_1^{\R^d}(0)}|\md_\cdot u(\hat v(\cdot))|^2\,\d\mathcal L^d(v).
\end{split}
\]
Letting $\mathscr I:L^0(T\X)\to L^0(T_{\rm GH}\X)$ be the isomorphism defined in Theorem \ref{thm:isotang} and recall Theorem  \ref{thm:mdd} we obtain that
\[
\fint_{B_1^{\R^d}(0)}|\md_\cdot u(\hat v(\cdot))|^2\,\d\mathcal L^d(v)=\fint_{B_1^{\R^d}(0)}|\d u(\mathscr I^{-1}(\hat v))|^2\,\d\mathcal L^d(v).
\]
Now let $e_1,\ldots,e_d\in\R^d$ be the canonical base and notice that the element $\e_i:=\mathscr I^{-1}(\hat e_i)$, $i=1,\ldots,d$, form a Hilbert base of $L^0(T\X)$ (by Theorem \ref{thm:isotang}) and that for $v=(v_1,\ldots,v_d)$ we have $\mathscr I^{-1}(\hat v)=\sum_{i=1}^dv_i\e_i$. Therefore taking into account the defining identity \eqref{eq:defs2op}  we see that
\[
S_2^2(\d u)=\fint_{B_1^{\R^d}(0)}|\d u(\mathscr I^{-1}(\hat v))|^2\,\d\mathcal L^d(v)
\]
and thus the conclusion follows from Theorem \ref{thm:exlim} and in particular by \eqref{eq:endensform}.
\end{proof}
\begin{remark}[Links with the directional energy]{\rm
In \cite{GT18} we showed how to adapt the notion of \emph{directional energy} - introduced in \cite{KS93} - to the case of maps from an $\RCD(K,N)$ space to a complete space, in particular defining the directional space $\ks^2_Z(\X,\Y_{\bar y})$, where $Z$ is a \emph{regular vector field} on $\X$ (see \cite{GT18} for the definitions). One of the things that we proved is that if $u\in W^{1,2}(\X,\Y_{\bar y})$, then for every regular vector field $Z$ we also have $u\in \ks^2_Z(\X,\Y_{\bar y})$ and the directional energy  density $\e_{2,Z}[u]$ is given by $|\d u(Z)|$, where $\d u$ is defined as in \ref{def:diff}.

Now observe that we proved in Corollary \ref{cor:ksh} that $W^{1,2}(\X,\Y_{\bar y})=\ks^{1,2}(\X,\Y_{\bar y})$ if $\X$ is $\RCD(K,N)$ (as in this case it satisfies the assumption of such corollary), therefore we deduce that if a map $u$ belongs to the latter space, it also belongs to $ \ks^2_Z(\X,\Y_{\bar y})$ for any regular vector field $Z$ and the  inequality
\[
\e_{2,Z}[u]=|\d u(Z)|\leq |Z||\d u|\stackrel{\eqref{eq:normequiv}}\leq c(d)^{-1}|Z| S_2(\d u)
\]
is the analogue of \cite[Inequality (1.8.i)]{KS93}. Similarly, the identity \eqref{eq:repr2} is the generalization of formula \cite[Inequality (1.8.1)]{KS93}.
}\fr\end{remark}

Now that we have a link between the energy density and the differential of $u$ we can use the closure-like property of the abstract differential to obtain the lower semicontinuity of the energy:
\begin{theorem}[Lower semicontinuity of the Korevaar-Schoen energy]\label{thm:lscks}
Let $(\X,\sfd,\mm)$ be a strongly rectifiable space  such  that $W^{1,2}(\X)$ is reflexive and let $(\Y,\sfd_\Y,\bar y)$ a complete metric space. Also, let  $(u_n)\subset W^{1,2}(\X,\Y_{\bar y})$ be $L^2(\X,\Y_{\bar y})$-converging to some $u\in L^2(\X,\Y_{\bar y})$ and such that 
\begin{equation}
\label{eq:supfin}
\sup_n\int |S_2(\d u_n)|^2\,\d\mm<\infty.
\end{equation}

Then $u\in W^{1,2}(\X,\Y_{\bar y})$ as well and for any $E\subset \X$ Borel  we have
\begin{equation}
\label{eq:tesi1}
\int_ES^2_2(\d u)\,\d\mm\leq\limi_{n\to\infty}\int_ES_2^2(\d u_n)\,\d\mm.
\end{equation}
In particular, if $(\X,\sfd,\mm)$ is a strongly rectifiable space, uniformly locally doubling and supports a Poincar\'e inequality, then the functional $\E_2:L^2(\X,\Y_{\bar y})\to[0,\infty]$ (recall \eqref{eq:defeks} and \eqref{eq:repr22}) is lower semicontinuous.
\end{theorem}
\begin{proof} Let $f:\Y\to\R$ be 1-Lipschitz with $f(\bar y)=0$. Then the very definition of $W^{1,2}(\X,\Y_{\bar y})$ ensures that $f\circ u_n\in W^{1,2}(\X)$ for every $n\in\N$ with 
\begin{equation}
\label{eq:normdop}
|\d (f\circ u_n)|\leq |\d u_n|\stackrel{\eqref{eq:normequiv}}\leq c(d)^{-1} S_2(\d u_n)\qquad\mm-a.e..
\end{equation}
In particular, by  our assumption \eqref{eq:supfin} up to pass to a non-relabeled subsequence we can assume that $(|\d u_n|)$ has a weak $L^2$-limit $G$. Similarly, taking into account that the $L^2(\X,\Y_{\bar y})$-convergence of $(u_n)$ to $u$ trivially yields that $f\circ u_n\to f\circ u$ in $L^2(\X)$,  we see that $\sup_n\|  f\circ u_n\|_{W^{1,2}}<\infty$. From the reflexivity of $W^{1,2}(\X)$ we deduce that $(f\circ u_n)$ is weakly relatively compact and this fact together with $L^2$-convergence grants weak $W^{1,2}$-convergence of $(f\circ u_n)$ to $f\circ u$. 

For $V\in L^2(T\X)$ the linear operator $f\mapsto \int \d f(V)\,\d\mm$ is continuous on $W^{1,2}(\X)$ and thus weakly continuous. Picking $V:=g \vv$ with $\vv\in L^\infty(T\X)$ fixed and $g\in L^2(\X)$ arbitrary we see that $\d(f\circ u_n)(\vv)\weakto \d (f\circ u)(\vv)$ in $L^2(\X)$. Recalling that \eqref{eq:iddif} grants that $\d (f\circ u_n)(\vv)\leq |\d u_n(\vv)|$ $\mm$-a.e.\ and letting $g$ be an arbitrary weak $L^2$-limit of some subsequence of $(|\d u_n(\vv)|)$ (the fact that this sequence is bounded in $L^2(\X)$ follows from \eqref{eq:normdop} and \eqref{eq:supfin}), we see that $g\leq |\vv|G$ $\mm$-a.e.. Thus letting $G_{\vv}$ be the essential liminf of all the weak $L^2$-limits of some subsequence of $(|\d u_n(\vv)|)$, we see that
\begin{equation}
\label{eq:intermedio}
\d(f\circ u)(\vv)\leq G_\vv\leq |\vv|G\qquad\mm-a.e..
\end{equation}
In particular, from the arbitrariness of $\vv$ we deduce that $|\d(f\circ u)|\leq G$ and then the arbitrariness of $f$ yields that $u\in W^{1,2}(\X,\Y_{\bar y})$.  In particular, $\d u$ is well defined and from the arbitrariness of $f$ in \eqref{eq:intermedio}  and \eqref{eq:iddif} again we conclude
\begin{equation}
\label{eq:dug}
|\d u(\vv)|\leq G_\vv\qquad\mm-a.e..
\end{equation}
Now fix a Hilbert base $\e_1,\ldots,\e_d$ of $\L^0(T\X)$ (the fact that it exists follows from Theorem \ref{thm:isotang}) and for $z\in\R^d$ put $\hat z:=\sum_iz_i\e_i\in L^\infty(T\X)$. Also, notice that if  $E\subset\X$ is Borel and  $g_n\weakto g$ in $L^2(\X)$, then we also have $\nchi_Eg_n\weakto \nchi_Eg$ in $L^2(\X)$ and therefore $\int_E g^2\,\d\mm\leq\limi_{n\to\infty}\int_Eg_n^2\,\d\mm$. Keeping this and the definition of $G_\vv$ in mind and applying Fatou's lemma we obtain 
\[
\begin{split}
\int_E\fint_{B_1^{\R^d}(0)}|\d u(\hat z)|^2\,\d\mathcal L^d(z)\,\d\mm&\stackrel{\eqref{eq:dug}}\leq\fint_{B_1^{\R^d}(0)}\int_E G_{\hat z}^2\,\d\mm\,\d\mathcal L^d(z)\\
&\stackrel{\phantom{\eqref{eq:dug}}} \leq\fint_{B_1^{\R^d}(0)}\limi_{n\to\infty }\int_E |\d u_n(\hat z)|^2\,\d\mm\,\d\mathcal L^d(z)\\
&\stackrel{\phantom{\eqref{eq:dug}}}\leq \limi_{n\to\infty }\int_E\fint_{B_1^{\R^d}(0)} |\d u_n(\hat z)|^2\,\d\mathcal L^d(z)\,\d\mm,
\end{split}
\]
which is \eqref{eq:tesi1}. For the second part of the statement recall that the doubling assumption implies that $W^{1,2}(\X)$ is reflexive (by point $(v)$ in Theorem \ref{thm:basesob}), so that the claim follows from what already proved, Corollary \ref{cor:ksh} and Theorem \ref{thm:repr}. 
\end{proof}
\begin{remark}[Non-linear Dirichlet forms]{\rm
We notice that if $(\X,\sfd,\mm)$ is strongly rectifiable, uniformly locally doubling and supports a Poincar\'e inequality, then we have just proved that the Korevaar-Schoen energy $\E_2$ is  a \emph{non-linear Dirichlet form} as axiomatized by Jost in \cite{Jost98} (see also \cite{Jost1997}).

Indeed:
\begin{itemize}
\item[i)] The quadratic contraction property
\[
\E_2(\varphi\circ u)\leq \Lip^2(\varphi)\E_2(u)
\]
for $u\in L^2(\X,\Y_{\bar y})$ and $\varphi:(\Y,\sfd_\Y,\bar y)\to(\Z,\sfd_\Z,\bar z)$ with $\varphi(\bar y)=\bar z$ is a direct consequence of the definition \eqref{eq:kse} of Korevaar-Schoen energy density at scale $r$ and of Theorem \ref{thm:exlim}. Notice that in \cite{Jost98} the function $\varphi$ is defined only on $u(\X)$ but one can always reduce to the case of $\varphi$ defined on the whole $\Y$ without altering the global Lipschitz constant, by suitably enlarging the target space $\Z$ (an operation which does not affect the value of the energy).
\item[ii)] The density of $\ks^{1,2}(\X,\R)=W^{1,2}(\X,\R)$ in $L^2(\X)$ follows noticing that $\Lip_{bs}(\X)\subset W^{1,2}(\X)$ and that $\Lip_{bs}(\X)$ is dense in $L^2(\X)$.
\item[iii)] The $L^2$-lower semicontinuity of $\E_2$ has been just proved in Theorem \ref{thm:lscks}.
\end{itemize}
\ }\fr\end{remark}

\subsection{Consistency in the case $\Y=\R$}
We already know from Corollary \ref{cor:ksh}  that if $(\X,\sfd,\mm)$ is uniformly locally doubling and supports a Poincar\'e inequality  we have $W^{1,2}(\X,\Y_{\bar y})=\ks^{1,2}(\X,\Y_{\bar y})$ as sets. We have also seen in Theorem \ref{thm:repr} (recall also the bounds \eqref{eq:normequiv}) that if $\X$ is also strongly rectifiable, then the corresponding notions of `energy density' are comparable via universal constants depending only on the dimension of $\X$. In general we cannot expect more than this, because the energy in $W^{1,2}(\X,\Y_{\bar y})$ is related to the operator norm of the differential, while that in $\ks^{1,2}(\X,\Y_{\bar y})$ by its 2-size (which as said can be seen as a generalization of the Hilbert-Schmidt norm - see also Lemma \ref{le:normhs}) and it is easy to see that for a linear map from $\R^d$ with values in some Banach space, in general we cannot say anything better than \eqref{eq:normequiv0} for what concerns the relation between the operator norm and the 2-size.

Yet, there is a particular and relevant case when these two quantities coincide, up to a multiplicative constant: this occurs if the target Banach space is $\R$, as shown in the following lemma.
\begin{lemma}\label{le:easy}
For any $d\in\N$   there is a constant $c(d)>0$ such that the following holds. Let $\ell:\R^d\to\R$ be linear. Then
\begin{equation}
\label{eq:samenorm}
\|\ell\|_{op}=c(d)S_2(\ell),
\end{equation}
where $\|\ell\|_{op}$ is the operator norm defined as $\sup_{|v|\leq1}|\ell(v)|$.

In particular, for any $L^0(\mm)$-linear and continuous map $T$ from  a Hilbert $L^0(\mm)$-module of dimension $d$ to a $L^0(\mm)$-module of local dimension bounded above by 1 we have
\begin{equation}
\label{eq:perr}
|T|=c(d)S_2(T)\qquad\mm-a.e..
\end{equation}
\end{lemma}
\begin{proof}
Both sides of \eqref{eq:samenorm} are positively 1-homogeneous and remain unchanged if we replace $\ell$ by $\ell\circ O$, with $O\in O(d)$. Since any two non-zero linear maps from $\R^d$ to $\R$ can be transformed one into the other via these transformations, we see that the ratio $\frac{S_2(\ell)}{\|\ell\|}$ does not depend on the particular non-zero $\ell$ chosen. The case $\ell=0$ follows as well because in this case \eqref{eq:samenorm} holds for any value of the constant $c(d)$.

For the second part of the statement we notice that on the set where the target module has dimension 0, the map $T$ must be 0 and thus the conclusion is trivially true. On the set where the dimension is 1 we use what we previously proved in conjunction with  \eqref{eq:defs2op} and \eqref{eq:normaopbase}.
\end{proof}
From this simple lemma we deduce the following consistency result, in line with the analogous one in \cite{KS93}:
\begin{proposition}[Consistency in the case $\Y=\R$]\label{prop:consR}
Let  $(\X,\sfd,\mm)$ be a strongly rectifiable space with uniformly locally doubling measure and supporting a Poincar\'e inequality. Then $W^{1,2}(\X,\R)=\ks^{1,2}(\X,\R)$ and for any function $u$ in these spaces it holds
\begin{equation}
\label{eq:samediff}
|\d u|=c(d)\e_2[u]\qquad\mm-a.e.,
\end{equation}
where $d$ is the dimension of $\X$ and  $c(d)$ is given by Lemma \ref{le:easy} above.
\end{proposition}
\begin{proof}
By Theorem \ref{thm:tangrd} that the tangent module of $\R$ equipped with the measure $\mu:=u_*(|\d u|^2\mm)$ has dimension bounded above by 1, hence the same holds for ${\rm Ext}u^*L^0(T_\mu\R)$. Also, by Theorem \ref{thm:dimcot} we know that $L^0(T\X)$ has dimension $d$ and thus we are in position to apply the above Lemma to deduce that  identity \eqref{eq:perr} holds for  $T:=\d u$. Then we conclude by the reproducing formula \eqref{eq:repr2}.
\end{proof}

\section{Maps defined on open sets}

\subsection{The spaces $W^{1,2}(\Omega)$, $W^{1,2}_0(\Omega)$ and $W^{1,2}(\Omega,\Y_{\bar y})$}
Here we recall the definition of the Sobolev spaces $W^{1,2}(\Omega)$ and $W^{1,2}_0(\Omega)$ of real valued Sobolev functions defined on an open subset $\Omega$ of a metric measure space and the related one $W^{1,2}(\Omega,\Y_{\bar y})$.

The definition of $W^{1,2}(\Omega)$ is based on the observation that if $f\in W^{1,2}(\X)$ and $\eta\in \Lip_{bs}(\X)$, then a simple application of the Leibniz rule shows that $\eta f\in W^{1,2}(\X)$ as well and by the locality property of the differential we also have $\d(\eta f)=\d f$ $\mm$-a.e.\ on $\{\eta=1\}$. 

We shall denote by $L^2_{loc}(\Omega)$ the space of real valued Borel functions on $\Omega$ which belong to $L^2(C)$ for every bounded closed set $C\subset\Omega$. We then give the following
\begin{definition}[The spaces $W^{1,2}(\Omega)$ and $W^{1,2}_0(\Omega)$]\label{def:sobopen} Let $(\X,\sfd,\mm)$ be a metric measure space and $\Omega\subset\X$ open.
The space $W^{1,2}_{loc}(\Omega)$ is the subset of $L^2_{loc}(\Omega)$ made of those functions $f$ such that $\eta f\in W^{1,2}(\X)$ for every $\eta\in \Lip_{bs}(\X)$ with $\supp(\eta)\subset\Omega$ (here $\eta f$ is intended to be 0 outside $\Omega$). For $f\in W^{1,2}_{loc}(\Omega)$ we define $\d f\in L^0(T\X)$ to be 0 outside $\Omega$ and via the formula
\[
\d f:=\d (\eta f)\quad\mm-a.e.\ on\ \{\eta=1\}\qquad\text{ for every $\eta\in \Lip_{bs}(\X)$ with $\supp(\eta)\subset\Omega$}
\]
inside $\Omega$.

We then define $W^{1,2}(\Omega)\subset W^{1,2}_{loc}(\Omega)$ as the collection of those $f\in W^{1,2}_{loc}(\Omega)$ such that $f,|\d f|\in L^2(\Omega)$ and equip it with the norm $\|f\|_{W^{1,2}(\Omega)}:=\sqrt{\|f\|_{L^2(\Omega)}^2+\||\d f|\|_{L^2(\Omega)}^2}$.

Finally, the space $W^{1,2}_0(\Omega)\subset W^{1,2}(\Omega)$ is defined as the $W^{1,2}(\Omega)$-closure of the space of functions $f\in W^{1,2}(\Omega)$ with $\supp(f)\subset\Omega$.
\end{definition}
We remark that the definition of $\d f$ for $f\in W^{1,2}_{loc}(\Omega)$ is well posed thanks to the locality property of the differential and the fact that $\Omega$ is open, which ensures that there are $\{\eta_n\}_{n\in\N}\subset\Lip_{bs}(\X)$ with support in $\Omega$ and such that $\cup_n\{\eta_n=1\}=\Omega$. 

Also, we point out that $W^{1,2}_0(\Omega)$ could be equivalently defined as the closure in $W^{1,2}(\X)$, rather than $W^{1,2}(\Omega)$, of those functions  $f\in W^{1,2}(\X)$ with $\supp(f)\subset\Omega$.

For later use we record here the following simple property of $W^{1,2}_0(\Omega)$:
\begin{proposition}\label{prop:order}
Let $(\X,\sfd,\mm)$ be a metric measure space, $\Omega\subset\X$ open, $f\in W^{1,2}_0(\Omega)$ and $g\in W^{1,2}(\Omega)$ be such that $0\leq g\leq f$ $\mm$-a.e.\ on $\Omega$.

Then $g\in W^{1,2}_0(\Omega)$.
\end{proposition}
\begin{proof}
Let $(f_n)\subset W^{1,2}(\X)$ with $\supp(f_n)\subset\Omega$ be $W^{1,2}$-converging to $f$. Then the maps $g_n:=g\wedge f_n$ also belong to $W^{1,2}(\X)$ and have support contained in $\Omega$, so that to conclude it is sufficient to show that $g_n\to g$ in $W^{1,2}(\Omega)$. Convergence in $L^2(\Omega)$ is obvious. Then notice that from the locality of minimal weak upper gradients and the fact that if $g_n<g$ then $g_n=f_n$ we have
\[
|D(g-g_n)|=\nchi_{\{g>g_n\}}|D(g-g_n)|\leq \nchi_{\{f>g\}\cap\{g>g_n\}}|D(g-f)|+|D(f-f_n)|%\nchi_{\{g>g_n\}}
%\leq \nchi_{\{g>g_n\}}\big(|D g|+|Dg_n|\big)\leq\nchi_{\{g>g_n\}}\big(|D g|+|Df|+|D(f-f_n)|\big)
\]
and therefore
\[
\||D(g-g_n)|\|_{L^2}\leq\|\nchi_{\{f>g\}\cap\{g>g_n\}}|D(g-f)|\|_{L^2}+\||D(f-f_n)|\|_{L^2}\to 0,
\]
having used the fact that  $\mm({\{f>g\}\cap\{g>g_n\}})\to 0$ as $n\to\infty$ and the absolute continuity of the integral.
\end{proof}
Now, given a pointed complete space $(\Y,\sfd_\Y,\bar y)$ the definition of $W^{1,2}(\Omega,\Y_{\bar y})$ can naturally be given by imitating the analogue one \ref{def:sobmetr}:
\begin{definition}[The space $W^{1,2}(\Omega,\Y_{\bar y})$]\label{def:w12o}Let $(\X,\sfd,\mm)$ be a metric measure space, $\Omega\subset\X$ open and $(\Y,\sfd_\Y,\bar y)$ a pointed complete space. 
The space $W^{1,2}(\Omega,\Y_{\bar y})$ is the collection of all the maps $u\in L^2(\Omega,\Y_{\bar y})$ for which there is $G\in L^2(\Omega)$ such that for any $f:\Y\to\R$ 1-Lipschitz we have $f\circ u\in W^{1,2}(\Omega)$ with $|D(f\circ u)|\leq G$ $\mm$-a.e.\ on $\Omega$. 

The least, in the $\mm$-a.e.\ sense, function $G$ for which the above holds is denoted by $|Du|$.
\end{definition}

\subsection{The space $\ks^{1,p}(\Omega,\Y_{\bar y})$}

In this  section we see how to adapt the theory discussed so far to the case of metric valued functions defined only on an open subset $\Omega$ of the space $\X$.

Let us start recalling the definition as given in \cite{KS93}:
\begin{definition}\label{def:ksloc}
Let $(\X,\sfd,\mm)$ be a metric measure space, $(\Y,\sfd_\Y,\bar y)$ a pointed complete space, $\Omega\subset\X$ open and $u\in L^2(\Omega,\Y_{\bar y})$.

Then for every $r>0$ we define $\kse_{2,r}[u,\Omega]:\Omega\to[0,\infty]$ as
\[
\kse_{2,r}[u,\Omega](x):=%\Big|\fint_{B_r(x)} \frac{\sfd_\Y(u(x),u(y))^p}{r^p}\,\d \mm(y) \Big|^{1/p}
\left\{\begin{array}{ll}
\displaystyle{\Big|\fint_{B_r(x)} \frac{\sfd^2_\Y(u(x),u(y))}{r^2}\,\d \mm(y) \Big|^{1/2}}&\qquad\text{if }B_r(x)\subset\Omega,\\
0&\qquad\text{otherwise}
\end{array}
\right.
\]
and say that $u\in\ks^{1,2}(\Omega,\Y_{\bar y})$ provided
\begin{equation}
\label{eq:defkso}
\E_2^\Omega(u):=\sup\lims_{r\downarrow0}\int_\Omega \varphi\,\kse^2_{2,r}[u,\Omega]\,\d\mm<\infty,
\end{equation}
where the $\sup$ is taken among all $\varphi:\X\to[0,1]$ continuous and such that $\supp(\varphi)$ is compact and contained in $\Omega$.
\end{definition}
\begin{remark}{\rm
In Definition \ref{def:ksloc} above we opted for the same choice made in \cite{KS93} to consider the $\lims$ in the defining formula \eqref{eq:defkso}. On the other hand,  in the defining formula \eqref{eq:defeks} we preferred the $\limi$ and thus for internal consistency it would perhaps have been preferable to use the $\limi$ also in \eqref{eq:defkso}. The point, however, is that the choice made is in fact irrelevant: following the arguments in the proof of Proposition \ref{prop:perloc} below one can easily check that the energy defined with the $\limi$ is finite if and only if so is the one defined by the $\lims$. In fact, taking into account the results in Theorem \ref{thm:ksomega} one can also see that the limit in \eqref{eq:defkso} exists.
}\fr\end{remark}
The following proposition provides an alternative description of functions in $\ks^{1,2}(\Omega,\Y_{\bar y})$ which is conceptually closer to Definition \ref{def:sobopen}:
\begin{proposition}\label{prop:perloc}
Let $(\X,\sfd,\mm)$ be uniformly locally doubling, supporting a Poincar\'e inequality and strongly rectifiable of dimension $d$, $\Omega\subset\X$ open, $(\Y,\sfd_\Y,{\bar y})$ a pointed complete space and $\iota:\Y\to\ell^\infty(\Y)$  the associated Kuratowski embedding (recall Lemma \ref{le:kur}).

Then a map $u:\Omega\to\Y$ belongs to $\ks^{1,2}(\Omega,\Y_{\bar y})$ if and only if the following two conditions hold:
\begin{itemize}
\item[i)] for every $K\subset\Omega$ compact there is $u_K\in\ks^{1,2}(\X,\ell^\infty(\Y))$ such that
\[
u_K=\iota\circ u\qquad\mm-a.e.\ on\ K.
\]
\item[ii)] The function $\e_2[u]:\Omega\to[0,\infty]$ defined by
\begin{equation}
\label{eq:defenloc}
\e_2[u]:=\e_2[u_K]\qquad\mm-a.e.\ on\ K,
\end{equation}
which is well defined thanks to Corollary \ref{cor:loced}, belongs to $L^2(\Omega)$.
\end{itemize}
Moreover, if these hold the maps $u_K$ can be chosen to satisfy
\begin{equation}
\label{eq:uku}
\E_2(u_K)\leq c\Big(\E_2^\Omega(u)+\sfd(K,\Omega^c)^{-2}\int_\Omega\sfd_\Y^2(u(x),\bar y)\,\d\mm(x)\Big),
\end{equation}
where $c$ is a universal constant (we will pick $c=25$) and  $\sfd(K,\Omega^c):=\inf_{x\in K\atop y\in\Omega^c}\sfd(x,y)$.
\end{proposition}
\begin{proof}\ \\
\noindent{\bf If} Fix $\varphi:\X\to[0,1]$ continuous and such that $\supp(\varphi)$ is compact and contained in $\Omega$. Then for every $r>0$ the set $K_r:=\{x\in\X:\sfd(x,\supp(\varphi))\leq r\}$  is compact (because $\X$, being complete and locally doubling, is proper) and for $r$ sufficiently small also contained in $\Omega$. Fix such $\bar r$ and notice that for any $x\in\supp(\varphi)$ and $r\in(0,\bar r)$ we have $\kse_{2,r}[u,\Omega](x)=\kse_{2,r}[u_{K_{\bar r}}](x)$.  Therefore recalling Theorem \ref{thm:exlim} to pass to the limit we deduce that
\[
\lims_{r\downarrow0}\int_\Omega \varphi\,\kse^2_{2,r}[u,\Omega]\,\d\mm=\lims_{r\downarrow0}\int \varphi\,\kse^2_{2,r}[u_{K_{\bar r}}]\,\d\mm=\int\varphi\, \e^2_2[u_{K_{\bar r}}]\,\d\mm\leq\int_{\Omega}  \e^2_2[u]\,\d\mm
\]
and the conclusion \eqref{eq:defkso} follows.

\noindent{\bf Only if} Fix $K\subset\Omega$ compact, for $r>0$ put $K_r:=\{x\in\X:\sfd(x,K)\leq r\}$ and put $\bar r:=\sfd(K,\Omega^c)/5>0$, so  that $K_{4\bar r}\subset\Omega$. Also, define $\eta:\X\to[0,1]$ as $\eta:=(1-\bar r^{-1}\sfd(\cdot,K_{\bar r}))^+$, so that $\Lip(\eta)\leq 5\sfd(K,\Omega^c)^{-1}$, and $\eta$ is identically 1 on $K_{\bar r}$ and 0 outside $K_{2\bar r}$. Put $u_K:=\eta\,\iota\circ u$, where it is intended that this function is identically 0 outside $\Omega$. Then from the trivial inequality
\[
\begin{split}
\sfd_{\ell^\infty(\Y)}\big(u_K(y),u_K(x)\big)&\leq \sfd_{\ell^\infty(\Y)}\big(\eta(y)\iota(u(y)),\eta(y)\iota(u(x))\big)+\sfd_{\ell^\infty(\Y)}\big(\eta(y)\iota(u(x)),\eta(x)\iota(u(x))\big)\\
&\leq \sfd_\Y(u(y),u(x))+\|\iota(u(x))\|_{\ell^\infty(\Y)}|\eta(y)-\eta(x)|,
\end{split}
\]
valid for any $x,y\in\Omega$ and the triangle inequality in $L^2(\X)$, we deduce that for any $r\in(0,\bar r)$ it holds
\[
\kse_{2,r}[u_K](x)\leq\left\{\begin{array}{ll}
\kse_{2,r}[u,\Omega]+\Lip(\eta)\sfd_\Y(u(x),\bar y),&\qquad\text{ if }x\in K_{3\bar r},\\
0,&\qquad\text{ otherwise}.
\end{array}\right.
\]
Now let $\varphi:\X\to[0,1]$ be continuous and such that $\supp(\varphi)$ is compact and contained in $\Omega$ and identically 1 on $K_{3\bar r}$. Then the above inequality ensures that
\[
\lims_{r\downarrow0}\int \kse^2_{2,r}[u_K]\,\d\mm\leq2\lims_{r\downarrow0}\int_\Omega\varphi \kse_{2,r}^2[u,\Omega]\,\d\mm+2\Lip^2(\eta)\sfd^2_{L^2(\Omega,\Y)}(u,\bar y)<\infty
\]
and thus point $(i)$ and the estimate \eqref{eq:uku} hold. To see that point $(ii)$ holds as well notice that $u_K=\iota\circ u$ on $K_{\bar r}$ and thus for any $r\in(0,\bar r)$ we have
\[
\kse_{2,r}[u_K](x)=\kse_{2,r}[u,\Omega](x)\qquad\forall x\in K.
\]
Therefore it holds
\begin{equation}
\label{eq:perconc}
\int_K\e_2^2[u_K]\,\d\mm=\lim_{r\downarrow0}\int_K\kse^2_{2,r}[u_K]\,\d\mm=\lim_{r\downarrow0}\int_K\kse^2_{2,r}[u,\Omega]\,\d\mm\leq \lims_{r\downarrow0}\int_\Omega \varphi\,\kse^2_{2,r}[u,\Omega]\,\d\mm
\end{equation}
and thus
\[
\int_\Omega\e_2^2[u]\,\d\mm=\sup_{K\subset\subset\Omega}\int_K\e_2^2[u]\,\d\mm\stackrel{\eqref{eq:defenloc}}=\sup_{K\subset\subset\Omega}\int_K\e_2^2[u_K]\,\d\mm\stackrel{\eqref{eq:perconc}}\leq \sup\lims_{r\downarrow0}\int_\Omega \varphi\,\kse^2_{2,r}[u,\Omega]\,\d\mm<\infty
\]
as desired, where  the last $\sup$ is taken among all $\varphi$'s as in Definition \ref{def:ksloc}. 
\end{proof}
The next result collects the main properties of functions in $\ks^{1,2}(\Omega,\Y_{\bar y})$.
\begin{theorem}\label{thm:ksomega} Let $(\X,\sfd,\mm)$ be locally uniformly doubling, supporting a Poincar\'e inequality and strongly rectifiable, $\Omega\subset\X$ open and $(\Y,\sfd_\Y,\bar y)$ a pointed and complete space.

Then the following hold:
\begin{itemize}
\item[i)] $\ks^{1,2}(\Omega,\Y_{\bar y})=W^{1,2}(\Omega,\Y_{\bar y})$ as sets,
\item[ii)] for any $u\in \ks^{1,2}(\Omega,\Y_{\bar y})$ we have
\[
\kse_{2,r}[u,\Omega]\quad\to\quad \e_2[u]\qquad\text{ $\mm$-a.e.\ and in $L^2_{loc}(\Omega)$ as $r\downarrow0$}
\]
where $\e_2[u]$ is given by \eqref{eq:defenloc}.
\item[iii)] Any $u\in\ks^{1,2}(\Omega,\Y_{\bar y})$ is approximately metrically differentiable $\mm$-a.e.\ in $\Omega$ (here we extend $u$ on the whole $\X$ declaring it to be constant outside $\Omega$ to apply the definition of approximate metric differentiability) and it holds
\begin{equation}
\label{eq:endenso}
\e_2[u](x)=S_2(\md_x(u))=S_2(\d u)(x)\qquad\mm-a.e.\ x\in\Omega.
\end{equation}
\item[iv)] The functional $\E^\Omega_2:L^2(\Omega,\Y_{\bar y})\to[0,+\infty]$ defined by \eqref{eq:defkso} is lower semicontinuous and can be written as
\[
\E_2^\Omega(u)
:=\left\{\begin{array}{ll}
\displaystyle{\int_\Omega\e_2^2[u]\,\d\mm},&\qquad\text{ if }u\in\ks^{1,2}(\Omega,\Y_{\bar y}),\\
+\infty,&\qquad\text{ otherwise}.
\end{array}
\right.
\]
\end{itemize}
\end{theorem}
\begin{proof}\ \\
\noindent{\bf (i)} Let  $\iota:\Y\to\ell^\infty(\Y)$ be the   Kuratowski embedding. 

Let $u\in\ks^{1,2}(\Omega,\Y_{\bar y})$, $f:\Y\to\R$ 1-Lipschitz and $\hat f:\ell^\infty(\Y)\to\R$ 1-Lipschitz and such that $ f=\hat f\circ \iota$ (recall Lemma \ref{le:ext}). Also, let $\eta:\X\to[0,1]$ be Lipschitz and with support compact and contained in $\Omega$. Then with the notation of Proposition \ref{prop:perloc} above we have that $\eta (f\circ u)=\eta(\hat f\circ u_{\supp(\eta)})$ and since $u_{\supp(\eta)}\in \ks^{1,p}(\X,\ell^\infty(\Y))=W^{1,p}(\X,\ell^\infty(\Y))$ by Corollary \ref{cor:ksh}, we see that the function $\eta (f\circ u)$, intended to be 0 outside $\Omega$, is in $W^{1,2}(\X)$. Moreover, putting for brevity $K:=\supp(\eta)$  and $c:=c(d)^{-1}$ (recall \eqref{eq:normequiv}), $\mm$-a.e.\ on $\{\eta=1\}\subset K$ we have
\[
|\d (\eta (f\circ u))|=|\d(\eta(\hat f\circ u_{K}))|=|\d(\hat f\circ u_K)|\leq |\d u_K|\stackrel{\eqref{eq:normequiv}}\leq  c\, S_2(\d u_K)\stackrel{\eqref{eq:repr2}}=c\,\e_2[u_K]\stackrel{\eqref{eq:defenloc}}=c\,\e_2[u].
\] 
By the arbitrariness of $\eta,f$, the very definition of $W^{1,2}(\Omega,\Y_{\bar y})$ and point $(ii)$ in Proposition \ref{prop:perloc} this grants that $u\in W^{1,2}(\Omega,\Y_{\bar y})$.

For the converse inclusion let $u\in W^{1,2}(\Omega,\Y_{\bar y})$ and recall that by point $(ii)$ of Proposition \ref{prop:basesobmet}  this is the same as to say that $\iota\circ u\in W^{1,2}(\Omega,\ell^\infty(\Y))$. Fix $K\subset\subset\Omega$ and let $\eta:\X\to[0,1]$ be Lipschitz, identically 1 on $K$ and with support contained in $\Omega$. Then from the very definition of $W^{1,2}(\Omega,\Y_{\bar y})$ and Lemma \ref{le:metcutoff}  we deduce that $u_K:=\eta \,\iota\circ u$, intended to be 0 outside $\Omega$, belongs to $W^{1,2}(\X,\ell^\infty(\Y))$ and thus, by Corollary \ref{cor:ksh}, to  $\ks^{1,2}(\X,\ell^\infty(\Y))$. Then taking into account the locality of minimal weak upper gradients and energy densities we obtain
\[
\e_2[u_K]=\e_2[\eta\,\iota\circ u]\stackrel{\eqref{eq:repr2},\eqref{eq:normequiv},\eqref{eq:samenorm2}}\leq |D(\eta\,\iota\circ u)|=|D(\iota\circ u)|=|D u|\qquad\mm-a.e.\ on\ K
\]
so that the conclusion follows from Proposition \ref{prop:perloc} and Definition \ref{def:w12o}.

\noindent{\bf (ii)} We need to prove that for any $K\subset\subset \Omega$ we have $\kse_{2,r}[u,\Omega]\to\e_2[u]$ in $L^2(K)$ and $\mm$-a.e.. To see this, use Proposition \ref{prop:perloc} to find $\tilde u\in \ks^{1,2}(\X,\ell^\infty(\Y))$ equal to $\iota\circ u$ in a compact neighbourhood of $K$. Then   for $r\ll1$ it holds $\kse_{2,r}[\tilde u]=\kse_{2,r}[u,\Omega]$ on $K$ and the conclusion follows from Theorem \ref{thm:exlim} and the very definition of $\e_2[u]$ given by \eqref{eq:defenloc}.

\noindent{\bf (iii)} Since we set $u$ to be constant outside $\Omega$, its differentiability outside $\Omega$ is trivial. Now let $K\subset\subset \Omega$ and $u_K$ as in  Proposition \ref{prop:perloc}.  Then by Proposition \ref{prop:metrdiff}, Theorem \ref{thm:exlim} and Theorem \ref{thm:repr} we know that $ u_K$ is approximately metrically differentiable and that \eqref{eq:endenso} holds for $ u_K$. Then the fact that $\mm$-a.e.\ point in $K$ is a density point and the very definition of approximate metric differentiability give that $u$ is approximately metrically differentiable $\mm$-a.e.\ in $K$ with $\md_x(u)=\md_{x}(u_K)$ $\mm$-a.e.\ in $K$. Similarly, from the locality of the differential it is easy to see that $S_2(\d u_K)=S_2(\d u)$ $\mm$-a.e.\ in $K$ so that recalling also the definition \eqref{eq:defenloc} we conclude.

Therefore the validity of \eqref{eq:endenso} $\mm$-a.e.\ on $K$ follows from the definition  \eqref{eq:defenloc} and the conclusion follows from the fact that we can write $\Omega$ as countable union of compact subsets.

\noindent{\bf (iv)} Let $(u_n)\subset L^2(\Omega,\Y_{\bar y})$ be converging to some $u$ in $L^2(\Omega,\Y_{\bar y})$ and with $\sup_n\E_2^\Omega(u_n)<\infty$. Fix  $K\subset\Omega$ compact, find a compact neighbourhood $\tilde K\subset\Omega$ of $K$ and apply Proposition \ref{prop:perloc} to $\tilde K$ to obtain functions $u_{\tilde K,n}\in L^2(\X,\ell^\infty(\Y))$ satisfying points $(i),(ii)$ of such proposition and the estimate \eqref{eq:uku} with $u_n$ in place of $u$. In particular, we have $\sup_n\E_2(u_{\tilde K,n})<\infty$ and $u_{\tilde K,n}\to \iota\circ u$ in $L^2(\tilde K,\ell^\infty(\Y))$.

Let $\eta:\X\to[0,1]$ be Lipschitz, identically 1 on $K$ and with support contained in $\tilde K$. Then from Corollary \ref{cor:ksh}, Lemma \ref{le:metcutoff}, the identity \eqref{eq:samenorm2}, the representation formula in Theorem \ref{thm:repr} and the bounds \eqref{eq:normequiv} we obtain that $\eta u_{\tilde K,n}\in \ks^{1,2}(\X,\ell^\infty(\Y))$ and  $\sup_n\E_2(\eta u_{\tilde K,n})<\infty$.

Since by construction we also have $\eta u_{\tilde K,n}\to u_K:=\eta \iota\circ u$ in $L^2(\X,\ell^\infty(\Y))$, we are in position to apply the first part of Theorem \ref{thm:lscks} with $E:=K$ and deduce that $u_K\in \ks^{1,2}(\X,\Y_{\bar y})$ and that
\begin{equation}
\label{eq:unifK}
\int_K\e_2^2[u_K]\,\d\mm\leq\limi_{n\to\infty}\int_K\e_2^2[u_{\tilde K,n}]\,\d\mm\stackrel{\eqref{eq:defenloc}}=\limi_{n\to\infty}\int_K\e_2^2[u_n]\,\d\mm\leq \limi_{n\to\infty}\int_\Omega\e_2^2[u_n]\,\d\mm.
\end{equation}
Since we have $u_K=\iota\circ u$ on $K$, the arbitrariness of $K$ and the uniform bound \eqref{eq:unifK} allow to apply Proposition \ref{prop:perloc} and deduce that $u\in \ks^{1,2}(\Omega,\Y)$. To conclude notice that
\[
\E_2^\Omega(u)=\sup_{K\subset\subset\Omega}\int_K\e_2^2[u]\,\d\mm\stackrel{\eqref{eq:defenloc}}=\sup_{K\subset\subset\Omega}\int_K\e_2^2[u_K]\,\d\mm\stackrel{\eqref{eq:unifK}}\leq \limi_{n\to\infty}\int_\Omega\e_2^2[u_n]\,\d\mm=\limi_{n\to\infty}\E_2^\Omega(u_n)
\]
\end{proof}

\subsection{Assigning a value at the boundary}
In this section we  introduce the space $\ks^{1,2}_{\bar u}(\Omega,\Y_{\bar y})\subset \ks^{1,2}(\Omega,\Y_{\bar y})$ of those maps `having the same value as $\bar u\in\ks^{1,2}(\Omega,\Y_{\bar y})$ at the boundary of $\Omega$'. This is possible regardless of the regularity of $\Omega$ thanks to the notion of $W^{1,2}_0(\Omega)$:
\begin{definition}[The space $\ks^{1,2}_{\bar u}(\Omega,\Y)$]\label{def:boundval}

Let $\bar u\in \ks^{1,2}(\Omega,\Y_{\bar y})$. Then the space $\ks^{1,2}_{\bar u}(\Omega,\Y_{\bar y})\subset \ks^{1,2}(\Omega,\Y_{\bar y})$ is defined as:
\[
\ks^{1,2}_{\bar u}(\Omega,\Y_{\bar y}):=\big\{u\in\ks^{1,2}(\Omega,\Y_{\bar y})\ :\ \sfd_\Y(u,\bar u)\in W^{1,2}_0(\Omega)\big\}.
\]
\end{definition}
We also define the associated energy functional $\E^\Omega_{2,\bar u}:L^2(\Omega,\Y_{\bar y})\to[0,+\infty]$ as
\[
\E^\Omega_{2,\bar u}(u)
:=\left\{\begin{array}{ll}
\displaystyle{\E_2^\Omega(u)=\int_\Omega \e_2^2[u]\,\d\mm},&\qquad\text{ if }u\in \ks^{1,2}_{\bar u}(\Omega,\Y_{\bar y}),\\
+\infty,&\qquad\text{ otherwise.}
\end{array}
\right.
\]

In order to understand the basic properties of $\ks^{1,2}_{\bar u}(\Omega,\Y_{\bar y})$ the following lemma will be useful:
\begin{lemma}\label{le:distw}Let  $(\X,\sfd,\mm)$ be a strongly rectifiable space with uniformly locally  doubling measure and supporting a Poincar\'e inequality, $\Omega\subset\X$ open and $(\Y,\sfd_\Y,\bar y)$ a pointed complete space.

Let $u,v\in \ks^{1,2}(\Omega,\Y_{\bar y})$. Then the map $x\mapsto \sfd_\Y(u(x),v(x))$ belongs to $W^{1,2}(\Omega)$ and
\begin{equation}
\label{eq:distw}
|D\sfd_\Y(u,v)|\leq c(d) \sqrt 2\sqrt{\e^2_2[u]+\e^2_2[v]}\qquad\mm-a.e.\ on \ \Omega,
\end{equation}
where  $c(d)$ is the constant defined in Proposition \ref{prop:consR}.
\end{lemma}
\begin{proof}
By Proposition \ref{prop:perloc} and the very definition of $W^{1,2}(\Omega)$ we see that it is sufficient to consider the case $\Omega=\X$. Thus let this be the case, equip $\Y^2$ with the distance
\[
\sfd^2_{\Y^2}\big((y_0,y_1),(y_0',y_1')\big):=\sfd^2_\Y(y_0,y_0')+\sfd^2_\Y(y_1,y_1')
\]
and notice that from the very Definition \ref{def:kss} and from Theorem \ref{thm:exlim} we have that $(u,v):\X\to\Y^2$ belongs to $\ks^{1,2}(\X,\Y^2_{(\bar y,\bar y)})$ with
\[
\e_2[(u,v)]=\sqrt{\e_2^2[u]+\e_2^2[v]}
\] 
Recalling that $\ks^{1,2}(\X,\Y^2)=W^{1,2}(\X,\Y^2)$ and that $\sfd_\Y:\Y^2\to\R$ is $\sqrt 2$-Lipschitz we see that $\sfd_\Y(u,v)\in W^{1,2}(\X)$ and, by a direct application of the definition of energy density as limit of the approximate energy densities, that $\e_2[\sfd_\Y(u,v)]\leq\sqrt 2\,\e_2[(u,v)]$. Then the bound \eqref{eq:distw}   comes from Proposition \ref{prop:consR}.
\end{proof}
We then have:
\begin{proposition}\label{prop:dw120} Let  $(\X,\sfd,\mm)$ be a strongly rectifiable space with uniformly locally doubling measure and supporting a Poincar\'e inequality, $\Omega\subset\X$ open and $(\Y,\sfd_\Y,\bar y)$ a pointed and complete space. Also, let $\bar u\in \ks^{1,2}(\Omega,\Y_{\bar y})$.  Then:
\begin{itemize}
\item[i)] $\E^\Omega_{2,\bar u}$ is lower semicontinuous.
\item[ii)] For any $u,v\in \ks^{1,2}_{\bar u}(\Omega,\Y_{\bar y})$ we have $\sfd_\Y(u,v)\in W^{1,2}_0(\Omega)$.
\end{itemize}
\end{proposition}
\begin{proof}\ \\
\noindent{\bf (i)} Let $(u_n)\subset \ks^{1,2}_{\bar u}(\Omega,\Y_{\bar y})$ be with 
\begin{equation}
\label{eq:perchiusura}
\sup_n\E_2^\Omega(u_n)<\infty
\end{equation}
 and $L^2(\Omega,\Y_{\bar y})$-converging to some $u$. By point $(iv)$ in Theorem \ref{thm:ksomega} we know that $u\in\ks^{1,2}(\Omega,\Y_{\bar y})$ and thus to conclude it is sufficient to prove that $\sfd_\Y(u,\bar u)\in W^{1,2}_0(\Omega)$. To this aim, notice that the functions $\sfd_\Y(u_n,\bar u)$, set to 0 outside $\Omega$, converge to $\sfd_\Y(u,\bar u)$ in $L^2(\X)$ as $n\to\infty$. Also, by Lemma \ref{le:distw} and our assumption \ref{eq:perchiusura} we know that $\sup_n\|\sfd_\Y(u_n,\bar u)\|_{W^{1,2}}<\infty$. Since $W^{1,2}(\X)$ is reflexive (recall point $(v)$ in Theorem \ref{thm:basesob}), bounded sequences are weakly relatively compact and in our case the $L^2$-convergence force the weak $W^{1,2}$-convergence of $(\sfd_\Y(u_n,\bar u))$ to $\sfd_\Y(u,\bar u)$. Since $\sfd_\Y(u_n,\bar u)$ belongs to the closed subspace $W^{1,2}_0(\Omega)$ of $W^{1,2}(\X)$, this proves that $\sfd_\Y(u,\bar u)\in W^{1,2}_0(\Omega)$ as well.

\noindent{\bf (ii)} Consequence of Proposition \ref{prop:order} and the trivial inequality $\sfd_\Y(u,v)\leq\sfd_\Y(\bar u,u)+\sfd_\Y(\bar u,v).$
\end{proof}

\section{The case of $\CAT(0)$ space as  target}\label{se:cattarg}
In this final section  we introduce our main assumption on the target space $\Y$ and derive, along the lines of \cite{KS93} an existence result for harmonic maps.

Recall that a curve $\gamma:[0,1]\to\Y$ is said a (constant speed) geodesic provided
\[
\sfd_\Y(\gamma_t,\gamma_s)=|s-t|\sfd_\Y(\gamma_0,\gamma_1)\qquad\forall t,s\in[0,1]
\]
and the following definition:
\begin{definition}[$\CAT(0)$ spaces]
A complete metric space $(\Y,\sfd_\Y)$ is said a $\CAT(0)$ space provided it is geodesic and for any constant speed geodesic $\gamma:[0,1]\to\Y$ and $y\in\Y$ it holds
\begin{equation}
\label{eq:defcat}
\sfd_\Y^2(y,\gamma_t)\leq (1-t)\sfd_\Y^2(y,\gamma_0)+t\sfd_\Y^2(y,\gamma_1)-t(1-t)\sfd_\Y^2(\gamma_0,\gamma_1)\qquad\forall t\in[0,1].
\end{equation}
\end{definition}
We emphasise that for us a $\CAT(0)$ space is complete, much like any other space considered in the manuscript. Other authors do not enforce this condition and refer to complete $\CAT(0)$ spaces as Hadamard spaces.

\bigskip

It can be proved (see e.g.\ \cite[Corollary 2.1.3]{KS93}) that in a $\CAT(0)$ space, for any two geodesics $\gamma,\eta$ and any $t\in[0,1]$ it holds
\begin{equation}
\label{eq:parin1}
\sfd_\Y^2(\gamma_t,\eta_t)\leq (1-t)\sfd^2_\Y(\gamma_0,\eta_0)+t\sfd_\Y^2(\gamma_1,\eta_1)-t(1-t)\big(\sfd_\Y(\gamma_0,\gamma_1)-\sfd_\Y(\eta_0,\eta_1)\big)^2
\end{equation}
%4 points $x,y,z,w$ the `parallelogram inequality'
%\begin{equation}
%\label{eq:parin}
%\begin{split}
%\sfd_\Y^2(x,y)+\sfd_\Y^2(z,w)\leq &\sfd_\Y^2(x,z)+\sfd_\Y^2(z,y)+\sfd_\Y^2(y,w)+\sfd_\Y^2(w,x)\\
%&\qquad\qquad-\alpha(\sfd_\Y^2(x,w)-\sfd_\Y^2(y,z))-(1-\alpha)(\sfd_\Y^2(x,z)-\sfd_\Y^2(y,w))
%\end{split}
%\end{equation}
%holds for any $\alpha\in[0,1]$, 
and that for any couple of points the geodesic connecting them is unique and continuously depend on the extrema as map from $\Y^2$ to $C([0,1],\Y)$. We shall denote by $\G^{x,y}$ the only geodesic connecting $x$ to $y$. The continuous dependence of $\G^{x,y}$ on $x,y$ grants that for given $u,v\in L^0(\Omega,\Y)$ the map $G^{u,v}_t$ defined as $x\mapsto \G^{u(x),v(x)}_t\in \Y$ also belongs to $L^0(\Omega,\Y)$. It is then easy to see that if $u,v\in  L^2(\Omega,\Y_{\bar y})$, then $\G^{u,v}_t\in L^2(\Omega,\Y_{\bar y})$ and a simple application of the definition shows that this is the only geodesic from $u$ to $v$, indeed recall that on any metric space, for any triple of points $p,q,r$ and $t\in(0,1)$ it holds
\begin{equation}
\label{eq:tint}
\frac{\d^2(p,r)}{t}+\frac{\d^2(r,q)}{1-t}\geq \sfd^2(p,q)\quad\text{ with equality iff $r$ is a $t$-intermediate point between $p$ and $q$.}
\end{equation}
Thus for any $u,v,w\in L^2(\Omega,\Y_{\bar y})$ we have
\[
\begin{split}
\frac{\d_{L^2}^2(u,w)}{t}+\frac{\d_{L^2}^2(w,z)}{1-t}&=\int_\Omega\frac{\d_{\Y}^2(u(x),w(x))}{t}+\frac{\d_{\Y}^2(w(x),z(x))}{1-t}\,\d\mm(x)\\
\text{by \eqref{eq:tint}}\qquad&\geq\int_\Omega\sfd^2_\Y(u(x),v(x))\,\d\mm(x)=\sfd_{L^2}^2(u,v),
\end{split}
\]
so that the equality case in \eqref{eq:tint} shows that $w$ is a $t$-intermediate point between $u$ and $v$ if and only if $w=\G^{u,v}_t$, as claimed. We also recall that if $\Y$ is a $\CAT(0)$ space, then so is   $L^2(\Omega,\Y_{\bar y})$, indeed for any $u,v,z\in L^2(\Omega,\Y_{\bar y})$ and $t\in[0,1]$ we have
\[
\begin{split}
\sfd_{L^2}^2(z,\G^{u,v}_t)&=\int_\Omega \sfd_{\Y}^2(z(x),\G^{u(x),v(x)}_t)\,\d\mm(x)\\
\text{by \eqref{eq:defcat}}\qquad&\leq \int_\Omega(1-t)\sfd_\Y^2(z(x),u(x))+t\sfd_\Y^2(z(x),v(x))-t(1-t)\sfd_\Y^2(u(x),v(x))\,\d\mm(x)\\
&=(1-t)\sfd_{L^2}^2(z,u)+t\sfd_{L^2}^2(z,v)-t(1-t)\sfd_\Y^2(u,v).
\end{split}
\]
The following lemma gathers the key properties of $\ks^{1,2}(\Omega,\Y_{\bar y})$ in the case when $\Y$ is a $\CAT(0)$ space:
\begin{lemma}\label{le:cat}
Let $(\X,\sfd,\mm)$ be a metric measure space, $\Omega\subset\X$ open, $(\Y,\sfd_\Y,\bar y)$ a pointed $\CAT(0)$ space and $u,v\in\ks^{1,2}(\Omega,\Y_{\bar y})$. Put $m:=\G^{u,v}_{1/2}$ and $d:=\sfd_\Y(u,v)$.

Then:
\begin{itemize}
\item[i)] $m\in\ks^{1,2}(\Omega,\Y_{\bar y})$, $d\in \ks^{1,2}(\Omega,\R)$ and
\begin{equation}
\label{eq:enpm}
2\e_2^2[m]+\frac12\e_2^2[d]\leq \e_2^2[u]+\e_2^2[v]\qquad\mm-a.e.\ on\ \Omega.
\end{equation}
\item[ii)] Assume that  $u,v\in\ks_{\bar u}^{1,2}(\Omega,\Y_{\bar y})$ for some $\bar u\in \ks^{1,2}(\Omega,\Y_{\bar y})$. Then $m\in\ks_{\bar u}^{1,2}(\Omega,\Y_{\bar y})$ as well.
\end{itemize}
\end{lemma}
\begin{proof}\ \\
\noindent{\bf (i)} Let $x,y\in\Omega$ and apply inequality \eqref{eq:parin1} to the geodesics $\gamma:=\G^{u(y),v(y)}$ and $\eta:=\G^{u(x),v(x)}$ and for $t:=\frac12$ to obtain
\[
2\sfd^2_\Y(m(y),m(x))+\frac12\big(d(y)-d(x)\big)^2\leq \sfd^2_\Y(u(y),u(x))+\sfd^2_\Y(v(y),v(x))
\]
and thus integrating in $y$ over $B_r(x)$ and dividing by $r^2$ we deduce
\[
2\kse^2_{2,r}[m,\Omega]+\frac12\kse^2_{2,r}[d,\Omega]\leq \kse^2_{2,r}[u,\Omega]+\kse^2_{2,r}[v,\Omega]\qquad on\ \Omega.
\]
Then the fact that $m\in\ks^{1,2}(\Omega,\Y_{\bar y})$ and $d\in \ks^{1,2}(\Omega,\R)$ follow from the very Definition \ref{def:ksloc} while the bound \eqref{eq:enpm}  from  point $(ii)$ in Theorem \ref{thm:ksomega}.

\noindent{\bf (ii)} Notice that 
\[
\sfd_\Y(m,\bar u)\leq \sfd_\Y(m, u)+\sfd_\Y(u,\bar u)\leq  \sfd_\Y(v, u)+\sfd_\Y(u,\bar u)\leq   \sfd_\Y(v, \bar u)+2\sfd_\Y(u,\bar u)  
\]
and that the rightmost side belongs to $W^{1,2}_0(\Omega)$ by assumption. Then the conclusion follows by Proposition \ref{prop:order}.
\end{proof}

The existence of a minimizer for $\E_{2,\bar u}^\Omega$ will follow from the bound \eqref{eq:enpm} and the following version of the Poincar\'e inequality:
\begin{lemma}
Let $(\X,\sfd,\mm)$ be a doubling space supporting a Poincar\'e inequality, $\Omega\subset\X$ open bounded with $\mm(\X\setminus\Omega)>0$. Then there is a constant $C>0$ depending only on the doubling and Poincar\'e constants of $\X$, on ${\rm diam}(\Omega)$ and on $\mm(\{x:\sfd(x,\Omega)\leq 1\})$ such that
\begin{equation}
\label{eq:poin2}
\int_\Omega |f|^2\,\d\mm\leq C\int_\Omega|D f|^2\,\d\mm\qquad\forall f\in W^{1,2}_0(\Omega).
\end{equation}
\end{lemma}
\begin{proof}
Recall that  $W^{1,2}_0(\Omega)$ can be defined as the closure in $W^{1,2}(\X)$ of the space of functions with support in $\Omega$. In particular, functions in $W^{1,2}_0(\Omega)$ are functions in $W^{1,2}(\X)$ which are 0 $\mm$-a.e.\ outside $\Omega$. Fix such function $f$, let $\Omega':=\{x:\sfd(x,\Omega)<1\}$ and ${\sf D}:={\rm diam}(\Omega')\leq{\rm diam}(\Omega)+2$. Then with the same notation of Proposition \ref{prop:whs} we know that $G_{\sf D}:=C({\sf D})M_{2\lambda{\sf D}}(|Df|)$ is   an Hajlasz upper gradient for $f$ at scale ${\sf D}$ and therefore
\[
|f(x)|\leq {\sf D}\big(G_{\sf D}(x)+G_{\sf D}(y)\big)\qquad\mm\times\mm-a.e.\ x,y\in \X^2\ \text{such that}\ x\in\Omega,\ y\in\Omega'\setminus\Omega.
\]
Squaring and integrating we obtain
\[
\mm(\Omega'\setminus\Omega)\int|f|^2\,\d\mm\leq 4\mm(\Omega'){\sf D}^2\int_\X G_{\sf D}^2\,\d\mm\stackrel{\eqref{eq:maxest}}\leq 4\mm(\Omega'){\sf D}^2C({\sf D})\int_\X|Df|^2\,\d\mm,
\]
which is the claim.
\end{proof}
We then have the following result:
\begin{theorem}\label{thm:unicat}
Let $(\X,\sfd,\mm)$ be a strongly rectifiable space with locally uniformly doubling measure and supporting a Poincar\'e inequality (in particular these holds if it is a $\RCD(K,N)$ space for some $K\in\R$ and $N\in[1,\infty)$) and $\Omega\subset\X$ a bounded open set with $\mm(\X\setminus \Omega)>0$. Let $(\Y,\sfd_\Y,\bar y)$ be a pointed $\CAT(0)$ space, $\bar u\in\ks^{1,2}(\Omega,\Y_{\bar y})$.
Then the functional $\E_{2,\bar u}^\Omega:L^2(\Omega,\Y)\to[0,\infty]$:
\begin{itemize}
\item[i)]   is convex and lower semicontinuous,
\item[ii)] admits a unique minimizer.
\end{itemize}
\end{theorem}
\begin{proof}\ \\
\noindent{\bf (i)} We already know that $\E_{2,\bar u}^\Omega:L^2(\Omega,\Y_{\bar y})\to[0,\infty]$ is lower semicontinuous and thus to conclude it is sufficient to show that $\E_{2}^\Omega:L^2(\Omega,\Y_{\bar y})\to[0,\infty]$ is  convex and that geodesics with endpoints in $\ks^{1,2}_{\bar u}(\Omega,\Y_{\bar y})$ lie entirely in $\ks^{1,2}_{\bar u}(\Omega,\Y_{\bar y})$. For the convexity of $\E_{2}^\Omega$  we integrate \eqref{eq:enpm} and disregard the term with $d$ to deduce that
\[
\E_{2}^\Omega(m)\leq\frac12(\E_{2}^\Omega(u)+\E_{2}^\Omega(v)),
\]
which is the convexity inequality for midpoints. Then a standard iteration argument based on dyadic partition and the lower semicontinuity of $\E_{2}^\Omega$ gives the required convexity. The same line of thought shows that to conclude it is sufficient to prove that for $u,v\in\ks^{1,2}_{\bar u}(\Omega,\Y_{\bar y})$ we have $m\in \ks^{1,2}_{\bar u}(\Omega,\Y_{\bar y})$: this is precisely the content of point $(ii)$ in Lemma \ref{le:cat} above.

\noindent{\bf (ii)} It is sufficient to prove that any minimizing sequence is $L^2(\Omega,\Y_{\bar y})$-Cauchy. Thus let $(u_n)\subset \ks^{1,2}_{\bar u}(\Omega,\Y_{\bar y})$ be such sequence and let $I:=\lim_n\E_{2,\bar u}^\Omega(u_n)=\inf  \E_{2,\bar u}^\Omega$. For every $n,m\in\N$ put ${\sf m}_{n,m}:=G^{u_n,u_m}_{\frac12}$, $d_{n,m}:=\sfd_\Y(u_n,u_m)$ and recall that $(ii)$ of  Proposition \ref{prop:dw120} gives $d_{n,m}\in W^{1,2}_0(\Omega)$ and $(ii)$ of Lemma \ref{le:cat} above gives ${\sf m}_{n,m}\in \ks^{1,2}_{\bar u}(\Omega,\Y_{\bar y})$, so that by $(i)$ of the same lemma we get
\[
\begin{split}
\frac1{2c(d)}\int_\Omega |Dd_{n,m}|^2\,\d\mm&\stackrel{\eqref{eq:samenorm2},\eqref{eq:samediff}}=\frac{1}2\int_\Omega \e_2[d_{n,m}]^2\,\d\mm\\
&\stackrel{\phantom{\eqref{eq:samediff}}\atop\eqref{eq:enpm}}\leq \E_{2,\bar u}^\Omega(u_n)+\E_{2,\bar u}^\Omega(u_m)-2\E_{2,\bar u}^\Omega({\sf m}_{n,m})\leq \E_{2,\bar u}^\Omega(u_n)+\E_{2,\bar u}^\Omega(u_m)-2I
\end{split}
\]
and therefore
\[
\lims_{n,m\to\infty}\int_\Omega |Dd_{n,m}|^2\,\d\mm=0.
\]
Hence the Poincar\'e inequality \eqref{eq:poin2} yields $\lims_{n,m\to\infty}\int_\Omega |d_{n,m}|^2\,\d\mm=0$, as desired.
\end{proof}
We conclude pointing out that for target spaces which are $\CAT(0)$ the energy density can be expressed - up to a multiplicative dimensional constant - as Hilbert-Schmidt norm of the differential, very much in line with the smooth case. This is due to the following result, which is (a particular case of) the main theorem in \cite{DMGSP18}:
\begin{theorem}[Universal infinitesimal Hilbertianity of $\CAT(0)$ spaces]\label{thm:uih}
Let $(\Y,\sfd_\Y)$ be a $\CAT(0)$ space and $\mu$ a non-negative and non-zero Borel measure on $\Y$ concentrated on a separable subset and giving finite mass to bounded sets. 

Then $W^{1,2}(\Y,\sfd_\Y,\mu)$ is a Hilbert space.
\end{theorem}
Recall that given a linear map $\ell:\R^d\to H$ with $H$ Hilbert, its Hilbert-Schmidt norm $\|\ell\|_{\sf HS}$  is given by
\[
\|\ell\|_{\sf HS}^2={\rm tr}(\ell^*\ell)=\sum_{i=1}^d|\ell(v_i)|_H^2,
\]
where $v_1,\ldots,v_d$ is any orthonormal base of $\R^d$, the fact that the result does not depend on the base chosen being well known and easy to check.

It is easy to see that, up to a constant, the Hilbert-Schmidt norm coincides with the 2-size:
\begin{lemma}\label{le:normhs}
Let $\ell:\R^d\to H$ be a linear operator, with $H$ being a Hilbert space. Then
\[
\|\ell\|_{\sf HS}=\sqrt{d+2}\,S_2(\ell).
\]
In particular, if $\mathscr H_1,\mathscr H_2$ are Hilbert $L^0(\mm)$-modules with $\mathscr H_1$ of dimension $d$ and $T:\mathscr H_1\to \mathscr H_2$ is $L^0(\mm)$-linear and continuous, then
\[
|T|_{\sf HS}=\sqrt{d+2}\,S_2(T)\qquad\mm-a.e..
\]
\end{lemma}
\begin{proof}
Consider the Lie group $SO(d)$ and, writing its elements in matrix form w.r.t.\ the canonical base of $\R^d$, think of it as subset of $(\R^d)^d$. For $i=1,\ldots,d$ let $\pi^i:(\R^d)^d\to\R^d$  be the canonical projection, let $\mu$ be the normalized Haar measure on $SO(d)$  and notice that by symmetry arguments we have $\pi^i_*\mu=\nu_{S^{d-1}}$ for every $i=1,\ldots,d$, where $\nu_{S^{d-1}}$ is the normalized volume measure on $S^{d-1}=\{(x_1,\ldots,x_d):\sum_i|x_i|^2=1\}\subset\R^d$.

Thus we know that for every $(v_1,\ldots,v_d)\in \supp(\mu)$ it holds $\|\ell\|_{\sf HS}^2=\sum_{i=1}^d|\ell(v_i)|^2_H$ and integrating w.r.t.\ $\mu$ we obtain
\[
\begin{split}
\|\ell\|_{\sf HS}^2=\int\sum_{i=1}^d|\ell(v_i)|^2_H\,\d\mu(v_1,\ldots,v_d)=\sum_{i=1}^d\int|\ell(v)|^2_H\,\d\pi^i_*\mu(v)=d\int|\ell(v)|^2_H\,\d\nu_{S^{d-1}}(v).
\end{split}
\]
On the other hand we have
\[
\begin{split}
S_2(\ell)^2&=\fint_{B_1(0)}|\ell(v)|^2_H\,\d v=d\int_0^1r^{d-1}\int |\ell(rv)|_H^2\,\d\nu_{S^{d-1}}(v)\,\d r=\frac{d}{d+2}\int|\ell(v)|^2_H\,\d\nu_{S^{d-1}}(v)
\end{split}
\]
and the conclusion follows.

The second part of the statement now easily follows from the first by considering a Hilbert base of $\mathscr H_1$ and writing everything in coordinates.
\end{proof}
From this last lemma we obtain the following representation formula for the energy density:
\begin{proposition}[Energy density as Hilbert-Schmidt norm]
Let $(\X,\sfd,\mm)$ be a strongly rectifiable space with uniformly locally doubling measure and supporting a Poincar\'e inequality (in particular these hold if it is a $\RCD(K,N)$ space for some $K\in\R$ and $N\in[1,\infty)$) and $\Omega\subset\X$ an open set. Let $(\Y,\sfd_\Y,\bar y)$ be a pointed $\CAT(0)$ space and $u\in\ks^{1,2}(\Omega,\Y_{\bar y})$.

Then for its energy density $\e_2[u]$ we have the representation formula
\[
\e_2[u]=(d+2)^{-\frac12}|\d u|_{\sf HS}\qquad\mm-a.e.,
\]
where $d$ is the dimension of $\X$.
\end{proposition}
\begin{proof} From Theorem \ref{thm:uih} we deduce that $L^0_\mu(T\Y)$ is a Hilbert module for any  measure $\mu$ as in the statement of the theorem. Then  $ {\rm Ext}\big((u^\ast L^0_\mu(T^\ast\Y))^\ast \big)$ is also a Hilbert module, so that it makes sense to speak about the Hilbert-Schmidt norm of $\d u$. Then the conclusion follows from  Lemma \ref{le:normhs} and of formula \eqref{eq:endenso}.
\end{proof}

\def\cprime{$'$} \def\cprime{$'$}

\end{document}